\documentclass[12pt]{amsart}
\usepackage{amssymb,amsthm,amsmath,amscd}
\usepackage{graphicx}
\newtheorem{thm}{Theorem}[section]
\newtheorem{lem}[thm]{Lemma}
\newtheorem{cor}[thm]{Corollary}
\newtheorem{prop}[thm]{Proposition}

\theoremstyle{remark}
\newtheorem{defn}[thm]{Definition}
\newtheorem{rem}[thm]{Remark}
\newtheorem{exa}[thm]{Example}
\newtheorem{notation}[thm]{Notation}

\newtheorem*{acknowledgement}{Acknowledgment}
\makeatletter
 
 \@addtoreset{equation}{section}
\makeatother

\title{{\bf Fibers of Cyclic Covering Fibrations of a Ruled Surface}}
\author{Makoto Enokizono}
\subjclass[2010]{14D06}
\thanks{
	{\bf Keywords:}
	fibered surface, singular fiber, cyclic covering.
}
\address{Makoto Enokizono,
	Department of Mathematics,
	Graduate School of Science,
	Osaka University,
	Toyonaka, Osaka 560-0043, Japan}
\email{m-enokizono@cr.math.sci.osaka-u.ac.jp}
\usepackage[top=35truemm,bottom=35truemm,left=25truemm,right=25truemm]{geometry}
\usepackage{amsmath,cases}
\usepackage{amsfonts}
\usepackage[all]{xy}
\usepackage{here}
\usepackage{latexsym}
 
\begin{document}
\maketitle
\begin{abstract}
We give an algorithm to classify singular fibers of finite cyclic covering fibrations of a ruled surface by using singularity diagrams. As the first application, we classify all fibers of 3-cyclic covering fibrations of genus 4 of a ruled surface and show that the signature of a complex surface with this fibration is non-positive by computing the local signature for any fiber. As the second application, we classify all fibers of hyperelliptic fibrations of genus 3 into 12 types according to the Horikawa index. We also prove that finite cyclic covering fibrations of a ruled surface have no multiple fibers if the degree of the covering is greater than 3.
\end{abstract}

\section*{Introduction}
Let $f\colon S\rightarrow B$ be a surjective morphism from a complex smooth projective surface $S$ to a smooth projective curve $B$ with connected fibers. 
The datum $(S,f,B)$ or simply $f$ is called a fibered surface or a fibration.
A fibered surface $f$ is said to be {\it relatively minimal} if there exist no $(-1)$-curves contained in a fiber of $f$. 
The genus $g$ of a fibered surface is defined to be the genus of a general fiber of $f$.

In the study of fibered surfaces, one of the central problem is the classification of singular fibers.
Any relatively minimal fibration of genus $0$ is a holomorphic $\mathbb{P}^1$-bundle (hence, no singular fibers). 
As is well known, all fibers of elliptic surfaces were classified by Kodaira in \cite{Ko}. 
As to fibrations of genus $2$, the complete list of singular fibers was obtained by Namikawa and Ueno in \cite{NamiUe}.
On the other hand, Horikawa \cite{Ho} showed that fibers of genus $2$ fibrations fall into $6$ types $(0)$, $({\rm I})$,\ldots,$({\rm V})$ according to the numerical invariant attached to singular fiber germs, which is nowadays called the Horikawa index (cf.\ \cite{ak}). 
When $g=3$, based on Matsumoto-Montesinos' theory, Ashikaga and Ishizaka \cite{AsIs} accomplished the topological classification with a vast list, which is comparable to Namikawa-Ueno's in genus $2$ case.

In \cite{enoki}, we studied primitive cyclic covering fibrations of type $(g,h,n)$.
Roughly speaking, it is a fibered surface of genus $g$ obtained as the relatively minimal model of an $n$ sheeted cyclic branched covering of another fibered surface of genus $h$.
Note that hyperelliptic fibrations are nothing more than such fibrations of type $(g,0,2)$.
Our main concern in \cite{enoki} was the slope of such fibrations of type $(g,h,n)$ and we established the lower bound of the slope for them.
Furthermore, we succeeded in giving even an upper bound.
In this paper, we are interested in singular fibers themselves appearing in primitive cyclic covering fibrations of type $(g,0,n)$, and give the complete lists of fibers when $(g,n)=(4,3), (3,2)$.

In \S1, we recall basic results in \cite{enoki} as the preliminaries.
In \S2, in order to extract detailed information from the singular points of the branch locus, we introduce the notion of {\em singularity diagrams} which is our main tool for studying fibers.
Though it enables us to handle all the possible fibers in theory, it is rather tedious in practice to carry it over for large $n$ and $g$.
In \S3, we consider the case where $n=3$ and $g=4$, and show the following:

\begin{thm}
Fibers of primitive cyclic covering fibrations of type $(4,0,3)$ fall into $32$ classes of types $(0_{i_1,\ldots,i_m})$, $({\rm I\hspace{-.1em}I\hspace{-.1em}I}_{i,j})$, $({\rm I\hspace{-.1em}V}_k)$, $({\rm V}_{i,j})$, $({\rm V\hspace{-.1em}I}_k)$, $({\rm V\hspace{-.1em}I\hspace{-.1em}I}_k)$ and $({\rm V\hspace{-.1em}I\hspace{-.1em}I\hspace{-.1em}I})$ listed in {\rm \S3} plus $9$ classes of types $({\rm I}_{i,j,l})$ and $({\rm I\hspace{-.1em}I}_{k,l})$ up to $(-2)$-curves.
\end{thm}

\begin{cor}
Let $f\colon S\to B$ be a primitive cyclic covering fibration of type $(4,0,3)$.
Then we have
$$
K_f^2=\frac{24}{7}\chi_f+\mathrm{Ind}
$$
and $\mathrm{Ind}$ is given by
\begin{align*}
{\rm Ind}=&\sum_{l\ge 0} \frac{3}{7}(l+1)\nu({\rm I}_{*,*,l})+\sum_{l\ge 0} \frac{3}{7}(l+2)\nu({\rm I\hspace{-.1em}I}_{*,l})+\frac{3}{7}\nu({\rm I\hspace{-.1em}I\hspace{-.1em}I}_{*,*})+\frac{16}{7}\nu({\rm I\hspace{-.1em}V}_{*}) \\
&+\frac{16}{7}\nu({\rm V}_{*,*})+\frac{16}{7}\nu({\rm V\hspace{-.1em}I}_{*,*})+\frac{26}{7}\nu({\rm V\hspace{-.1em}I\hspace{-.1em}I}_{*})+\frac{33}{7}\nu({\rm V\hspace{-.1em}I\hspace{-.1em}I\hspace{-.1em}I}) \\
\end{align*}
where $\nu(*)$ denotes the number of fibers of type $(*)$ and $\nu({\rm I}_{*,*,l}):=\sum_{i,j}\nu({\rm I}_{i,j,l})$, etc.
\end{cor}

Recall that Ueno and Xiao showed independently that the signature of a complex surface with a genus $2$ fibration is non-positive, answering affirmatively to a conjecture by Persson. 
If $S$ is a complex surface with a primitive cyclic covering fibration of type $(g,0,n)$, then, as shown in \cite{enoki}, the signature of $S$ can be expressed as the total sum of the local signature for fibers. 
We can compute the local signature for each type of fibers in the above theorem, and find that it is negative for any singular fiber. 
Therefore, we obtain the following:

\begin{cor}
The signature of a surface with a primitive cyclic covering fibration of type $(4,0,3)$ is not positive.
\end{cor}

In \S4, we turn our attention to hyperelliptic fibrations of genus $3$ (i.e., the case where $n=2$ and $g=3$). We classify all fibers into $12$ types $(0)$, $({\rm I})$,\ldots,$({\rm X\hspace{-.1em}I})$ according to the Horikawa index and show the following:

\begin{thm}
Let $f\colon S\rightarrow B$ be a relatively minimal hyperelliptic fibration of genus $3$. Then
$$
K_f^2=\frac{8}{3}\chi_f+{\rm Ind}
$$
and {\rm Ind} is given by 
\begin{eqnarray*}
{\rm Ind}&=&\sum_{i} \frac{2}{3}i\nu({\rm I}_{i,0,0})+\sum_{i,k\ge 1} \left(\frac{2}{3}i+\frac{5}{3}k-1\right)\nu({\rm I}_{i,0,k})+\sum_{i} \left(\frac{2}{3}i+\frac{5}{3}\right)\nu({\rm I}_{i,0,\infty})\\
&+&\sum_{i,j\ge 1,k\ge 1} \left(\frac{2}{3}i+\frac{5}{3}(j+k)-2\right)\nu({\rm I}_{i,j,k})+\sum_{i,j\ge 1} \left(\frac{2}{3}i+\frac{5}{3}j+\frac{2}{3}\right)\nu({\rm I}_{i,j,\infty})\\
&+&\sum_{i} \left(\frac{2}{3}i+\frac{10}{3}\right)\nu({\rm I}_{i,\infty,\infty})+\sum_{i,j,k} \left(\frac{2}{3}i+\frac{5}{3}(j+k)\right)\nu({\rm I\hspace{-.1em}I}_{i,j,k})\\
&+&\sum_{i,j} \left(\frac{2}{3}i+\frac{5}{3}j+\frac{8}{3}\right)\nu({\rm I\hspace{-.1em}I\hspace{-.1em}I}_{i,j})+\sum_{i,j} \left(\frac{2}{3}i+\frac{5}{3}j+\frac{4}{3}\right)\nu({\rm I\hspace{-.1em}V}_{i,j})\\
&+&\sum_{j} \left(\frac{5}{3}j+\frac{4}{3}\right)\nu({\rm V}_{j})+\sum_{j} \left(\frac{5}{3}j+\frac{5}{3}\right)\nu({\rm V\hspace{-.1em}I}_{j})\\
&+&\frac{4}{3}\nu({\rm V\hspace{-.1em}I\hspace{-.1em}I}_{0})+\sum_{j\ge 1} \left(\frac{5}{3}j+\frac{1}{3}\right)\nu({\rm V\hspace{-.1em}I\hspace{-.1em}I}_{j})+\sum_{j\ge 1} \left(\frac{5}{3}j+\frac{2}{3}\right)\nu({\rm V\hspace{-.1em}I\hspace{-.1em}I\hspace{-.1em}I}_{j})\\
&+&\frac{4}{3}\nu({\rm I\hspace{-.1em}X})+\frac{7}{3}\nu({\rm X})+\frac{10}{3}\nu({\rm X\hspace{-.1em}I}),\\
\end{eqnarray*}
where $\nu(*)$ denotes the number of fibers of type $(*)$.
\end{thm}

\noindent This is comparable to Horikawa's result \cite{Ho} in genus $2$. We remark that Horikawa himself obtained a similar list, but never published.

Multiple fibers are among interesting singular fibers. 
It is known that there exists a hyperelliptic fibration $f$ with a double fiber for any odd $g$ (for example, any fiber of type $({\rm I\hspace{-.1em}I\hspace{-.1em}I})$ listed in \S4 is a double fiber). 
Moreover, we can construct an example of primitive cyclic covering fibrations of type $(g,0,3)$ with a triple fiber for any $g\ge 4$ satisfying $g\equiv 1$ (mod $3$), $g\neq 7$. However, we have the following assertion which imposes an unexpected limitation for the existence of multiple fibers.

\begin{prop}
Let $f\colon S\rightarrow B$ be a primitive cyclic covering fibration of type $(g,0,n)$. If $n\le 3$, then any multiple fiber of $f$ is an $n$-fold fiber. If $n\ge 4$, then $f$ has no multiple fibers.
\end{prop}

\begin{acknowledgement}\normalfont
The author expresses his sincere gratitude to his supervisor Professor Kazuhiro Konno for many suggestions and warm encouragement.
\end{acknowledgement}
\section{Preliminaries}
In this section, we recall and state without proofs basic results obtained in \cite{enoki} in order to fix notation.
\subsection{Definition}
\ 
\smallskip

Let $Y$ be a smooth projective surface and $R$ an effective divisor on $Y$ which is divisible by $n$ in the Picard group $\mathrm{Pic}(Y)$, that is, $R$ is linearly equivalent to $n\mathfrak{d}$ for some divisor $\mathfrak{d}\in \mathrm{Pic}(Y)$. 
Then we can construct a finite $n$-sheeted covering of $Y$ with branch locus $R$ as follows. 
Put
$\mathcal{A}=\bigoplus_{j=0}^{n-1}\mathcal{O}_Y(-j\mathfrak{d})$
and introduce a graded $\mathcal{O}_Y$-algebra structure on $\mathcal{A}$ by multiplying the section of $\mathcal{O}_Y(n\mathfrak{d})$ defining $R$. 
We call $Z:=\mathrm{Spec}_Y(\mathcal{A})$ equipped with the natural surjective morphism $\varphi\colon Z\to Y$ a  {\em classical 
$n$-cyclic covering} of $Y$ branched over $R$, according to \cite{bar}. 
Locally, $Z$ is defined by $z^n=r(x,y)$, where $r(x,y)$ denotes the 
local analytic equation of $R$. 
From this, one sees that $Z$ is normal if and only if $R$ is reduced, and $Z$ is smooth if and only if so is $R$. 
When $Z$ is smooth, we have
\begin{equation}\label{cyceq}
 \varphi^*R=nR_0,\; K_Z=\varphi^*K_Y+(n-1)R_0,\; \mathrm{Aut}(Z/Y)\simeq \mathbb{Z}/n\mathbb{Z}
\end{equation}
where $R_0$ is the effective divisor (usually called the ramification 
divisor) on $Z$ defined locally by $z=0$, and $\mathrm{Aut}(Z/Y)$ is the 
covering transformation group for $\varphi$.

\begin{defn}\normalfont \label{primdef}
A relatively minimal fibration $f\colon S\to B$ of genus $g\geq 2$ is called a 
primitive cyclic covering fibration of type $(g,h,n)$, if there exist a (not 
necessarily relatively minimal) fibration 
$\widetilde{\varphi}\colon \widetilde{W}\to B$ of genus $h\geq 0$, and a 
 classical $n$-cyclic covering 
$$
\widetilde{\theta}\colon \widetilde{S}=
\mathrm{Spec}_{\widetilde{W}}\left(\bigoplus_{j=0}^{n-1}
\mathcal{O}_{\widetilde{W}}(-j\widetilde{\mathfrak{d}})\right)\to \widetilde{W}
$$ 
branched over a smooth curve $\widetilde{R}\in |n\widetilde{\mathfrak{d}}|$
for some $n\geq 2$ and $\widetilde{\mathfrak{d}}\in 
\mathrm{Pic}(\widetilde{W})$ 
such that 
$f$ is the relatively minimal 
model of $\widetilde{f}:=\widetilde{\varphi}\circ \widetilde{\theta}$.
\end{defn}

In this paper, $f\colon S\to B$ denotes a primitive cyclic covering fibration of type $(g,0,n)$ and 
we freely use the notation in Definition \ref{primdef}. 
Let $\widetilde{F}$ and $\widetilde{\Gamma}$ be general fibers of 
$\widetilde{f}$ and $\widetilde{\varphi}$, respectively. 
Then the restriction map 
$\widetilde{\theta}|_{\widetilde{F}}\colon \widetilde{F}\to 
\widetilde{\Gamma}$ is a classical $n$-cyclic covering branched over 
$\widetilde{R}\cap \widetilde{\Gamma}$. 
Since the genera of $\widetilde{F}$ and $\widetilde{\Gamma}$ are $g$ and 
$0$, respectively, the Hurwitz formula gives us
\begin{equation}\label{r}
r:=\widetilde{R}\widetilde{\Gamma}=\frac{2g}{n-1}+1.
\end{equation}
Note that $r$ is a multiple of $n$. 
Let $\widetilde{\sigma}$ be a generator of $\mathrm{Aut}(\widetilde{S}/\widetilde{W})
\simeq \mathbb{Z}/n\mathbb{Z}$ and $\rho\colon \widetilde{S}\to S$ the natural 
birational morphism. 
By assumption, $\mathrm{Fix}(\widetilde{\sigma})$ is a disjoint 
union of smooth curves and $\widetilde{\theta}(\mathrm{Fix}(\widetilde{\sigma}))=\widetilde{R}$. 
Let $\varphi\colon W\to B$ be a relatively minimal model of 
$\widetilde{\varphi}$ and $\widetilde{\psi}\colon \widetilde{W}\to W$ the 
natural birational morphism. 
Since $\widetilde{\psi}$ is a succession of blow-ups, 
we can write $\widetilde{\psi}=\psi_1\circ \cdots \circ \psi_N$, where 
$\psi_i\colon W_i\to W_{i-1}$ denotes the blow-up at $x_i\in W_{i-1}$ 
$(i=1,\dots,N)$ with $W_0=W$ and $W_N=\widetilde{W}$. 
We define reduced curves $R_i$ on $W_i$ inductively as $R_{i-1}=(\psi_i)_*R_i$ starting 
from $R_N=\widetilde{R}$ down to $R_0=:R$. 
We also put $E_i=\psi_i^{-1}(x_i)$ and 
$m_i=\mathrm{mult}_{x_i}(R_{i-1})$ for $i=1,2,\dots, N$.

\begin{lem}\label{multlem}
With the above notation, the following hold for any $i=1,\dots, N$.

\smallskip

\noindent
$(1)$ Either $m_i\in n\mathbb{Z}$ or $m_i\in n\mathbb{Z}+1$.
Moreover, $m_i\in n\mathbb{Z}$ holds if and only if $E_i$ is not contained in $R_i$.

\smallskip

\noindent
$(2)$ 
 $R_i=\psi_i^*R_{i-1}-n\displaystyle{\left[\frac{m_i}{n}\right]}E_i$, 
 where $[t]$ is the greatest integer not exceeding $t$.

\smallskip

\noindent
$(3)$ There exists $\mathfrak{d}_i\in \mathrm{Pic}(W_i)$ such that 
$\mathfrak{d}_i=\psi_i^*\mathfrak{d}_{i-1}$ and $R_i\sim n\mathfrak{d}_i$, $\mathfrak{d}_N=\widetilde{\mathfrak{d}}$.
\end{lem}

Let $E$ be a $(-1)$-curve on a fiber of $\widetilde{f}$. 
If $E$ is not contained in ${\rm Fix}(\widetilde{\sigma})$, then one can see easily that $L:=\widetilde{\theta}(E)$ is a $(-1)$-curve and $\widetilde{\theta}^{\ast}L$ is a sum of $n$ disjoint $(-1)$-curves containing $E$. 
Contracting them and $L$, we may assume that any $(-1)$-curve on a fiber of $\widetilde{f}$ is contained in ${\rm Fix}(\widetilde{\sigma})$. 
Then, it follows that $\widetilde{\sigma}$ induces an automorphism $\sigma$ of $S$ over $B$ and $\rho$ is the blow-up of all isolated fixed points of $\sigma$ (cf.\ \cite{enoki}). 
One sees easily that there is a one-to-one correspondence between $(-k)$-curves contained in ${\rm Fix}(\widetilde{\sigma})$ and $(-kn)$-curves contained in $\widetilde{R}$ via $\widetilde{\theta}$. 
Hence, the number of blow-ups in $\rho$ is that of vertical $(-n)$-curves contained in $\widetilde{R}$. 
Since $\widetilde{\varphi}\colon \widetilde{W}\to B$ is a ruled surface, a relatively minimal model of it is not unique. By performing elementary transformations, we can choose a standard one:

\begin{lem} \label{eltrlem}
There exists a relatively minimal model $\varphi\colon W\rightarrow B$ of $\widetilde{\varphi}$ such that
if $n=2$ and $g$ is even, then
$$
{\rm mult}_x(R)\le \frac{r}{2}=g+1
$$
for all $x\in R$, and otherwise,
$$
{\rm mult}_x(R_h)\le \frac{r}{2}=\frac{g}{n-1}+1
$$
for all $x\in R_h$, where $R_h$ denotes the horizontal part of $R$, that is, the sum of all $\varphi$-horizontal components of $R$.
\end{lem}

In the sequel, we will tacitly assume that our relatively minimal model $\varphi\colon W\rightarrow B$ of $\widetilde{\varphi}$ enjoys the property of 
the lemma.

\subsection{Slope equality, singularity indices and local signature}
\ 
\smallskip

Let $f\colon S\to B$ be a primitive cyclic covering fibration of type $(g,0,n)$.

\begin{defn}[Singularity index $\alpha$]\label{sinddef}
Let $k$ be a positive integer.
For $p\in B$, we consider all the singular points (including infinitely near 
 ones) of $R$ on the fiber $\Gamma_p$ of $\varphi\colon W\to B$ over $p$.
We let $\alpha_k(F_p)$ be the number of singular points of multiplicity 
 either $kn$ or $kn+1$ among them, and call it the {\em $k$-th singularity 
 index} of $F_p$, the fiber of $f\colon S\to B$ over $p$.
Clearly, we have $\alpha_k(F_p)=0$ except for a finite number of $p\in B$.
We put $\alpha_k=\sum_{p\in B}\alpha_k(F_p)$ and call it the {\em $k$-th 
 singularity index} of $f$.

We also define $0$-th singularity index $\alpha_0(F_p)$ as follows.
Let $D_1$ be the sum of all $\widetilde{\varphi}$-vertical $(-n)$-curves 
 contained in $\widetilde{R}$ and put $\widetilde{R}_0=\widetilde{R}-D_1$.
Then, $\alpha_0(F_p)$ is the ramification index of 
 $\widetilde{\varphi}|_{\widetilde{R}_0}\colon \widetilde{R}_0\to B$ over $p$, 
 that is, the ramification index of 
 $\widetilde{\varphi}|_{(\widetilde{R}_0)_h}\colon (\widetilde{R}_0)_h\to B$ 
 over $p$ minus the sum of the topological Euler number of irreducible 
 components of $(\widetilde{R}_0)_v$ over $p$.
Then $\alpha_0(F_p)=0$ except for a finite number of $p\in B$, and we have
$$
\sum_{p\in B}\alpha_0(F_p)=(K_{\widetilde{\varphi}}+\widetilde{R}_0)\widetilde{R}_0
$$
by definition.
We put $\alpha_0=\sum_{p\in B}\alpha_0(F_p)$ and call it the {\em $0$-th singularity index} of $f$.

Let $\varepsilon(F_p)$ be the number of $(-1)$-curves contained in $F_p$, 
and put $\varepsilon=\sum_{p\in B}\varepsilon(F_p)$.
This is no more than the number of blow-ups appearing in 
$\rho\colon \widetilde{S}\to S$.
\end{defn}

Let $\mathcal{A}_{g,0,n}$ be the set of all fiber germs of primitive 
cyclic covering fibrations of type $(g,0,n)$.
Then the singularity indices $\alpha_k$, $\varepsilon$ can be regarded as $\mathbb{Z}$-valued functions on $\mathcal{A}_{g,0,n}$ naturally.
Recall that the following slope equality holds, which is a generalization of the hyperelliptic case (cf. \cite{pi1}):

\begin{thm}\label{slopeeq}
Let $f\colon S\to B$ be a primitive cyclic covering fibration of type $(g,0,n)$.
Then
$$
K_f^2=\frac{24(g-1)(n-1)}{2(2n-1)(g-1)+n(n+1)}\chi_f+\sum_{p\in B}\mathrm{Ind}(F_p),
$$
where $\mathrm{Ind}\colon \mathcal{A}_{g,0,n}\rightarrow \mathbb{Q}_{\geq 0}$ is defined by
\begin{eqnarray*}
\mathrm{Ind}(F_p)=n\sum_{k\ge 1}\left(\frac{(n+1)(n-1)(r-nk)k}{(2n-1)r-3n}-1\right)\alpha_k(F_p)+
\varepsilon(F_p),
\end{eqnarray*}
which is called the Horikawa index of $F_p$.
\end{thm}

Now, we state a topological application of the slope equality.
For an oriented compact real $4$-dimensional manifold $X$, 
the signature ${\rm Sign}(X)$ is defined to be the number of positive eigenvalues minus 
the number of negative eigenvalues of the intersection form on $H^2(X)$. 
Using the singularity indices, we observe the local concentration of ${\rm Sign}(S)$ on a finite number of fiber germs.

\begin{cor} \label{signcor}
Let $f\colon S\to B$ be a primitive cyclic covering fibration of type $(g,0,n)$. Then,
\begin{equation*}
{\rm Sign}(S)=\sum_{p\in B}\sigma(F_p),
\end{equation*}
where $\sigma\colon \mathcal{A}_{g,0,n}\rightarrow \mathbb{Q}$ is defined by
\begin{eqnarray*}
\sigma(F_p)&=&\frac{-(n-1)(n+1)r}{3n(r-1)}\alpha_0(F_p)+\sum_{k\ge 1}\left(\frac{(n-1)(n+1)(-nk^2+rk)}{3(r-1)}-n\right)\alpha_k(F_p)\\
&&+\frac{1}{3n(r-1)}((n+2)(2n-1)r-3n)\varepsilon(F_p),
\end{eqnarray*}
which is called the local signature of $F_p$.
\end{cor}

\section{Singularity diagrams}
Let $f\colon S\to B$ be a primitive cyclic covering fibration of type $(g,0,n)$ and we freely use the notations in the previous section.
Let $C$ stand for a fiber $\Gamma$ of $\varphi$ or an exceptional curve $E_i$ of $\psi_i$ for some $i$. In the latter case, for the time being, we drop the index and set $R=R_i$ for simplicity. 
Let $R^{\prime}$ be the closure of $R\setminus C$, that is, $R^{\prime}=R-C$ when $C$ is contained in $R$, or $R^{\prime}=R$ when $C$ is not contained in $R$. 
Put $C\cap R^{\prime}=\{x_1,x_2,\dots, x_l\}$. 
We consider a local analytic branch $D$ of $R^{\prime}$ around $x_i$ which has multiplicity $m\ge 2$ at $x_i$ (i.e. $D$ has a cusp $x_i$). Then we have one of the following:

\smallskip

\noindent
(i) $D$ is not tangent to $C$ at $x_i$. If we blow $x_i$ up, then the proper transform of $D$ does not meet that of $C$. Hence, we have $(DC)_{x_i}=m$, where $(DC)_{x_i}$ denotes the local intersection number of $D$ and $C$ at $x_i$.

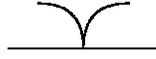
\setlength\unitlength{0.2cm}
\begin{figure}[H]
\begin{center}
 \begin{picture}(5,4)
 \put(-5,0){\line(1,0){10}}
 \qbezier(0,0)(0,3)(3,3)
 \qbezier(0,0)(0,3)(-3,3)
 \end{picture}
\end{center}
\end{figure}

\smallskip

\noindent
(ii) $D$ is tangent to $C$ at $x_i$. If we blow $x_i$ up, then one of the following three cases occurs.

\smallskip

(ii.1) The proper transform of $D$ is tangent to neither that of $C$ nor the exceptional curve.

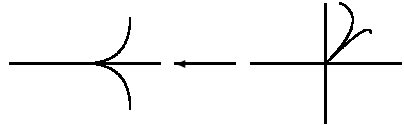
\setlength\unitlength{0.2cm}
\begin{figure}[H]
\begin{center}
 \begin{picture}(13,4)
 \put(-13,0){\line(1,0){10}}
 \qbezier(-8,0)(-5,0)(-5,3)
 \qbezier(-8,0)(-5,0)(-5,-3)
 \put(3,0){\line(1,0){10}}
 \qbezier(8,0)(11,3)(11,2)
 \qbezier(8,0)(11,3)(9,4)
 \put(8,-4){\line(0,1){8}}
 \put(2,0){\vector(-1,0){4}}
 \end{picture}
\end{center}
\end{figure}

\bigskip

(ii.2) The proper transform of $D$ is tangent to the exceptional curve.
Then, the multiplicity $m^{\prime}$ of the proper transform of $D$ at the singular point is less than $m$ and we have $(DC)_{x_i}=m+m^{\prime}$.

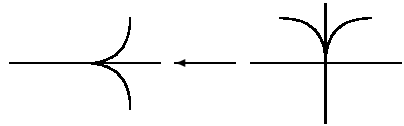
\setlength\unitlength{0.2cm}
\begin{figure}[H]
\begin{center}
 \begin{picture}(13,4)
 \put(-13,0){\line(1,0){10}}
 \qbezier(-8,0)(-5,0)(-5,3)
 \qbezier(-8,0)(-5,0)(-5,-3)
 \put(3,0){\line(1,0){10}}
 \qbezier(8,0)(8,3)(11,3)
 \qbezier(8,0)(8,3)(5,3)
 \put(8,-4){\line(0,1){8}}
 \put(2,0){\vector(-1,0){4}}
 \end{picture}
\end{center}
\end{figure}

\bigskip

(ii.3) The proper transform of $D$ is still tangent to that of $C$.

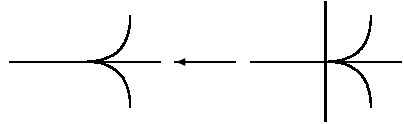
\setlength\unitlength{0.2cm}
\begin{figure}[H]
\begin{center}
 \begin{picture}(13,4)
 \put(-13,0){\line(1,0){10}}
 \qbezier(-8,0)(-5,0)(-5,3)
 \qbezier(-8,0)(-5,0)(-5,-3)
 \put(3,0){\line(1,0){10}}
 \qbezier(8,0)(11,0)(11,3)
 \qbezier(8,0)(11,0)(11,-3)
 \put(8,-4){\line(0,1){8}}
 \put(2,0){\vector(-1,0){4}}
 \end{picture}
\end{center}
\end{figure}

\bigskip

\noindent
We perform blowing-ups at $x_i$ and points infinitely near to it.
Then the case (ii.3) may occur repeatedly, but at most a finite number of times.
Suppose that the proper transform of $D$ becomes not tangent to that of $C$ just after $k$-th blow-up.
If the proper transform of $D$ is as in (ii.1) after $k$-th blow-up (or $D$ is as in (i) when $k=0$), then we have $(DC)_{x_i}=(k+1)m$.
If the proper transform of $D$ is as in (ii.2) after $k$-th blow-up, then we have $(DC)_{x_i}=km+m'$.
In either case, it is convenient to consider as if $D$ consists of $m$ local branches $D_1,\dots, D_m$ smooth at $x_i$ 
and such that $(D_jC)_{x_i}=k+1$ for $j=1,\dots, m$ in the former case and 
$$
(D_jC)_{x_i}=\left\{
\begin{array}{cl}
k, & \text{for }j=1,\dots, m-m',\\
k+1, & \text{for }j=m-m'+1,\dots, m
\end{array}
\right.
$$
in the latter case. 
We call $D_j$ a {\em virtual local branch} of $D$.

\begin{notation}\label{siknotation}
For a positive integer $k$, we let $s_{i,k}$ be the number of such virtual local branches $D_{\bullet}$ satisfying 
$(D_\bullet C)_{x_i}=k$, among those of all local analytic branches of $R-C$ around $x_i$.
Here, when $\mathrm{mult}_{x_i}(D)=1$, we regard $D$ itself as a virtual local branch.
We let $i_{\max}$ be the biggest integer $k$ satisfying $s_{i,k}\neq 0$.
By the definition of $s_{i,k}$, we have 
$$
(R^{\prime}C)_{x_i}=\sum_{k=1}^{i_{\rm max}} ks_{i,k}.
$$
\end{notation}

We put $x_{i,1}=x_i$ and $m_{i,1}=m_i$. 
If $m_{i,1}>1$, we define $\psi_{i,1}\colon W_{i,1}\rightarrow W$ to be the blow-up at $x_{i,1}$ and put $E_{i,1}=\psi_{i,1}^{-1}(x_{i,1})$ and $R_{i,1}=\psi_{i,1}^\ast R-n[m_{i,1}/n]E_{i,1}$. 
Inductively, we define $x_{i,j}$, $m_{i,j}$ to be the intersection point of the proper transform of $C$ and $E_{i,j-1}$, the multiplicity of $R_{i,j-1}$ at $x_{i,j}$, 
and if $m_{i,j}>1$, we define $\psi_{i,j}\colon W_{i,j}\to W_{i,j-1}$, $E_{i,j}$ and $R_{i,j}$ to be the blow-up at $x_{i,j}$, the exceptional curve for $\psi_{i,j}$ and $R_{i,j}=\psi_{i,j}^\ast R_{i,j-1}-n[m_{i,j}/n]E_{i,j}$, respectively. 
Let
$$
i_{\mathrm{bm}}=\max\{j\mid m_{i,j}>1\}, 
$$
be the number of blowing-ups occuring over $x_i$.
We may assume that $i_{\mathrm{bm}}\geq (i+1)_{\mathrm{bm}}$ for 
$i=1,\dots,l-1$ after rearranging the index if necessary.

Then the following two lemmas hold.

\begin{lem} \label{mslem}
If $C$ is contained in $R$, then the following hold.

\smallskip

\noindent
$(1)$ If $n\ge 3$, then $i_{\rm bm}= i_{\rm max}$ for all $i$. 
If $n=2$, then $i_{\rm bm}= i_{\rm max}$ $($resp.\ $i_{\rm bm}= i_{\rm max}+1)$ if and only if $m_{i,i_{\rm max}}\in 2\mathbb{Z}$ $($resp.\ $m_{i,i_{\rm max}}\in 2\mathbb{Z}+1)$.

\smallskip

\noindent
$(2)$ $m_{i,1}=\sum_{k=1}^{i_{\rm max}} s_{i,k}+1$ and $m_{i,i_{\rm bm}}\in n\mathbb{Z}$.

\smallskip

\noindent
$(3)$ $m_{i,j}\in n\mathbb{Z}$ $($resp. $n\mathbb{Z}+1)$ if and only if $m_{i,j+1}=\sum_{k=j+1}^{i_{\rm max}} s_{i,k}+1$ $($resp. $\sum_{k=j+1}^{i_{\rm max}} s_{i,k}+2)$.
\end{lem}

\begin{proof}
(1) is clear from the definitions of $i_{\rm max}$ and $i_{\rm bm}$. 
Since $m_{i,j}$ is the number of virtual local branches of $R_{i,j-1}$ through $x_{i,j}$ and Lemma \ref{multlem}, we have $m_{i,1}=\sum_{k=1}^{i_{\rm max}} s_{i,k}+1$ and $(3)$. 
In (2), ``$+1$'' is the contribution of $C$.
If $m_{i,i_{\mathrm{max}}}\in n\mathbb{Z}+1$, then $x_{i,i_{\max+1}}$ is a double point, 
which is not allowed when $n\geq 3$ by Lemma~\ref{multlem}.
Thus we have shown (2). 
\end{proof}

\begin{lem} \label{mslem2}
If $C$ is not contained in $R$, then the following hold.

\smallskip

\noindent
$(1)$ $i_{\rm bm}\le i_{\rm max}$. If $i_{\rm bm}< i_{\rm max}$, then $m_{i,i_{\rm bm}}\in n\mathbb{Z}$, and $s_{i,k}=0$ for $i_{\rm bm}<k<i_{\rm max}$, and $s_{i,i_{\rm max}}=1$.

\smallskip

\noindent
$(2)$ $m_{i,1}=\sum_{k=1}^{i_{\rm max}} s_{i,k}$.

\smallskip

\noindent
$(3)$ $m_{i,j}\in n\mathbb{Z}$ $($resp. $n\mathbb{Z}+1)$ if and only if $m_{i,j+1}=\sum_{k=j+1}^{i_{\rm max}} s_{i,k}$ $($resp. $\sum_{k=j+1}^{i_{\rm max}} s_{i,k}+1)$.
\end{lem}

\begin{proof}
By the same argument as in the proof of Lemma \ref{mslem}, we have (2), (3). Suppose that $i_{\rm bm}< i_{\rm max}$. (1) follows from the definition of $i_{\rm bm}$ and (3).
\end{proof}

Put $t=R^{\prime}C$ and $c=\sum_{i=1}^l i_{\rm bm}$. 
If $C$ is a fiber $\Gamma$ of $\varphi$, then $t$ is the number of branch points $r$. 
If $C$ is an exceptional curve, then $t$ is the multiplicity of $R$ at the point to which $C$ is contracted. 
Clearly, $c$ is the number of blow-ups occuring over $C$ and $t=\sum_{i=1}^l \sum_{k=1}^{i_{\rm max}} ks_{i,k}$. 
When $C$ is contained in $R$ (resp. not contained in $R$), let $c_i$ be the number of $j=1,\ldots,i_{\rm bm}$ satisfying $m_{i,j}=\sum_{k=j+1}^{i_{\rm max}} s_{i,k}+2$ (resp.\ $m_{i,j}=\sum_{k=j+1}^{i_{\rm max}} s_{i,k}+1$). From Lemmas \ref{mslem} and \ref{mslem2}, $c_i$ can be regarded as the number of $j=1,\cdots,i_{\rm bm}-1$ such that $m_{i,j}\in n\mathbb{Z}+1$. Set $d_{i,j}=[m_{i,j}/n]$.

\begin{prop} \label{tcprop}
If $C$ is contained in $R$, the following equalities hold:
\begin{eqnarray}
t+c+\sum_{i=1}^l c_i&=&\sum_{i=1}^l \sum_{j=1}^{i_{\rm bm}} m_{i,j},\\
\frac{t+c}{n}&=&\sum_{i=1}^l \sum_{j=1}^{i_{\rm bm}} d_{i,j}.  \label{type1}
\end{eqnarray}
\end{prop}

\begin{proof}
From Lemma \ref{mslem}, we get $\sum_{k=1}^{i_{\rm max}} ks_{i,k}=\sum_{j=1}^{i_{\rm bm}} m_{i,j}-i_{\rm bm}-c_i$, which gives us the desired two equalities.
\end{proof}

\begin{prop} \label{tcprop2}
If $C$ is not contained in $R$, the following equalities hold:
\begin{eqnarray}
t+\sum_{i=1}^l c_i&=&\sum_{i=1}^l \sum_{j=1}^{i_{\rm bm}} m_{i,j}+\sum_{i=1}^l (i_{\rm max}-i_{\rm bm}),\\
\frac{t}{n}&=&\sum_{i=1}^l \sum_{j=1}^{i_{\rm bm}} d_{i,j}+\frac{1}{n}\sum_{i=1}^l (i_{\rm max}-i_{\rm bm}+m_{i,i_{\rm bm}}-nd_{i,i_{\rm bm}}).  \label{type0}
\end{eqnarray}
\end{prop}

\begin{proof}
From Lemma \ref{mslem2}, we get $\sum_{k=1}^{i_{\rm max}} ks_{i,k}=\sum_{j=1}^{i_{\rm bm}}m_{i,j}+i_{\rm max}-i_{\rm bm}-c_i$, which gives us the desired two equalities.
\end{proof}

We collect some properties of $m_{i,j}$.

\begin{lem} \label{mijlem}
The following hold:

\smallskip

\noindent
$(1)$ If $n\ge 3$, then $m_{i,j}\ge m_{i,j+1}$. 
If $n=2$, then $m_{i,j}+1\ge m_{i,j+1}$ with equality holding only if $m_{i,j-1}$ is even $($when $j>1)$ and $m_{i,j}$ is odd.

\smallskip

\noindent
$(2)$ If $m_{i,j-1}\in n\mathbb{Z}+1$ and $m_{i,j}\in n\mathbb{Z}$, then $m_{i,j} > m_{i,j+1}$.
\end{lem}

\begin{proof}
If $m_{i,j} < m_{i,j+1}$, then $s_{i,j}=0$ and $m_{i,j}+1=m_{i,j+1}$, since $m_{i,j}-m_{i,j+1}=s_{i,j}-1, s_{i,j}$, or $s_{i,j}+1$. 
Moreover, we have $m_{i,j}\in n\mathbb{Z}+1$ from Lemmas \ref{mslem} and \ref{mslem2}, and then $m_{i,j+1}\in n\mathbb{Z}+2$,
which is impossible when $n\ge 3$ from Lemma~\ref{multlem}.
If $n=2$ and $m_{i,j}+1=m_{i,j+1}$, then $m_{i,j-1}$ is even since $s_{i,j}=0$.

If $m_{i,j-1}\in n\mathbb{Z}+1$ and $m_{i,j}\in n\mathbb{Z}$, then $m_{i,j}-m_{i,j+1}=s_{i,j}+1 > 0$ by Lemmas \ref{mslem} and \ref{mslem2}. Hence $(2)$ follows.
\end{proof}

In particular, we have non-trivial $\alpha_k$ only when $k\leq [r/2n]$ by Lemma~\ref{eltrlem}.

\begin{defn}\normalfont
By using the datum $\{m_{i,j}\}$, one can construct a diagram with entries 
$(x_{i,j},m_{i,j})$ as in Table \ref{mij}. 
We call it the {\it singularity diagram} of $C$.

\begin{table}[H] 
\begin{center}
\caption{singularity diagram}
 \begin{tabular}{cccc} \label{mij}
  $\#$& & & \\ \cline{1-1}
 \multicolumn{1}{|c|}{$(x_{1,1_{\rm bm}},m_{1,1_{\rm bm}})$}& & & \\ \cline{1-1}
 \multicolumn{1}{|c|}{}& & & $\#$\\ \cline{4-4}                     
 \multicolumn{1}{|c|}{$\cdots$}& & & \multicolumn{1}{|c|}{$(x_{l,l_{\rm bm}},m_{l,l_{\rm bm}})$}\\ \cline{4-4}
 \multicolumn{1}{|c|}{}&$\cdots$ & & \multicolumn{1}{|c|}{$\cdots$}\\ \cline{1-4}
 \multicolumn{1}{|c|}{$(x_{1,1},m_{1,1})$}&$\cdots$ & & \multicolumn{1}{|c|}{$(x_{l,1},m_{l,1})$}\\ \cline{1-4}
 \end{tabular}
\end{center}
\end{table}

\noindent 
On the top of the $i$-th column (indicated by $\#$ in Table~\ref{mij}), we write $(i_{\rm max}-i_{\rm bm})$ when $i_{\rm bm}< i_{\rm max}$ and leave it blank when $i_{\rm bm}= i_{\rm max}$. 
We say that the singularity diagram of $C$ is {\em of type} $0$ (resp. {\em of type} $1$) if $C\not \subset R$ (resp. $C\subset R$).
\end{defn}

\begin{defn}\normalfont
Let $\widetilde{\psi}=\psi_1\circ \cdots \circ \psi_N\colon \widetilde{W}\to W$ be a decomposition of the natural birational morphism into a series of blow-ups. We may assume that $\psi_1$,\ldots, $\psi_{N_p}$ are all blow-ups at  points over $p\in B$ for simplicity.
Let $C_1$ be the fiber $\Gamma_p$ of $\varphi$ over $p$ and $C_k$ the exceptional curve for $\psi_{k-1}$ for $k=2,\ldots, N_p+1$. 
Let $\mathcal{D}_k$ be the singularity diagram of $C_k$. 
We call $\mathcal{D}_1, \mathcal{D}_2,\ldots, \mathcal{D}_{N_p+1}$ a {\em sequence of singularity diagrams associated with} $\Gamma_p$.

Put $t^k:=R'C_k$, $l^k:=\#(R'\cap C_k)$ and  let $(x_{i,j}^k,m_{i,j}^k)$, $i=1,\dots,l^k$, $j=1,\dots,i_{\mathrm{bm}}$ denote entries of $\mathcal{D}_k$.
For any $1\leq k \leq N_p+1$ and $1\leq i\leq l^k$ for which $i_{\mathrm{bm}}>0$, there exists a uniquely determined index $k'$ such that 
$C_{k'}$ is the exceptional curve for 
the blow-up at $x_{i,1}^k$.
That is, if we put $I:=\{(k,i)\mid 1\leq i\leq N_p+1,\; 1\leq i\leq l^k \mbox{ and }
i_{\mathrm{bm}}>0\}$ then we get a well-defined map $\Phi: I\to \{2,\dots, N_p+1\}$ that sends $(k,i)$ to $k'$.
It is clear that $\Phi$ is bijective and $k<k'$.
When $\nu=\Phi(\mu,i)$ for some $(\mu,i)\in I$, we say that 
 $\mathcal{D}_{\mu}$ forks to $\mathcal{D}_{\nu}$ and write  
$\mathcal{D}_\mu \overset{i}{\rightsquigarrow} \mathcal{D}_{\nu}$.
\end{defn}

From the definition of sequences of singularity diagrams, we clearly have the following:

\begin{lem} \label{connlem}
With the above notation, let $\mathcal{D}_\mu \overset{i}{\rightsquigarrow} \mathcal{D}_{\nu}$. 
Then, $t^\nu=m_{i,1}^{\mu}$ and $\mathcal{D}_\nu$ is of type $k$ if $t^\nu\equiv k$ $(\text{mod}\  n)$.
Moreover, the following hold.

\smallskip

\noindent
$(1)$ For every $1\le \mu'\le \mu$, $i'$, $j'$ satisfying 
$(x_{i',j'}^{\mu'},m_{i',j'}^{\mu'})=(x_{i,1}^{\mu},m_{i,1}^{\mu})$, one 
 of the following holds:

\begin{itemize}

\item[$({\rm a})$] If $j'<i'_{\mathrm{bm}}$, then $\mathcal{D}_\nu$ has $(x_{i',j'+1}^{\mu'},m_{i',j'+1}^{\mu'})$ as an entry in the bottom row.

\smallskip

\item[$({\rm b})$] If $j'=i'_{\mathrm{bm}}$ and $\mathcal{D}_{\mu'}$ is of type $1$, then $\mathcal{D}_\nu$ has $(1)$ as an entry in the bottom row.

\smallskip

\item[$({\rm c})$] If $j'=i'_{\mathrm{bm}}$, $\mathcal{D}_{\mu'}$ is of type $0$ and $i'_{\mathrm{max}}-i'_{\mathrm{bm}}\ge 2$, then $\mathcal{D}_\nu$ has $(1)$ as an entry in the bottom row.

\smallskip

\item[$({\rm d})$] If $j'=i'_{\mathrm{bm}}$, $\mathcal{D}_{\mu'}$ is of type $0$ and $i'_{\mathrm{max}}-i'_{\mathrm{bm}}=1$, then $\mathcal{D}_\nu$ has $(s)$ for some $s\ge 1$ as an entry in the bottom row.

\end{itemize}

\smallskip

\noindent
$(2)$ In $(1)$, distinct $\mathcal{D}_{\mu'}$'s produce distinct bottom 
 entries of $\mathcal{D}_\nu$.

\smallskip

\noindent
$(3)$ In $(1)$, 
$\mathcal{D}_{\mu'}$ and $\mathcal{D}_\nu$ have no entries in common except 
 the entry appeared in the case $({\rm a})$.
\end{lem}

\begin{defn}
Fix  $n\ge 2$ and $t\in n\mathbb{Z}_{>0}\cup (n\mathbb{Z}_{>0}+1)$.
Suppose that we are given the following datum.

\smallskip

\noindent
(i) Non-negative integers $c$, $l$ and $i_{\mathrm{bm}}$ for $i=1,\dots,l$ satisfying $l>0$, $c=\sum_{i=1}^l i_{\mathrm{bm}}$ and $i_{\mathrm{bm}}\ge (i+1)_{\mathrm{bm}}$ for $i=1,\dots,l-1$.

\smallskip

\noindent
(ii) A non-empty set $\mathcal{S}$ and $c$ pairs $(x_{i,j},m_{i,j})\in \mathcal{S}\times (n\mathbb{Z}_{>0}\cup (n\mathbb{Z}_{>0}+1))$ for $i=1,\ldots,l$, $j=1,\ldots,i_{\mathrm{bm}}$ such that $x_{i,j}\neq x_{i',j'}$ if $(i,j)\neq (i',j')$, and $\{m_{i,j}\}$ satisfies Lemma~\ref{mijlem}.
Moreover, one of the following holds.

\smallskip

(ii,0) There are integers $i_{\mathrm{max}}$ satisfying 
$i_{\mathrm{max}}\ge i_{\mathrm{bm}}$ for $i=1,\dots,l$ such that $i_{\mathrm{max}}$ and $d_{i,j}:=\displaystyle{\left[\frac{m_{i,j}}{n}\right]}$ 
satisfy \eqref{type0}.

\smallskip

(ii,1) $m_{i,i_{\mathrm{bm}}}\in n\mathbb{Z}$ and \eqref{type1} holds.

\smallskip

\noindent
Then we can construct a diagram $\mathcal{D}$ as in Table~\ref{mij}.
We call it an {\em abstract singularity diagram} for $(n,t)$.
For $k=0,1$, the diagram $\mathcal{D}$ is said to be {\em of type $k$} when  (ii,k) holds.

\end{defn}

\begin{defn}
A sequence $\mathcal{D}_1,\mathcal{D}_2,\dots,\mathcal{D}_N$ of abstract singularity diagrams is said to be {\em admissible} if there exists a bijection $\Phi\colon I\to \{2,\dots,N\}$ such that $\mu$ and $\nu:=\Phi(\mu,i)$ satisfy $\mu<\Phi(\mu,i)$ and $(1)$, $(2)$, $(3)$ of Lemma~\ref{connlem} for any $(\mu,i)\in I$, 
where $I:=\{(\mu,i)|1\le \mu\le N,\; 1\le i\le l^\mu\; \text{and}\; i_{\mathrm{bm}}>0\}$.

\end{defn}

\begin{defn}
Two abstract singularity diagrams $\mathcal{D}$ and $\mathcal{D}'$ are {\em equivalent} if $\mathcal{D}$ and $\mathcal{D}'$ are the same up to elements $x_{i,j}\in \mathcal{S}$ and a replacement of columns with the same height.
Two admissible sequences $\mathcal{D}_1$,\dots,$\mathcal{D}_N$ and $\mathcal{D}'_1$,\dots,$\mathcal{D}'_N$ of abstract singularity diagrams are {\em equivalent}, if there exists a bijection $\Psi\colon \{1,\dots,N\}\to \{1,\dots,N\}$ with $\Psi(1)=1$ such that $\mathcal{D}_k$ is equivalent to $\mathcal{D}_{\Psi(k)}$ for any $1\le k\le N$ and $\mathcal{D}_\mu \overset{i}{\rightsquigarrow} \mathcal{D}_\nu$ if and only if $\mathcal{D}_{\Psi(\mu)} \overset{i}{\rightsquigarrow} \mathcal{D}_{\Psi(\nu)}$ for any $\mu<\nu$ and $1\le i\le l^\mu$ (after a suitable replacement of columns of $D_\mu$ with the same height).

\end{defn}
It is clear that any singularity diagram is an abstract singularity diagram,
and any sequence $\mathcal{D}_1$,\dots,$\mathcal{D}_N$ of singularity diagrams associated with $\Gamma_p$ is an admissible sequence of abstract singularity diagrams with $t^1=r$.
We are able to classify all fibers of primitive cyclic covering fibrations of type $(g,0,n)$ by classifying equivalent classes of admissible sequences of abstract singularity diagrams with $t^1=r$.
Indeed, any fiber $F_p$ of a primitive cyclic covering fibration $f$ of type $(g,0,n)$ can be reconstructed by a sequence of singularity diagrams associated with $\Gamma_p$ via the $n$-cyclic covering $\widetilde{\theta}\colon \widetilde{S}\to \widetilde{W}$.

\section{Fibers of primitive cyclic covering fibrations of type $(4,0,3)$}

In this section, we prove the following theorem:
 
\begin{thm} \label{classthm}
All fibers of primitive cyclic covering fibrations of type $(4,0,3)$ are classified into $32$ classes of types $(0_{i_1,\ldots,i_m})$, $({\rm I\hspace{-.1em}I\hspace{-.1em}I}_{i,j})$, $({\rm I\hspace{-.1em}V}_k)$, $({\rm V}_{i,j})$, $({\rm V\hspace{-.1em}I}_k)$, $({\rm V\hspace{-.1em}I\hspace{-.1em}I}_k)$ and $({\rm V\hspace{-.1em}I\hspace{-.1em}I\hspace{-.1em}I})$ plus $9$ classes of types $({\rm I}_{i,j,l})$ and $({\rm I\hspace{-.1em}I}_{k,l})$ up to $(-2)$-curves, as listed in the following tables. Moreover, the local signature of any singular fiber is negative.
\end{thm}

To explain the meaning of indices attached to (0), (I), $\dots$, (V\hspace{-.1em}I\hspace{-.1em}I\hspace{-.1em}I) and the geometric features of the corresponding fibers, it would be convenient to have a list of some numerical datum for each type of singular fibers.

\setlength\unitlength{0.22cm}
\begin{table}[H]
\begin{center}
\begin{tabular}{|c|c|c|c|}
\multicolumn{4}{c}{} \\ \hline
 & $(0_{i_1,\ldots,i_m})$ & $({\rm I}_{i,j,l})$ & $({\rm I\hspace{-.1em}I}_{k,l})$  \\ \hline
 \begin{minipage}{17mm}
 \begin{center}
 \ \\
 The\\
shape\\
of $F_p$\\
$(k=0)$
 \end{center}
 \end{minipage} &
\begin{minipage}{40mm}
   \begin{center}
    \begin{picture}(0,6)
 \put(0,-5){\line(0,1){10}}
 \put(-0.5075,0.25){$\bullet$}
 \put(-0.5075,1.85){$\bullet$}
 \put(-0.5075,3.45){$\bullet$}
 \put(-0.5075,-1.35){$\bullet$}
 \put(-0.5075,-2.95){$\bullet$}
 \put(-0.5075,-4.55){$\bullet$}
    \end{picture}
  \end{center}
 \end{minipage}&
\begin{minipage}{40mm}
  \begin{center}
   \begin{picture}(0,6)
 \put(-5,-5){\line(0,1){10}}
 \put(-5.5,0.5){\line(2,1){2.7}}
 \put(-5.5,2){\line(2,1){2.7}}
 \put(-5.5,3.5){\line(2,1){2.7}}
 \put(-1.1,0.5){\line(-2,1){2.7}}
 \put(-1.1,2){\line(-2,1){2.7}}
 \put(-1.1,3.5){\line(-2,1){2.7}}
 \put(5,-5){\line(0,1){10}}
 \put(5.5,0.5){\line(-2,1){2.7}}
 \put(5.5,2){\line(-2,1){2.7}}
 \put(5.5,3.5){\line(-2,1){2.7}}
 \put(1.1,0.5){\line(2,1){2.7}}
 \put(1.1,2){\line(2,1){2.7}}
 \put(1.1,3.5){\line(2,1){2.7}}
 \put(-5.5075,-3.05){$\bullet$}
 \put(-5.5075,-1.55){$\bullet$}
 \put(-5.5075,-4.55){$\bullet$}
 \put(4.4925,-3.05){$\bullet$}
 \put(4.4925,-1.55){$\bullet$}
 \put(4.4925,-4.55){$\bullet$}
 \put(-0.9,0.5){$\cdots$}
 \put(-0.9,2){$\cdots$}
 \put(-0.9,3.5){$\cdots$}
\put(-6.5,-4){$I$}
\put(5.5,-4){$J$}
   \end{picture}
  \end{center}
 \end{minipage}&
\begin{minipage}{40mm}
  \begin{center}
   \begin{picture}(0,6)
 \put(-5.1,-5){\line(0,1){10}}
 \put(-4.9,-5){\line(0,1){10}}
 \put(-5.5,0.5){\line(2,1){2.7}}
 \put(-5.5,2){\line(2,1){2.7}}
 \put(-5.5,3.5){\line(2,1){2.7}}
 \put(-1.1,0.5){\line(-2,1){2.7}}
 \put(-1.1,2){\line(-2,1){2.7}}
 \put(-1.1,3.5){\line(-2,1){2.7}}
 \put(5.8,0.5){\line(-2,1){2.7}}
 \put(5.8,2){\line(-2,1){2.7}}
 \put(5.8,3.5){\line(-2,1){2.7}}
 \put(1.4,0.5){\line(2,1){2.7}}
 \put(1.4,2){\line(2,1){2.7}}
 \put(1.4,3.5){\line(2,1){2.7}}
  \put(-5.5,0.7){\line(2,1){2.7}}
 \put(-5.5,2.2){\line(2,1){2.7}}
 \put(-5.5,3.7){\line(2,1){2.7}}
 \put(-1.1,0.7){\line(-2,1){2.7}}
 \put(-1.1,2.2){\line(-2,1){2.7}}
 \put(-1.1,3.7){\line(-2,1){2.7}}
 \put(5.8,0.7){\line(-2,1){2.7}}
 \put(5.8,2.2){\line(-2,1){2.7}}
 \put(5.8,3.7){\line(-2,1){2.7}}
 \put(1.4,0.7){\line(2,1){2.7}}
 \put(1.4,2.2){\line(2,1){2.7}}
 \put(1.4,3.7){\line(2,1){2.7}}
\put(3.3,0.6){\line(2,1){2.7}}
 \put(3.3,2.1){\line(2,1){2.7}}
 \put(3.3,3.6){\line(2,1){2.7}}
\put(4.3,0.1){\line(2,1){2.7}}
 \put(4.3,1.6){\line(2,1){2.7}}
 \put(4.3,3.1){\line(2,1){2.7}}
 \put(-5.5075,-3.05){$\bullet$}
 \put(-5.5075,-1.55){$\bullet$}
 \put(-5.5075,-4.55){$\bullet$}
 \put(-0.9,0.5){$\cdots$}
 \put(-0.9,2){$\cdots$}
 \put(-0.9,3.5){$\cdots$}
\put(-4.5,-4){$K$}
   \end{picture}
  \end{center}
 \end{minipage}
\\
 & & & \\
 & & & \\ \hline
 $\alpha_0(F_p)$ & $k$ $(k=0,\ldots,5)$ & $k$ $(k=0,\ldots,4)$ & $3+k$ $(k=0,1,2)$
\\ \hline
 $\alpha_1(F_p)$ & $0$ & $l+1$ $(l\ge 0)$ & $l+2$ $(l\ge 0)$ \\ \hline
 $\varepsilon(F_p)$ & $0$ & $0$ & $0$ \\ \hline
 $\mathrm{Ind}(F_p)$ & 0 & $\frac{3}{7}(l+1)$ & $\frac{3}{7}(l+2)$ \\ \hline
 $\sigma(F_p)$ & $-\frac{16}{15}k$ & $-\frac{16}{15}k-\frac{7}{5}l-\frac{7}{5}$ & $-\frac{16}{15}k-\frac{7}{5}l-6$ \\ \hline
\end{tabular}
\end{center}
\end{table}

\begin{table}[H]
\begin{center}
\begin{tabular}{|c|c|c|c|}
\multicolumn{4}{c}{} \\ \hline
 &  $({\rm I\hspace{-.1em}I\hspace{-.1em}I}_{i,j})$ & $({\rm I\hspace{-.1em}V}_k)$ & $({\rm V}_{i,j})$  \\ \hline
 \begin{minipage}{17mm}
 \begin{center}
 \ \\
 The\\
shape\\
of $F_p$\\
$(k=0)$
 \end{center}
 \end{minipage} &
\begin{minipage}{40mm}
   \begin{center}
    \begin{picture}(0,6)
 \qbezier(0,0)(4,0)(4,3)
 \qbezier(0,0)(4,0)(4,-3)
 \qbezier(0,0)(-4,0)(-4,3)
 \qbezier(0,0)(-4,0)(-4,-3)
 \put(-0.5075,-0.65){$\bullet$}
 \put(2.4925,0.25){$\bullet$}
 \put(2.4975,-1.25){$\bullet$}
 \put(-3.5075,0.25){$\bullet$}
 \put(-3.5075,-1.25){$\bullet$}
\put(-0.42,-1.5){$x$}
    \end{picture}
  \end{center}
 \end{minipage}&
\begin{minipage}{40mm}
  \begin{center}
   \begin{picture}(0,6)
 \put(0,-5){\line(0,1){10}}
 \qbezier(-4.25,-2.5)(0,0)(4.25,2.5)
 \qbezier(-4.25,2.5)(0,0)(4.25,-2.5)
 \put(-0.5075,-0.55){$\circ$}
 \put(-0.5075,1.45){$\bullet$}
 \put(1.1975,0.45){$\bullet$}
 \put(1.1975,-1.55){$\bullet$}
 \put(-2.2075,0.45){$\bullet$}
 \put(-2.2075,-1.55){$\bullet$}
 \put(-0.5075,-2.55){$\bullet$}
   \end{picture}
  \end{center}
 \end{minipage}&
 \begin{minipage}{40mm}
   \begin{center}
    \begin{picture}(0,6)
 \put(-5,-2){\line(1,0){10}} 
 \put(0.1,-5){\line(0,1){10}}
 \put(-0.1,-5){\line(0,1){10}}
 \put(-3,1){\line(1,0){6}}
 \put(-3,2){\line(1,0){6}}
 \put(-3,3){\line(1,0){6}}
 \put(-0.5075,-2.55){$\circ$}
 \put(-0.5075,-1.05){$\bullet$}
 \put(-0.5075,-4.05){$\bullet$}
 \put(-2.0075,-2.55){$\bullet$}
 \put(0.9925,-2.55){$\bullet$}
\put(-4,-4){$I$}
\put(-0.4,5.5){$J$}
    \end{picture}
  \end{center}
 \end{minipage}
\\
 & & & \\
 & & & \\ \hline
 $\alpha_0(F_p)$ & $1+k$ $(k=0,1,2,3)$ & $k$ $(k=0,1,2,3)$ & $2+k$ $(k=0,1,2)$  \\ \hline
 $\alpha_1(F_p)$ & $1$ & $3$ & $3$ \\ \hline
 $\varepsilon(F_p)$ & $0$ & $1$ & $1$ \\ \hline
 $\mathrm{Ind}(F_p)$ & $\frac{3}{7}$ & $\frac{16}{7}$ & $\frac{16}{7}$ \\ \hline
 $\sigma(F_p)$ & $-\frac{16}{15}k-\frac{37}{15}$ & $-\frac{16}{15}k-\frac{16}{15}$ & $-\frac{16}{15}k-\frac{16}{5}$ \\ \hline
\end{tabular}
\end{center}
\end{table}

\begin{table}[H]
\begin{center}
\begin{tabular}{|c|c|c|c|}
\multicolumn{4}{c}{} \\ \hline
 & $({\rm V\hspace{-.1em}I}_k)$ & $({\rm V\hspace{-.1em}I\hspace{-.1em}I}_k)$ & $({\rm V\hspace{-.1em}I\hspace{-.1em}I\hspace{-.1em}I})$ \\ \hline
 \begin{minipage}{17mm}
 \begin{center}
 \ \\
 The\\
shape\\
of $F_p$\\
$(k=0)$
 \end{center}
 \end{minipage} &
\begin{minipage}{40mm}
   \begin{center}
    \begin{picture}(0,6)
 \put(-2,-5){\line(0,1){10}}
 \put(-2.2,-5){\line(0,1){10}} 
 \put(-1.8,-5){\line(0,1){10}}  
 \put(-2.5,0.5){\line(2,1){4}}
 \put(-2.5,2){\line(2,1){4}}
 \put(-2.5,3.5){\line(2,1){4}}
 \put(-2.5,0.3){\line(2,1){4}}
 \put(-2.5,1.8){\line(2,1){4}}
 \put(-2.5,3.3){\line(2,1){4}}
 \put(4.5,0.5){\line(-2,1){4}}
 \put(4.5,2){\line(-2,1){4}}
 \put(4.5,3.5){\line(-2,1){4}}
 \put(-2.5075,-3.05){$\circ$}
 \put(-2.5075,-1.55){$\bullet$}
 \put(-2.5075,-4.55){$\bullet$}
    \end{picture}
  \end{center}
 \end{minipage}&
\begin{minipage}{40mm}
  \begin{center}
   \begin{picture}(0,6)
 \put(-5,2.5){\line(1,0){10}}
 \put(-5,-2.5){\line(1,0){10}} 
 \put(0.1,-5){\line(0,1){10}}
 \put(-0.1,-5){\line(0,1){10}}
 \put(-0.5075,1.95){$\circ$}
 \put(-0.5075,-3.05){$\circ$}
 \put(-0.5075,-0.55){$\bullet$}
 \put(1.4975,1.95){$\bullet$}
 \put(-2.5075,1.95){$\bullet$}
 \put(1.4975,-3.05){$\bullet$}
 \put(-2.5075,-3.05){$\bullet$}
   \end{picture}
  \end{center}
 \end{minipage}&
\begin{minipage}{40mm}
  \begin{center}
   \begin{picture}(0,6)
 \qbezier(-3.5,-3.5)(0,0)(3.5,3.5)
 \qbezier(-3.5,3.5)(0,0)(3.5,-3.5)
 \qbezier(-3.5,-3.3)(0,0.2)(3.5,3.7)
 \qbezier(-3.5,3.7)(0,0.2)(3.5,-3.3)
 \qbezier(-3.5,-3.7)(0,-0.2)(3.5,3.3)
 \qbezier(-3.5,3.3)(0,-0.2)(3.5,-3.7)
 \put(-0.5075,-0.55){$\circ$}
 \put(0.8925,0.85){$\circ$}
 \put(-1.9075,0.85){$\circ$}
 \put(0.8925,-1.95){$\bullet$}
 \put(-1.9075,-1.95){$\bullet$}
   \end{picture}
  \end{center}
 \end{minipage}
\\
 & & & \\
 & & & \\ \hline
 $\alpha_0(F_p)$ & $4+k$ $(k=0,1)$ & $1+k$ $(k=0,1,2)$ & $4$ \\ \hline
 $\alpha_1(F_p)$ & $3$ & $4$ & $4$ \\ \hline
 $\varepsilon(F_p)$ & $1$ & $2$ & $3$ \\ \hline
 $\mathrm{Ind}(F_p)$ & $\frac{16}{7}$ & $\frac{26}{7}$ & $\frac{33}{7}$ \\ \hline
 $\sigma(F_p)$ & $-\frac{16}{15}k-\frac{16}{3}$ & $-\frac{16}{15}k-\frac{2}{5}$ & $-\frac{7}{15}$ \\ \hline
\end{tabular}
\end{center}
\end{table}

\noindent
In the above tables, the symbol $\bullet$ stands for a point on a $1$-dimensional fixed component of $\sigma$, while the symbol $\circ$ stands for an isolated fixed point of $\sigma$. 
Note that the fibers given in the tables are the most typical (generic) ones and 
we will obtain the whole list by degenerating them, that is, some of $\bullet$'s can overlap one another to give different topological types of fibers.
A fiber of type (0) is a non-singular curve of genus $4$ in the generic 
case and, it will have a cusp when two $\bullet$'s overlap, etc.
A component with three $\bullet$ or $\circ$ in total is an elliptic curve when the three are distinct, 
or a rational curve with one smooth point ($\bullet$ or $\circ$) and one cusp $\bullet$ when two of $\bullet$ overlap, 
or three smooth rational curves intersecting in one point $\bullet$ when three of $\bullet$ overlap. 
A component with no fixed point is a smooth rational curve. 
The indices $i$, $j$, etc. are determined as follows: 
For $({\rm I}_{i,j,l})$ and $({\rm V}_{i,j})$, the index $i$ (resp. $j$) counts the number of $\bullet$ 
lost by overlapping on the curve $I$ (resp. $J$), $0\leq i,j\leq 2$.
For $({\rm I\hspace{-.1em}I\hspace{-.1em}I}_{i,j})$, the index $i$ counts 
the number of $\bullet$ lost by overlapping with $\bullet$ at $x$ and the index $j$ counts the number of $\bullet$ lost by overlapping 
not with $\bullet$ at $x$, leaving the original $\bullet$ at $x$ untouched. 
We remark that such overlappings at $x$ can occur only on one of two components meeting at $x$. 
Hence we have $(i,j)=(0,0), (0,1), (0,2)$, $(1,0), (1,1)$, $(2,0)$ and $(2,1)$ 
when $({\rm I\hspace{-.1em}I\hspace{-.1em}I}_{i,j})$.
For $(0_{i_1,\ldots,i_m})$, the indices $i_j$ are positive integers with 
$\sum_{j=1}^mi_j=6$ and $i_1\geq i_2\geq \cdots \geq i_m$; when $i_j>1$, 
one can understand that $i_j-1$ $\bullet$'s out of $6-m$ overlap with the $j$-th $\bullet$.
There are $11$ such sequences $\{i_j\}_{j=1}^m$.
The integer $k$ is the total number of $\bullet$ lost by overlapping, and then $k=\sum_{j=1}^{m} (i_j-1)$ for $(0_{i_1,\ldots,i_m})$, 
$k=i+j$ for $({\rm I}_{i,j,l})$, $({\rm I\hspace{-.1em}I\hspace{-.1em}I}_{i,j})$ or $({\rm V}_{i,j})$. 
The index $l$ counts the number of $(-2)$-curves in one chain connecting the curve $I$ with $J$ for $({\rm I}_{i,j,l})$, 
or the number of $2$-fold $(-2)$-curves in one chain intersecting with the curve $K$ for $({\rm I\hspace{-.1em}I}_{k,l})$. 
The number of topologically different fibers of types $(0_{i_1,\ldots,i_m})$, $({\rm I\hspace{-.1em}I\hspace{-.1em}I}_{i,j})$, 
$({\rm I\hspace{-.1em}V}_k)$, $({\rm V}_{i,j})$, $({\rm V\hspace{-.1em}I}_k)$, $({\rm V\hspace{-.1em}I\hspace{-.1em}I}_k)$ 
and $({\rm V\hspace{-.1em}I\hspace{-.1em}I\hspace{-.1em}I})$ is $11$, $7$, $4$, $4$, $2$, $3$ and $1$, respectively. 
The number of topologically different fibers of types $({\rm I}_{i,j,l})$ and $({\rm I\hspace{-.1em}I}_{k,l})$ up to the number $l$ is $6$ and $3$, 
respectively.\\
 \\
\indent In order to classify fibers of primitive cyclic covering fibrations of genus $4$ of type $(4,0,3)$, it is sufficient to classify admissible sequences $\mathcal{D}_1,\mathcal{D}_2,\dots,\mathcal{D}_N$ of abstract singularity diagrams with $n=3$, $t^1=r=6$ and $m_{i,1}^1\le 3+k$ if $\mathcal{D}_1$ is of type $k$ from Lemma~\ref{eltrlem}.
We proceed with the following steps.

\begin{enumerate}

\item[(i)] Classify abstract singularity diagrams of type $k$ for $(3,6)$ with $m_{i,1}\le 3+k$ for $k=0,1$.

\item[(ii)] Classify admissible sequences of abstract singularity diagrams with $n=3$, $t^1=3+k$ and $\mathcal{D}_1$ is of type $k$ for $k=0,1$.

\item[(iii)] Classify admissible sequences of abstract singularity diagrams with $n=3$, $t^1=r=6$ and $m_{i,1}^1\le 3+k$ if $\mathcal{D}_1$ is of type $k$ for $k=0,1$.

\end{enumerate}

\noindent (i) All abstract singularity diagrams of type $0$ for $(3,6)$ with $m_{i,1}\le 3$ are as follows.

\begin{itemize}
\item $c=0$

\begin{table}[H]
\begin{center}
\begin{tabular}{c}

 \begin{minipage}{0.7\hsize}
 \begin{tabular}{c}
 \\
 $(6)$\ \ \ \  $(5,1)$\ \ \ \  $(4,2)$\ \ \ \  $(4,1,1)$\ \  $\cdots$\ \ $(2,1,1,1,1)$\ \ \ \   $(1,1,1,1,1,1)$
 \end{tabular}
 \end{minipage}

\end{tabular}
\end{center}
\end{table}

\item $c=1$

\begin{table}[H]
\begin{center}
\begin{tabular}{c}

 \begin{minipage}{0.2\hsize}
 \begin{tabular}{cc}
  & \\ \cline{1-1}
 \multicolumn{1}{|l|}{3}& $(3)$ \\ \cline{1-1}
 \end{tabular}
 \end{minipage}

\begin{minipage}{0.2\hsize}
 \begin{tabular}{cc}
  & \\ \cline{1-1}
 \multicolumn{1}{|l|}{3}& $(2,1)$ \\ \cline{1-1}
 \end{tabular}
 \end{minipage}

\begin{minipage}{0.2\hsize}
 \begin{tabular}{cc}
  & \\ \cline{1-1}
 \multicolumn{1}{|l|}{3}& $(1,1,1)$ \\ \cline{1-1}
 \end{tabular}
 \end{minipage}

\begin{minipage}{0.1\hsize}
 \begin{tabular}{cc}
  $(1)$ & \\ \cline{1-1}
 \multicolumn{1}{|c|}{3}& $(2)$ \\ \cline{1-1}
 \end{tabular}
 \end{minipage}

\end{tabular}
\end{center}
\end{table}

\begin{table}[H]
\begin{center}
\begin{tabular}{c}

\begin{minipage}{0.2\hsize}
 \begin{tabular}{cc}
 $(1)$ & \\ \cline{1-1}
 \multicolumn{1}{|c|}{3}& $(1,1)$ \\ \cline{1-1}
 \end{tabular}
 \end{minipage}

\begin{minipage}{0.2\hsize}
 \begin{tabular}{cc}
  $(2)$ & \\ \cline{1-1}
 \multicolumn{1}{|c|}{3}& $(1)$ \\ \cline{1-1}
 \end{tabular}
 \end{minipage}

\begin{minipage}{0.1\hsize}
 \begin{tabular}{cc}
  $(3)$& \\ \cline{1-1}
 \multicolumn{1}{|c|}{3}& \\ \cline{1-1}
 \end{tabular}
 \end{minipage}

\end{tabular}
\end{center}
\end{table}

\item $c=2$

\begin{table}[H]
\begin{center}
\begin{tabular}{c}

 \begin{minipage}{0.2\hsize}
 \begin{tabular}{cc}
  & \\ \cline{1-2}
 \multicolumn{1}{|l|}{3}& \multicolumn{1}{|l|}{3} \\ \cline{1-2}
 \end{tabular}
 \end{minipage}

\begin{minipage}{0.2\hsize}
 \begin{tabular}{c}
  \\ \cline{1-1}
 \multicolumn{1}{|l|}{3} \\ \cline{1-1}
 \multicolumn{1}{|l|}{3} \\ \cline{1-1}
 \end{tabular}
 \end{minipage}

\end{tabular}
\end{center}
\end{table}

\end{itemize}

All abstract singularity diagrams of type $1$ for $(3,6)$ with $m_{i,1}\le 4$ are as follows.

\begin{table}[H]
\begin{center}
\begin{tabular}{c}

 \begin{minipage}{0.2\hsize}
 \begin{tabular}{ccc}
  & & \\ \cline{1-3}
 \multicolumn{1}{|l|}{3}& \multicolumn{1}{|l|}{3} & \multicolumn{1}{|l|}{3}\\ \cline{1-3}
 \end{tabular}
 \end{minipage}

 \begin{minipage}{0.2\hsize}
 \begin{tabular}{cc}
  & \\ \cline{1-1}
 \multicolumn{1}{|l|}{3}& \\ \cline{1-2}
 \multicolumn{1}{|l|}{3}& \multicolumn{1}{|l|}{3} \\ \cline{1-2}
 \end{tabular}
 \end{minipage}

 \begin{minipage}{0.2\hsize}
 \begin{tabular}{cc}
  & \\ \cline{1-1}
 \multicolumn{1}{|l|}{3}& \\ \cline{1-2}
 \multicolumn{1}{|l|}{4}& \multicolumn{1}{|l|}{3} \\ \cline{1-2}
 \end{tabular}
 \end{minipage}

 \begin{minipage}{0.2\hsize}
 \begin{tabular}{c}
   \\ \cline{1-1}
 \multicolumn{1}{|l|}{3} \\ \cline{1-1}
 \multicolumn{1}{|l|}{3} \\ \cline{1-1}
 \multicolumn{1}{|l|}{3} \\ \cline{1-1}
 \end{tabular}
 \end{minipage}

 \begin{minipage}{0.05\hsize}
 \begin{tabular}{c}
   \\ \cline{1-1}
 \multicolumn{1}{|l|}{3} \\ \cline{1-1}
 \multicolumn{1}{|l|}{4} \\ \cline{1-1}
 \multicolumn{1}{|l|}{4} \\ \cline{1-1}
 \end{tabular}
 \end{minipage}
\end{tabular}
\end{center}
\end{table}

\noindent (ii) All admissible sequences of abstract singularity diagrams with $n=3$, $t^1=3$ and $\mathcal{D}_1$ being of type $0$ are as follows.

\begin{table}[H]
\begin{center}
\begin{tabular}{c}

 \begin{minipage}{0.15\hsize}
 \begin{tabular}{c}
   \\ \cline{1-1}
 \multicolumn{1}{|c|}{$(x_1,3)$}\\ \cline{1-1}
 \multicolumn{1}{c}{\lower1ex\hbox{$\mathcal{D}^0$}}
 \end{tabular}
 \end{minipage}

\begin{minipage}{0.15\hsize}
 \begin{tabular}{c}
   \\ \cline{1-1}
 \multicolumn{1}{|c|}{$(x_2,3)$}\\ \cline{1-1}
 \multicolumn{1}{c}{\lower1ex\hbox{$\mathcal{D}_{x_1}^0$}}
 \end{tabular}
 \end{minipage}

\begin{minipage}{0.1\hsize}
 \begin{tabular}{c}
  $\cdots$
 \end{tabular}
 \end{minipage}

\begin{minipage}{0.15\hsize}
 \begin{tabular}{c}
   \\ \cline{1-1}
 \multicolumn{1}{|c|}{$(x_k,3)$}\\ \cline{1-1}
 \multicolumn{1}{c}{\lower1ex\hbox{$\mathcal{D}_{x_{k-1}}^0$}}
 \end{tabular}
 \end{minipage}

 \begin{minipage}{0.25\hsize}
 \begin{tabular}{ccc}
  & & \\
 $(1,1,1)$ or &$(2,1)$ or &$(3)$\\
 \multicolumn{3}{c}{\lower1ex\hbox{$\mathcal{D}_{x_k}^0$}}
 \end{tabular}
 \end{minipage}

\end{tabular}
\end{center}
\end{table}

\noindent
Here and in the sequel, $\mathcal{D}_*^k$ means the diagram $\mathcal{D}_*$ is of type $k$.

All admissible sequences of abstract singularity diagrams with $n=3$, $t^1=4$ and $\mathcal{D}_1$ being of type $1$ are the following $3$ classes.

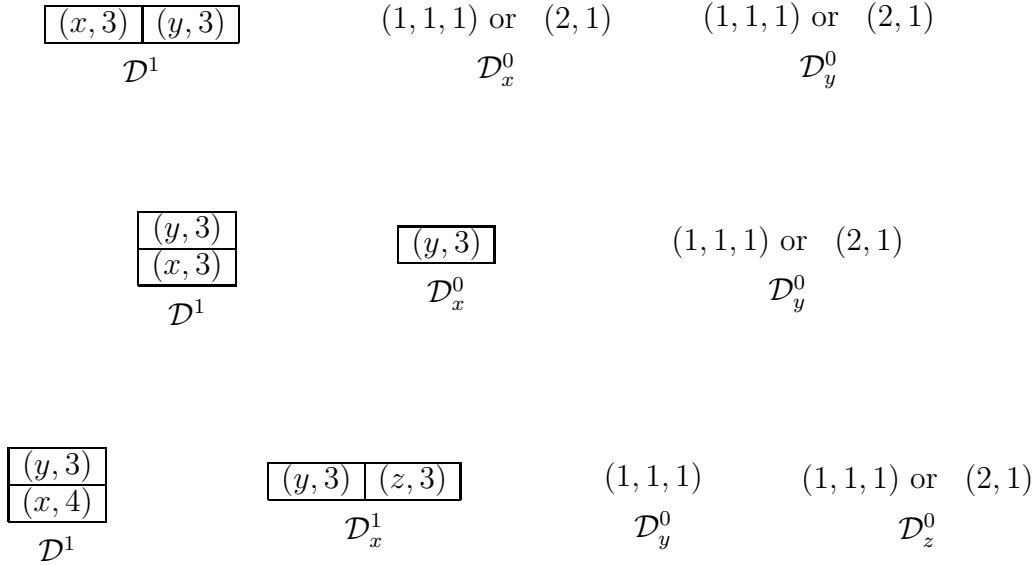
\begin{table}[H]
\begin{center}
\begin{tabular}{c}

 \begin{minipage}{0.25\hsize}
 \begin{tabular}{cc}
   & \\ \cline{1-2}
 \multicolumn{1}{|c|}{$(x,3)$} &  \multicolumn{1}{|c|}{$(y,3)$} \\ \cline{1-2}
 \multicolumn{2}{c}{\lower1ex\hbox{$\mathcal{D}^1$}}
 \end{tabular}
 \end{minipage}

\begin{minipage}{0.25\hsize}
 \begin{tabular}{cc}
  & \\
 $(1,1,1)$ or &$(2,1)$\\
 \multicolumn{2}{c}{\lower1ex\hbox{$\mathcal{D}_x^0$}}
 \end{tabular}
 \end{minipage}

 \begin{minipage}{0.25\hsize}
 \begin{tabular}{cc}
  & \\
 $(1,1,1)$ or &$(2,1)$ \\
 \multicolumn{2}{c}{\lower1ex\hbox{$\mathcal{D}_y^0$}}
 \end{tabular}
 \end{minipage}

\end{tabular}
\end{center}
\end{table}

\begin{table}[H]
\begin{center}
\begin{tabular}{c}

 \begin{minipage}{0.2\hsize}
 \begin{tabular}{c}
  \\ \cline{1-1}
 \multicolumn{1}{|c|}{$(y,3)$} \\ \cline{1-1}
 \multicolumn{1}{|c|}{$(x,3)$} \\ \cline{1-1}
 \multicolumn{1}{c}{\lower1ex\hbox{$\mathcal{D}^1$}}
 \end{tabular}
 \end{minipage}

 \begin{minipage}{0.2\hsize}
 \begin{tabular}{c}
  \\ \cline{1-1}
 \multicolumn{1}{|c|}{$(y,3)$} \\ \cline{1-1}
 \multicolumn{1}{c}{\lower1ex\hbox{$\mathcal{D}_x^0$}}
 \end{tabular}
 \end{minipage}

 \begin{minipage}{0.2\hsize}
 \begin{tabular}{cc}
  & \\
 $(1,1,1)$ or &$(2,1)$ \\
 \multicolumn{2}{c}{\lower1ex\hbox{$\mathcal{D}_y^0$}}
 \end{tabular}
 \end{minipage}

\end{tabular}
\end{center}
\end{table}

\begin{table}[H]
\begin{center}
\begin{tabular}{c}

 \begin{minipage}{0.2\hsize}
 \begin{tabular}{c}
  \\ \cline{1-1}
 \multicolumn{1}{|c|}{$(y,3)$} \\ \cline{1-1}
 \multicolumn{1}{|c|}{$(x,4)$} \\ \cline{1-1}
 \multicolumn{1}{c}{\lower1ex\hbox{$\mathcal{D}^1$}}
 \end{tabular}
 \end{minipage}

 \begin{minipage}{0.25\hsize}
 \begin{tabular}{cc}
   & \\ \cline{1-2}
 \multicolumn{1}{|c|}{$(y,3)$} &  \multicolumn{1}{|c|}{$(z,3)$} \\ \cline{1-2}
 \multicolumn{2}{c}{\lower1ex\hbox{$\mathcal{D}_x^1$}}
 \end{tabular}
 \end{minipage}

 \begin{minipage}{0.15\hsize}
 \begin{tabular}{c}
  \\
 $(1,1,1)$ \\
 \multicolumn{1}{c}{\lower1ex\hbox{$\mathcal{D}_y^0$}}
 \end{tabular}
 \end{minipage}

\begin{minipage}{0.2\hsize}
 \begin{tabular}{cc}
  & \\
 $(1,1,1)$ or &$(2,1)$ \\
 \multicolumn{2}{c}{\lower1ex\hbox{$\mathcal{D}_z^0$}}
 \end{tabular}
 \end{minipage}

\end{tabular}
\end{center}
\end{table}

\noindent (iii) In order to classify admissible sequences of abstract singularity diagrams with $n=3$, $t^1=r=6$ and $m_{i,1}^1\le 3+k$ if $\mathcal{D}_1$ is of type $k$, we may consider admissible sequences to be continued from each column of diagrams classified in (i).
Using the classification in (ii), we can classify them as follows.

\begin{enumerate}

\item 
\begin{table}[H]
\begin{center}
\begin{tabular}{c}

 \begin{minipage}{0.15\hsize}
 \begin{tabular}{c}
   \\ \cline{1-1}
 \multicolumn{1}{|c|}{$(x_1,3)$}\\ \cline{1-1}
 \multicolumn{1}{c}{\lower1ex\hbox{$\mathcal{D}^0$}}
 \end{tabular}
 \end{minipage}

\begin{minipage}{0.12\hsize}
 \begin{tabular}{c}
   \\ \cline{1-1}
 \multicolumn{1}{|c|}{$(x_2,3)$}\\ \cline{1-1}
 \multicolumn{1}{c}{\lower1ex\hbox{$\mathcal{D}_{x_1}^0$}}
 \end{tabular}
 \end{minipage}

\begin{minipage}{0.08\hsize}
 \begin{tabular}{c}
  $\cdots$
 \end{tabular}
 \end{minipage}

\begin{minipage}{0.15\hsize}
 \begin{tabular}{c}
   \\ \cline{1-1}
 \multicolumn{1}{|c|}{$(x_k,3)$}\\ \cline{1-1}
 \multicolumn{1}{c}{\lower1ex\hbox{$\mathcal{D}_{x_{k-1}}^0$}}
 \end{tabular}
 \end{minipage}

 \begin{minipage}{0.25\hsize}
 \begin{tabular}{ccc}
  & & \\
 $(1,1,1)$ or &$(2,1)$ or &$(3)$\\
 \multicolumn{3}{c}{\lower1ex\hbox{$\mathcal{D}_{x_k}^0$}}
 \end{tabular}
 \end{minipage}

\end{tabular}
\end{center}
\end{table}

\setlength\unitlength{0.28cm}
\begin{figure}[H]
\begin{center}
\begin{tabular}{c}

 \begin{minipage}{0.6\hsize}
 \begin{center}
\begin{picture}(13,5)
 \put(-7.5,-1){$x_1$}
 \multiput(-13,0)(0.4,0){25}{\line(1,0){0.2}}
 \qbezier(-8,0)(-8,3)(-5,3)
 \qbezier(-8,0)(-8,3)(-11,3)
 \qbezier(-8,0)(-8,2.5)(-4,2.5)
 \qbezier(-8,0)(-8,-3)(-5,-3)
 \qbezier(-8,0)(-8,-3)(-11,-3)
 \qbezier(-8,0)(-8,-2.5)(-4,-2.5)
 \multiput(9,-4)(0.4,0){10}{\line(1,0){0.2}}
 \multiput(10,-5)(0,0.4){10}{\line(0,1){0.2}}
 \multiput(7,-2)(0.4,0){10}{\line(1,0){0.2}}
 \multiput(8,-3)(0,0.4){10}{\line(0,1){0.2}}
 \multiput(4,1)(0,0.4){10}{\line(0,1){0.2}}
 \put(3,2){\line(1,0){2}}
 \put(3,3){\line(1,0){2}}
 \put(3,4){\line(1,0){2}}
 \put(6,1){$\cdots$}
 \put(-0.5,1){$k$}
 \put(2,0){\vector(-1,0){1.5}}
 \multiput(-0.5,0)(0.4,0){3}{\line(1,0){0.2}}
 \put(-0.5,0){\vector(-1,0){1.5}}
 \end{picture}
 \end{center}
 \end{minipage}

\begin{minipage}{0.2\hsize}
 \begin{center}
 \begin{picture}(13,5)
 \put(-7.5,-1){$x_1$}
 \multiput(-13,0)(0.4,0){25}{\line(1,0){0.2}}
 \qbezier(-8,0)(-8,3)(-5,3)
 \qbezier(-8,0)(-8,3)(-11,3)
 \qbezier(-8,0)(-8,2.5)(-4,2.5)
 \qbezier(-8,0)(-8,-2.5)(-4,-2.5)
 \multiput(9,-4)(0.4,0){10}{\line(1,0){0.2}}
 \multiput(10,-5)(0,0.4){10}{\line(0,1){0.2}}
 \multiput(7,-2)(0.4,0){10}{\line(1,0){0.2}}
 \multiput(8,-3)(0,0.4){10}{\line(0,1){0.2}}
 \multiput(4,1)(0,0.4){10}{\line(0,1){0.2}}
 \qbezier(4,2.5)(4,3.5)(5.5,3.5)
 \qbezier(4,2.5)(4,1.5)(5.5,1.5)
 \put(3,4){\line(1,0){2}}
 \put(6,1){$\cdots$}
 \put(-0.5,1){$k$}
 \put(2,0){\vector(-1,0){1.5}}
 \multiput(-0.5,0)(0.4,0){3}{\line(1,0){0.2}}
 \put(-0.5,0){\vector(-1,0){1.5}}
 \end{picture}
 \end{center}
 \end{minipage}

\end{tabular}
\end{center}
\end{figure}


\setlength\unitlength{0.28cm}
\begin{figure}[H]
\begin{center}
\begin{tabular}{c}

 \begin{minipage}{0\hsize}
 \begin{center}
 \begin{picture}(13,5)
 \put(-7.5,-1){$x_1$}
 \multiput(-13,0)(0.4,0){25}{\line(1,0){0.2}}
 \qbezier(-8,0)(-8,3)(-5.5,3)
 \qbezier(-8,0)(-8,3)(-10.5,3)
 \multiput(9,-4)(0.4,0){10}{\line(1,0){0.2}}
 \multiput(10,-5)(0,0.4){10}{\line(0,1){0.2}}
 \multiput(7,-2)(0.4,0){10}{\line(1,0){0.2}}
 \multiput(8,-3)(0,0.4){10}{\line(0,1){0.2}}
 \multiput(4,1)(0,0.4){10}{\line(0,1){0.2}}
 \qbezier(4,3)(4,5)(5.5,5)
 \qbezier(4,3)(4,1)(2.5,1)
 \put(6,1){$\cdots$}
 \put(-0.5,1){$k$}
 \put(2,0){\vector(-1,0){1.5}}
 \multiput(-0.5,0)(0.4,0){3}{\line(1,0){0.2}}
 \put(-0.5,0){\vector(-1,0){1.5}}
 \end{picture}
 \end{center}
 \end{minipage}

\end{tabular}
\end{center}
\end{figure}

\bigskip

\item 
\begin{table}[H]
\begin{center}
\begin{tabular}{c}

 \begin{minipage}{0.2\hsize}
 \begin{tabular}{c}
 $(1^*)$  \\ \cline{1-1}
 \multicolumn{1}{|c|}{$(x,3)$}\\ \cline{1-1}
 \multicolumn{1}{c}{\lower1ex\hbox{$\mathcal{D}^0$}}
 \end{tabular}
 \end{minipage}

 \begin{minipage}{0.45\hsize}
 \begin{tabular}{cccc}
  & & \\
 $(1^*,1,1)$ or &$(2^*,1)$ or &$(2,1^*)$ or &$(3^*)$\\
 \multicolumn{4}{c}{\lower1ex\hbox{$\mathcal{D}_x^0$}}
 \end{tabular}
 \end{minipage}

\end{tabular}
\end{center}
\end{table}

\setlength\unitlength{0.28cm}
\begin{figure}[H]
\begin{center}
\begin{tabular}{c}

 \begin{minipage}{0.7\hsize}
 \begin{center}
 \begin{picture}(13,5)
 \put(-8.5,-1.5){$x$}
 \multiput(-13,0)(0.4,0){25}{\line(1,0){0.2}}
 \qbezier(-11,-3)(-8,0)(-5,3)
 \qbezier(-11,3)(-8,0)(-5,-3)
 \qbezier(-8,0)(-5,0)(-5,2)
 \qbezier(-8,0)(-11,0)(-11,2)
 \multiput(3,0)(0.4,0){25}{\line(1,0){0.2}}
 \qbezier(6,4)(8,4)(10,4)
 \qbezier(6,2)(8,2)(10,2)
 \qbezier(5,-1)(8,0)(11,1)
 \put(2,0){\vector(-1,0){4}}
 \multiput(8,-5)(0,0.4){25}{\line(0,1){0.2}}
 \end{picture}
 \end{center}
 \end{minipage}

 \begin{minipage}{0.7\hsize}
 \begin{picture}(13,5)
 \put(-7.5,-1.3){$x$}
 \multiput(-13,0)(0.4,0){25}{\line(1,0){0.2}}
 \qbezier(-8,0)(-5,0)(-5,3)
 \qbezier(-8,0)(-5,0)(-5,-3)
 \put(-8,-4){\line(0,1){8}}
 \multiput(3,0)(0.4,0){25}{\line(1,0){0.2}}
 \qbezier(8,0)(8,2.5)(10.5,2.5)
 \qbezier(8,0)(8,-2.5)(10.5,-2.5)
 \put(6,4){\line(1,0){4}}
 \put(2,0){\vector(-1,0){4}}
 \multiput(8,-5)(0,0.4){25}{\line(0,1){0.2}}
 \end{picture}
 \end{minipage}

\end{tabular}
\end{center}
\end{figure}

\ \\

\setlength\unitlength{0.28cm}
\begin{figure}[H]
\begin{center}
\begin{tabular}{c}

 \begin{minipage}{0.7\hsize}
 \begin{center}
 \begin{picture}(13,5)
 \put(-8.5,-1.3){$x$}
 \multiput(-13,0)(0.4,0){25}{\line(1,0){0.2}}
 \qbezier(-8,0)(-8,3)(-5,3)
 \qbezier(-8,0)(-8,3)(-11,3)
 \qbezier(-8,0)(-5,0)(-5,2)
 \qbezier(-8,0)(-11,0)(-11,2)
 \multiput(3,0)(0.4,0){25}{\line(1,0){0.2}}
 \qbezier(8,3)(8,4.5)(9.5,4.5)
 \qbezier(8,3)(8,1.5)(9.5,1.5)
 \qbezier(5,-1)(8,0)(11,1)
 \put(2,0){\vector(-1,0){4}}
 \multiput(8,-5)(0,0.4){25}{\line(0,1){0.2}}
 \end{picture}
 \end{center}
 \end{minipage}

 \begin{minipage}{0.7\hsize}
 \begin{picture}(13,5)
 \put(-8.5,-1.3){$x$}
 \multiput(-13,0)(0.4,0){25}{\line(1,0){0.2}}
 \qbezier(-8,0)(-5,0)(-5,2.5)
 \qbezier(-8,0)(-5,0)(-5,-2.5)
 \multiput(3,0)(0.4,0){25}{\line(1,0){0.2}}
 \qbezier(8,0)(8,2)(10,2)
 \qbezier(8,0)(8,-2)(6,-2)
 \put(2,0){\vector(-1,0){4}}
 \multiput(8,-5)(0,0.4){25}{\line(0,1){0.2}}
 \end{picture}
 \end{minipage}

\end{tabular}
\end{center}
\end{figure}

\bigskip

\item 
\begin{table}[H]
\begin{center}
\begin{tabular}{c}

 \begin{minipage}{0.2\hsize}
 \begin{tabular}{c}
 $(2)$  \\ \cline{1-1}
 \multicolumn{1}{|c|}{$(x,3)$}\\ \cline{1-1}
 \multicolumn{1}{c}{\lower1ex\hbox{$\mathcal{D}^0$}}
 \end{tabular}
 \end{minipage}

 \begin{minipage}{0.2\hsize}
 \begin{tabular}{cc}
  & \\
 $(1,1,1)$ or &$(2,1)$ \\
 \multicolumn{2}{c}{\lower1ex\hbox{$\mathcal{D}_x^0$}}
 \end{tabular}
 \end{minipage}

\end{tabular}
\end{center}
\end{table}

\setlength\unitlength{0.28cm}
\begin{figure}[H]
\begin{center}
\begin{tabular}{c}

 \begin{minipage}{0.55\hsize}
 \begin{center}
 \begin{picture}(13,5)
 \put(-8.5,-1.3){$x$}
 \multiput(-13,0)(0.4,0){25}{\line(1,0){0.2}}
 \qbezier(-11,-3)(-8,0)(-5,3)
 \qbezier(-11,3)(-8,0)(-5,-3)
 \qbezier(-8,0)(-5,0)(-5,2)
 \qbezier(-8,0)(-11,0)(-11,-2)
 \multiput(3,0)(0.4,0){25}{\line(1,0){0.2}}
 \qbezier(6,4)(8,4)(10,4)
 \qbezier(6,2)(8,2)(10,2)
 \qbezier(8,0)(10.5,0)(10.5,1.5)
 \qbezier(8,0)(5.5,0)(5.5,1.5)
 \put(2,0){\vector(-1,0){4}}
 \multiput(8,-5)(0,0.4){25}{\line(0,1){0.2}}
 \end{picture}
 \end{center}
 \end{minipage}

 \begin{minipage}{0.25\hsize}
 \begin{center}
 \begin{picture}(9,5)
 \put(-8.5,-1.3){$x$}
 \multiput(-13,0)(0.4,0){25}{\line(1,0){0.2}}
 \qbezier(-8,0)(-8,3)(-5,3)
 \qbezier(-8,0)(-8,3)(-11,3)
 \qbezier(-8,0)(-5,0)(-5,2)
 \qbezier(-8,0)(-11,0)(-11,-2)
 \multiput(3,0)(0.4,0){25}{\line(1,0){0.2}}
 \qbezier(8,3)(8,4.5)(9.5,4.5)
 \qbezier(8,3)(8,1.5)(9.5,1.5)
 \qbezier(8,0)(10.5,0)(10.5,1.5)
 \qbezier(8,0)(5.5,0)(5.5,1.5)
 \put(2,0){\vector(-1,0){4}}
 \multiput(8,-5)(0,0.4){25}{\line(0,1){0.2}}
 \end{picture}
 \end{center}
 \end{minipage}

\end{tabular}
\end{center}
\end{figure}

\bigskip

\item 
\begin{table}[H]
\begin{center}
\begin{tabular}{c}

 \begin{minipage}{0.2\hsize}
 \begin{tabular}{c}
 $(3)$  \\ \cline{1-1}
 \multicolumn{1}{|c|}{$(x,3)$}\\ \cline{1-1}
 \multicolumn{1}{c}{\lower1ex\hbox{$\mathcal{D}^0$}}
 \end{tabular}
 \end{minipage}

 \begin{minipage}{0.2\hsize}
 \begin{tabular}{cc}
  & \\
 $(1,1,1)$ or &$(2,1)$ \\
 \multicolumn{2}{c}{\lower1ex\hbox{$\mathcal{D}_x^0$}}
 \end{tabular}
 \end{minipage}

\end{tabular}
\end{center}
\end{table}

\setlength\unitlength{0.28cm}
\begin{figure}[H]
\begin{center}
\begin{tabular}{c}

 \begin{minipage}{0.55\hsize}
 \begin{center}
 \begin{picture}(13,5)
 \put(-8.5,-1.3){$x$}
 \multiput(-13,0)(0.4,0){25}{\line(1,0){0.2}}
 \qbezier(-11,-3)(-8,0)(-5,3)
 \qbezier(-11,3)(-8,0)(-5,-3)
 \qbezier(-8,0)(-5,0)(-5,2)
 \qbezier(-8,0)(-11,0)(-11,2)
 \multiput(3,0)(0.4,0){25}{\line(1,0){0.2}}
 \qbezier(6,4)(8,4)(10,4)
 \qbezier(6,2)(8,2)(10,2)
 \qbezier(8,0)(10.5,0)(10.5,1.5)
 \qbezier(8,0)(5.5,0)(5.5,-1.5)
 \put(2,0){\vector(-1,0){4}}
 \multiput(8,-5)(0,0.4){25}{\line(0,1){0.2}}
 \end{picture}
 \end{center}
 \end{minipage}

 \begin{minipage}{0.25\hsize}
 \begin{center}
 \begin{picture}(9,5)
 \put(-8.5,-1.3){$x$}
 \multiput(-13,0)(0.4,0){25}{\line(1,0){0.2}}
 \qbezier(-8,0)(-8,3)(-5,3)
 \qbezier(-8,0)(-8,3)(-11,3)
 \qbezier(-8,0)(-5,0)(-5,2)
 \qbezier(-8,0)(-11,0)(-11,2)
 \multiput(3,0)(0.4,0){25}{\line(1,0){0.2}}
 \qbezier(8,3)(8,4.5)(9.5,4.5)
 \qbezier(8,3)(8,1.5)(9.5,1.5)
 \qbezier(8,0)(10.5,0)(10.5,1.5)
 \qbezier(8,0)(5.5,0)(5.5,-1.5)
 \put(2,0){\vector(-1,0){4}}
 \multiput(8,-5)(0,0.4){25}{\line(0,1){0.2}}
 \end{picture}
 \end{center}
 \end{minipage}

\end{tabular}
\end{center}
\end{figure}

\bigskip

\item 
\begin{table}[H]
\begin{center}
\begin{tabular}{c}

 \begin{minipage}{0.15\hsize}
 \begin{tabular}{c}
   \\ \cline{1-1}
 \multicolumn{1}{|c|}{$(y_1,3)$}\\ \cline{1-1}
 \multicolumn{1}{|c|}{$(x,3)$}\\ \cline{1-1}
 \multicolumn{1}{c}{\lower1ex\hbox{$\mathcal{D}^0$}}
 \end{tabular}
 \end{minipage}

 \begin{minipage}{0.12\hsize}
 \begin{tabular}{c}
  \\ \cline{1-1}
 \multicolumn{1}{|c|}{$(y_1,3)$} \\ \cline{1-1}
 \multicolumn{1}{c}{\lower1ex\hbox{$\mathcal{D}_x^0$}}
 \end{tabular}
 \end{minipage}

 \begin{minipage}{0.08\hsize}
 \begin{tabular}{c}
  \\ 
 $\cdots$ \\
 \end{tabular}
 \end{minipage}

\begin{minipage}{0.15\hsize}
 \begin{tabular}{c}
  \\ \cline{1-1}
 \multicolumn{1}{|c|}{$(y_k,3)$} \\ \cline{1-1}
 \multicolumn{1}{c}{\lower1ex\hbox{$\mathcal{D}_{y_{k-1}}^0$}}
 \end{tabular}
 \end{minipage}

 \begin{minipage}{0.3\hsize}
 \begin{tabular}{ccc}
 & & \\
 $(1,1,1)$ or &$(2,1)$ or &$(3)$\\
 \multicolumn{3}{c}{\lower1ex\hbox{$\mathcal{D}_{y_k}^0$}}
 \end{tabular}
 \end{minipage}

\end{tabular}
\end{center}
\end{table}

\setlength\unitlength{0.28cm}
\begin{figure}[H]
\begin{center}
\begin{tabular}{c}

 \begin{minipage}{0\hsize}
 \begin{center}
 \begin{picture}(21,5)
 \put(-16.5,-1.3){$x$}
 \multiput(-21,0)(0.4,0){25}{\line(1,0){0.2}}
 \qbezier(-16,0)(-19,0)(-19,3)
 \qbezier(-16,0)(-19,0)(-19,-3)
 \qbezier(-16,0)(-13,0)(-13,3)
 \qbezier(-16,0)(-13,0)(-13,-3)
 \qbezier(-16,0)(-19,0)(-20,1.5)
 \qbezier(-16,0)(-13,0)(-12,1.5)
 \put(0.8,-1.3){$y_1$}
 \multiput(-5,0)(0.4,0){25}{\line(1,0){0.2}}
 \qbezier(0,0)(0,3)(3,3)
 \qbezier(0,0)(0,3)(-3,3)
 \qbezier(0,0)(0,2.5)(4,2.5)
 \qbezier(0,0)(0,-3)(3,-3)
 \qbezier(0,0)(0,-3)(-3,-3)
 \qbezier(0,0)(0,-2.5)(4,-2.5)
 \multiput(17,-4)(0.4,0){10}{\line(1,0){0.2}}
 \multiput(17,-2)(0.4,0){10}{\line(1,0){0.2}}
 \multiput(18,-5)(0,0.4){10}{\line(0,1){0.2}}
 \multiput(15,-3)(0.4,0){10}{\line(1,0){0.2}}
 \multiput(16,-4)(0,0.4){10}{\line(0,1){0.2}}
 \multiput(12,1)(0,0.4){10}{\line(0,1){0.2}}
 \put(11,2){\line(1,0){2}}
 \put(11,3){\line(1,0){2}}
 \put(11,4){\line(1,0){2}}
 \put(14,1){$\cdots$}
 \put(7.5,1){$k$}
\put(-6,0){\vector(-1,0){4}}
 \put(10,0){\vector(-1,0){1.5}}
 \multiput(7.5,0)(0.4,0){3}{\line(1,0){0.2}}
 \put(7.5,0){\vector(-1,0){1.5}}
 \end{picture}
 \end{center}
 \end{minipage}

\end{tabular}
\end{center}
\end{figure}

\ \\

\setlength\unitlength{0.28cm}
\begin{figure}[H]
\begin{center}
\begin{tabular}{c}

 \begin{minipage}{0\hsize}
 \begin{center}
 \begin{picture}(21,5)
 \put(-16.5,-1.3){$x$}
 \multiput(-21,0)(0.4,0){25}{\line(1,0){0.2}}
 \qbezier(-16,0)(-13,0)(-13,3)
 \qbezier(-16,0)(-13,0)(-13,-3)
 \qbezier(-16,0)(-19,0)(-20,1.5)
 \qbezier(-16,0)(-13,0)(-12,1.5)
 \put(-0.7,-1.3){$y_1$}
 \multiput(-5,0)(0.4,0){25}{\line(1,0){0.2}}
 \qbezier(0,0)(3,3)(4,1.5)
 \qbezier(0,0)(3,3)(1.5,4)
 \qbezier(2,-2)(0,0)(-2,2)
 \multiput(17,-4)(0.4,0){10}{\line(1,0){0.2}}
 \multiput(17,-2)(0.4,0){10}{\line(1,0){0.2}}
 \multiput(18,-5)(0,0.4){10}{\line(0,1){0.2}}
 \multiput(15,-3)(0.4,0){10}{\line(1,0){0.2}}
 \multiput(16,-4)(0,0.4){10}{\line(0,1){0.2}}
 \multiput(12,1)(0,0.4){10}{\line(0,1){0.2}}
 \qbezier(12,3.5)(12,4.5)(13,4.5)
 \qbezier(12,3.5)(12,2.5)(13,2.5)
 \put(11,2){\line(1,0){2}}
 \put(14,1){$\cdots$}
 \put(7.5,1){$k$}
\put(-6,0){\vector(-1,0){4}}
 \put(10,0){\vector(-1,0){1.5}}
 \multiput(7.5,0)(0.4,0){3}{\line(1,0){0.2}}
 \put(7.5,0){\vector(-1,0){1.5}}
 \end{picture}
 \end{center}
 \end{minipage}

\end{tabular}
\end{center}
\end{figure}

\ \\

\setlength\unitlength{0.28cm}
\begin{figure}[H]
\begin{center}
\begin{tabular}{c}

 \begin{minipage}{0\hsize}
 \begin{center}
 \begin{picture}(21,5)
 \put(-16.5,-1.3){$x$}
 \multiput(-21,0)(0.4,0){25}{\line(1,0){0.2}}
 \qbezier(-16,0)(-13,0)(-13,2)
 \qbezier(-16,0)(-13,0)(-13,-2)
 \put(0.8,-1.3){$y_1$}
 \multiput(-5,0)(0.4,0){25}{\line(1,0){0.2}}
 \qbezier(0,0)(3,3)(4,2)
 \qbezier(0,0)(3,3)(2,4)
 \multiput(17,-4)(0.4,0){10}{\line(1,0){0.2}}
 \multiput(17,-2)(0.4,0){10}{\line(1,0){0.2}}
 \multiput(18,-5)(0,0.4){10}{\line(0,1){0.2}}
 \multiput(15,-3)(0.4,0){10}{\line(1,0){0.2}}
 \multiput(16,-4)(0,0.4){10}{\line(0,1){0.2}}
 \multiput(12,1)(0,0.4){10}{\line(0,1){0.2}}
 \qbezier(12,3)(12,4)(13,4)
 \qbezier(12,3)(12,2)(11,2)
 \put(14,1){$\cdots$}
 \put(7.5,1){$k$}
\put(-6,0){\vector(-1,0){4}}
 \put(10,0){\vector(-1,0){1.5}}
 \multiput(7.5,0)(0.4,0){3}{\line(1,0){0.2}}
 \put(7.5,0){\vector(-1,0){1.5}}
 \end{picture}
 \end{center}
 \end{minipage}

\end{tabular}
\end{center}
\end{figure}

\bigskip

\item 
\begin{table}[H]
\begin{center}
\begin{tabular}{c}

 \begin{minipage}{0.2\hsize}
 \begin{tabular}{c}
   \\ \cline{1-1}
 \multicolumn{1}{|c|}{$(x,3)$}\\ \cline{1-1}
 \multicolumn{1}{c}{\lower1ex\hbox{$\mathcal{D}^1$}}
 \end{tabular}
 \end{minipage}

 \begin{minipage}{0.2\hsize}
 \begin{tabular}{cc}
  & \\
 $(1,1,1)$ or &$(2,1)$ \\
 \multicolumn{2}{c}{\lower1ex\hbox{$\mathcal{D}_x^0$}}
 \end{tabular}
 \end{minipage}

\end{tabular}
\end{center}
\end{table}

\setlength\unitlength{0.28cm}
\begin{figure}[H]
\begin{center}
\begin{tabular}{c}

 \begin{minipage}{0.7\hsize}
 \begin{center}
 \begin{picture}(13,5)
 \put(-8.5,-1.3){$x$}
 \put(-13,0){\line(1,0){10}}
 \qbezier(-11,-3)(-8,0)(-5,3)
 \qbezier(-11,3)(-8,0)(-5,-3)
 \put(3,0){\line(1,0){10}}
 \qbezier(6,4)(8,4)(10,4)
 \qbezier(6,2)(8,2)(10,2)
 \put(2,0){\vector(-1,0){4}}
 \multiput(8,-5)(0,0.4){25}{\line(0,1){0.2}}
 \end{picture}
 \end{center}
 \end{minipage}

 \begin{minipage}{0.7\hsize}
 \begin{picture}(13,5)
 \put(-8.5,-1.3){$x$}
 \put(-13,0){\line(1,0){10}}
 \qbezier(-8,0)(-8,3)(-5,3)
 \qbezier(-8,0)(-8,3)(-11,3)
 \put(3,0){\line(1,0){10}}
 \qbezier(8,3)(8,1)(11,1)
 \qbezier(8,3)(8,5)(11,5)
 \put(2,0){\vector(-1,0){4}}
 \multiput(8,-5)(0,0.4){25}{\line(0,1){0.2}}
 \end{picture}
 \end{minipage}

\end{tabular}
\end{center}
\end{figure}

\bigskip

\item 
\begin{table}[H]
\begin{center}
\begin{tabular}{c}

 \begin{minipage}{0.2\hsize}
 \begin{tabular}{c}
  \\ \cline{1-1}
 \multicolumn{1}{|c|}{$(y,3)$} \\ \cline{1-1}
 \multicolumn{1}{|c|}{$(x,3)$} \\ \cline{1-1}
 \multicolumn{1}{c}{\lower1ex\hbox{$\mathcal{D}^1$}}
 \end{tabular}
 \end{minipage}

 \begin{minipage}{0.15\hsize}
 \begin{tabular}{c}
   \\ \cline{1-1}
 \multicolumn{1}{|c|}{$(y,3)$}\\ \cline{1-1}
 \multicolumn{1}{c}{\lower1ex\hbox{$\mathcal{D}_x^0$}}
 \end{tabular}
 \end{minipage}

 \begin{minipage}{0.2\hsize}
 \begin{tabular}{cc}
  & \\
 $(1,1,1)$ or &$(2,1)$ \\
 \multicolumn{2}{c}{\lower1ex\hbox{$\mathcal{D}_y^0$}}
 \end{tabular}
 \end{minipage}

\end{tabular}
\end{center}
\end{table}\hfill

\setlength\unitlength{0.28cm}
\begin{figure}[H]
\begin{center}
\begin{tabular}{c}

 \begin{minipage}{0\hsize}
 \begin{center}
 \begin{picture}(21,5)
 \put(-16.5,-1.3){$x$}
 \put(-21,0){\line(1,0){10}}
 \qbezier(-16,0)(-13,0)(-13,3)
 \qbezier(-16,0)(-19,0)(-19,-3)
 \qbezier(-16,0)(-19,0)(-19,3)
 \qbezier(-16,0)(-13,0)(-13,-3)
 \put(-0.5,-1.3){$y$}
 \put(-5,0){\line(1,0){10}}
 \qbezier(-3,-3)(0,0)(3,3)
 \qbezier(3,-3)(0,0)(-3,3)
 \multiput(0,-5)(0,0.4){25}{\line(0,1){0.2}}
 \put(11,0){\line(1,0){10}}
 \qbezier(14,1.3)(16,1.3)(18,1.3)
 \qbezier(14,2.6)(16,2.6)(18,2.6)
 \multiput(11,4)(0.4,0){25}{\line(1,0){0.2}}
 \multiput(16,-5)(0,0.4){25}{\line(0,1){0.2}}
 \put(-6,0){\vector(-1,0){4}}
 \put(10,0){\vector(-1,0){4}}
 \end{picture}
 \end{center}
 \end{minipage}

\end{tabular}
\end{center}
\end{figure}

\ \\

\setlength\unitlength{0.28cm}
\begin{figure}[H]
\begin{center}
\begin{tabular}{c}

 \begin{minipage}{0\hsize}
 \begin{picture}(21,5)
 \put(-16.5,-1.3){$x$}
 \put(-21,0){\line(1,0){10}}
 \qbezier(-16,0)(-13,0)(-13,3)
 \qbezier(-16,0)(-13,0)(-13,-3)
 \put(-0.5,-1.3){$y$}
 \put(-5,0){\line(1,0){10}}
 \qbezier(0,0)(3,3)(3,2)
 \qbezier(0,0)(3,3)(1,4)
 \multiput(0,-5)(0,0.4){25}{\line(0,1){0.2}}
 \put(11,0){\line(1,0){10}}
 \qbezier(16,2)(16,3.5)(18,3.5)
 \qbezier(16,2)(16,0.5)(18,0.5)
 \multiput(11,4)(0.4,0){25}{\line(1,0){0.2}}
 \multiput(16,-5)(0,0.4){25}{\line(0,1){0.2}}
 \put(-6,0){\vector(-1,0){4}}
 \put(10,0){\vector(-1,0){4}} 
 \end{picture}
 \end{minipage}

\end{tabular}
\end{center}
\end{figure}

\bigskip

\item
\begin{table}[H]
\begin{center}
\begin{tabular}{c}

 \begin{minipage}{0.18\hsize}
 \begin{tabular}{c}
  \\ \cline{1-1}
 \multicolumn{1}{|c|}{$(y,3)$} \\ \cline{1-1}
 \multicolumn{1}{|c|}{$(x,4)$} \\ \cline{1-1}
 \multicolumn{1}{c}{\lower1ex\hbox{$\mathcal{D}^1$}}
 \end{tabular}
 \end{minipage}

 \begin{minipage}{0.25\hsize}
 \begin{tabular}{cc}
   & \\ \cline{1-2}
 \multicolumn{1}{|c|}{$(y,3)$} &  \multicolumn{1}{|c|}{$(z,3)$} \\ \cline{1-2}
 \multicolumn{2}{c}{\lower1ex\hbox{$\mathcal{D}_x^1$}}
 \end{tabular}
 \end{minipage}

\begin{minipage}{0.15\hsize}
 \begin{tabular}{c}
  \\
 $(1,1,1)$ \\
 \multicolumn{1}{c}{\lower1ex\hbox{$\mathcal{D}_y^0$}}
 \end{tabular}
 \end{minipage}

 \begin{minipage}{0.25\hsize}
 \begin{tabular}{cc}
  & \\
 $(1,1,1)$ or &$(2,1)$ \\
 \multicolumn{2}{c}{\lower1ex\hbox{$\mathcal{D}_z^0$}}
 \end{tabular}
 \end{minipage}

\end{tabular}
\end{center}
\end{table}

\setlength\unitlength{0.28cm}
\begin{figure}[H]
\begin{center}
\begin{tabular}{c}

 \begin{minipage}{0\hsize}
 \begin{center}
 \begin{picture}(29,6)
 \put(-25.5,-1.3){$x$}
 \put(-29,0){\line(1,0){10}}
 \qbezier(-24,0)(-21,0)(-21,3)
 \qbezier(-24,0)(-27,0)(-27,3)
 \qbezier(-24,0)(-24,3)(-22,4)
 \qbezier(-24,0)(-24,3)(-26,4)
 \qbezier(-24,0)(-24,-3)(-22,-4)
 \qbezier(-24,0)(-24,-3)(-26,-4)
 \put(-7.5,-1.3){$y$}
 \put(-7.5,3.5){$z$}
 \put(-13,0){\line(1,0){10}}
 \put(-8,-5){\line(0,1){10}}
 \qbezier(-9.5,5.5)(-8,4)(-6.5,2.5)
 \qbezier(-9.5,2.5)(-8,4)(-6.5,5.5)
 \qbezier(-10,-1.75)(-8,0)(-6,1.75)
 \put(10.5,4.5){$z$}
 \put(3,0){\line(1,0){10}}
 \put(3,4){\line(1,0){10}}
 \qbezier(6,2)(8,2)(10,2)
 \qbezier(10,3)(11,4)(12,5)
 \qbezier(10,5)(11,4)(12,3)
 \multiput(8,-5)(0,0.4){25}{\line(0,1){0.2}} 
 \put(19,0){\line(1,0){10}}
 \put(19,3){\line(1,0){10}}
 \qbezier(22,1.5)(24,1.5)(26,1.5)
 \qbezier(26,4)(27,4)(28,4)
 \qbezier(26,5)(27,5)(28,5)
 \multiput(27,1)(0,0.4){15}{\line(0,1){0.2}}
 \multiput(24,-5)(0,0.4){25}{\line(0,1){0.2}}
 \put(-14,0){\vector(-1,0){4}}
 \put(2,0){\vector(-1,0){4}}
 \put(18,0){\vector(-1,0){4}}
 \end{picture}
 \end{center}
 \end{minipage}

\end{tabular}
\end{center}
\end{figure}

\ \\

\setlength\unitlength{0.28cm}
\begin{figure}[H]
\begin{center}
\begin{tabular}{c}

 \begin{minipage}{0\hsize}
 \begin{center}
 \begin{picture}(29,6)
 \put(-24.5,-1.3){$x$}
 \put(-29,0){\line(1,0){10}}
 \qbezier(-24,0)(-21,0)(-21,3)
 \qbezier(-24,0)(-27,0)(-27,3)
 \qbezier(-24,0)(-24,3)(-22,4)
 \qbezier(-24,0)(-24,3)(-26,4)
 \put(-7.5,-1.3){$y$}
 \put(-9,3.5){$z$}
 \put(-13,0){\line(1,0){10}}
 \put(-8,-5){\line(0,1){10}}
 \qbezier(-8,4)(-6,4)(-6,2.5)
 \qbezier(-8,4)(-6,4)(-6,5.5)
 \qbezier(-10,-1.75)(-8,0)(-6,1.75)
 \put(10.5,3){$z$}
 \put(3,0){\line(1,0){10}}
 \put(3,4){\line(1,0){10}}
 \qbezier(6,2)(8,2)(10,2)
 \qbezier(11,4)(11,5)(12,5)
 \qbezier(11,4)(11,5)(10,5)
 \multiput(8,-5)(0,0.4){25}{\line(0,1){0.2}} 
 \put(19,0){\line(1,0){10}}
 \put(19,3){\line(1,0){10}}
 \qbezier(22,1.5)(24,1.5)(26,1.5)
 \qbezier(27,4.5)(27,5.5)(28.5,5.5)
 \qbezier(27,4.5)(27,3.5)(28.5,3.5)
 \multiput(27,1)(0,0.4){15}{\line(0,1){0.2}}
 \multiput(24,-5)(0,0.4){25}{\line(0,1){0.2}}
 \put(-14,0){\vector(-1,0){4}}
 \put(2,0){\vector(-1,0){4}}
 \put(18,0){\vector(-1,0){4}}
 \end{picture}
 \end{center}
 \end{minipage}

\end{tabular}
\end{center}
\end{figure}

\bigskip

\item 
\begin{table}[H]
\begin{center}
\begin{tabular}{c}

 \begin{minipage}{0.2\hsize}
 \begin{tabular}{c}
  \\ \cline{1-1}
 \multicolumn{1}{|c|}{$(z,3)$} \\ \cline{1-1}
 \multicolumn{1}{|c|}{$(y,3)$} \\ \cline{1-1}
 \multicolumn{1}{|c|}{$(x,3)$} \\ \cline{1-1}
 \multicolumn{1}{c}{\lower1ex\hbox{$\mathcal{D}^1$}}
 \end{tabular}
 \end{minipage}

 \begin{minipage}{0.2\hsize}
 \begin{tabular}{c}
   \\
   \\ \cline{1-1}
 \multicolumn{1}{|c|}{$(y,3)$} \\ \cline{1-1}
 \multicolumn{1}{c}{\lower1ex\hbox{$\mathcal{D}_x^0$}}
 \end{tabular}
 \end{minipage}

\begin{minipage}{0.15\hsize}
 \begin{tabular}{c}
   \\
   \\ \cline{1-1}
 \multicolumn{1}{|c|}{$(z,3)$} \\ \cline{1-1}
 \multicolumn{1}{c}{\lower1ex\hbox{$\mathcal{D}_y^0$}}
 \end{tabular}
 \end{minipage}

 \begin{minipage}{0.2\hsize}
 \begin{tabular}{cc}
  & \\
  & \\
 $(1,1,1)$ or &$(2,1)$ \\
 \multicolumn{2}{c}{\lower1ex\hbox{$\mathcal{D}_z^0$}}
 \end{tabular}
 \end{minipage}

\end{tabular}
\end{center}
\end{table}

\setlength\unitlength{0.28cm}
\begin{figure}[H]
\begin{center}
\begin{tabular}{c}

 \begin{minipage}{0\hsize}
 \begin{center}
 \begin{picture}(29,6)
 \put(-24.5,-1.3){$x$}
 \put(-29,0){\line(1,0){10}}
 \qbezier(-24,0)(-21,0)(-21,2)
 \qbezier(-24,0)(-27,0)(-27,2)
 \qbezier(-24,0)(-21,0)(-21,-2)
 \qbezier(-24,0)(-27,0)(-27,-2)
 \put(-7.8,-1.3){$y$}
 \put(-13,0){\line(1,0){10}}
\multiput(-8,-5)(0,0.4){25}{\line(0,1){0.2}} 
 \qbezier(-8,0)(-5,0)(-5,3)
 \qbezier(-8,0)(-11,0)(-11,3)
 \qbezier(-8,0)(-5,0)(-5,-3)
 \qbezier(-8,0)(-11,0)(-11,-3)
 \put(8,-1.5){$z$}
\put(3,0){\line(1,0){10}}
\multiput(8,-5)(0,0.4){25}{\line(0,1){0.2}}
\multiput(3,4)(0.4,0){25}{\line(1,0){0.2}} 
 \qbezier(5,-3)(8,0)(11,3)
 \qbezier(5,3)(8,0)(11,-3)
 \put(19,0){\line(1,0){10}}
 \qbezier(22,1.3)(24,1.3)(26,1.3)
 \qbezier(22,2.6)(24,2.6)(26,2.6)
 \multiput(19,4)(0.4,0){25}{\line(1,0){0.2}}
 \multiput(24,-5)(0,0.4){25}{\line(0,1){0.2}}
  \multiput(27,1)(0,0.4){15}{\line(0,1){0.2}}
 \put(-14,0){\vector(-1,0){4}}
 \put(2,0){\vector(-1,0){4}}
 \put(18,0){\vector(-1,0){4}}
 \end{picture}
 \end{center}
 \end{minipage}

\end{tabular}
\end{center}
\end{figure}

\ \\

\setlength\unitlength{0.28cm}
\begin{figure}[H]
\begin{center}
\begin{tabular}{c}

 \begin{minipage}{0\hsize}
 \begin{center}
 \begin{picture}(29,6)
 \put(-24.5,-1.3){$x$}
 \put(-29,0){\line(1,0){10}}
 \qbezier(-24,0)(-21,0)(-21,3)
 \qbezier(-24,0)(-21,0)(-21,-3)
 \put(-7.8,-1.3){$y$}
 \put(-13,0){\line(1,0){10}}
\multiput(-8,-5)(0,0.4){25}{\line(0,1){0.2}} 
 \qbezier(-8,0)(-5,0)(-5,3)
 \qbezier(-8,0)(-5,0)(-5,-3)
 \put(8,-1.5){$z$}
\put(3,0){\line(1,0){10}}
\multiput(8,-5)(0,0.4){25}{\line(0,1){0.2}}
\multiput(3,4)(0.4,0){25}{\line(1,0){0.2}}
 \qbezier(8,0)(11,3)(11,2)
 \qbezier(8,0)(11,3)(9,3.5)
 \put(19,0){\line(1,0){10}}
 \qbezier(24,2)(24,3.5)(26,3.5)
 \qbezier(24,2)(24,0.5)(26,0.5)
 \multiput(19,4)(0.4,0){25}{\line(1,0){0.2}}
 \multiput(24,-5)(0,0.4){25}{\line(0,1){0.2}}
  \multiput(27,1)(0,0.4){15}{\line(0,1){0.2}}
 \put(-14,0){\vector(-1,0){4}}
 \put(2,0){\vector(-1,0){4}}
 \put(18,0){\vector(-1,0){4}}
 \end{picture}
 \end{center}
 \end{minipage}

\end{tabular}
\end{center}
\end{figure}

\bigskip

\item
\begin{table}[H]
\begin{center}
\begin{tabular}{c}

 \begin{minipage}{0.16\hsize}
 \begin{tabular}{c}
  \\ \cline{1-1}
 \multicolumn{1}{|c|}{$(z,3)$} \\ \cline{1-1}
 \multicolumn{1}{|c|}{$(y,4)$} \\ \cline{1-1}
 \multicolumn{1}{|c|}{$(x,4)$} \\ \cline{1-1}
 \multicolumn{1}{c}{\lower1ex\hbox{$\mathcal{D}^1$}}
 \end{tabular}
 \end{minipage}

 \begin{minipage}{0.16\hsize}
 \begin{tabular}{c}
   \\ \cline{1-1}
 \multicolumn{1}{|c|}{$(w,3)$} \\ \cline{1-1}
 \multicolumn{1}{|c|}{$(y,4)$} \\ \cline{1-1}
 \multicolumn{1}{c}{\lower1ex\hbox{$\mathcal{D}_x^1$}}
 \end{tabular}
 \end{minipage}

\begin{minipage}{0.25\hsize}
 \begin{tabular}{cc}
   & \\
   & \\ \cline{1-2}
 \multicolumn{1}{|c|}{$(z,3)$} & \multicolumn{1}{|c|}{$(w,3)$}\\ \cline{1-2}
 \multicolumn{2}{c}{\lower1ex\hbox{$\mathcal{D}_y^1$}}
 \end{tabular}
 \end{minipage}

 \begin{minipage}{0.15\hsize}
 \begin{tabular}{c}
  \\
  \\
 $(1,1,1)$ \\
 \multicolumn{1}{c}{\lower1ex\hbox{$\mathcal{D}_z^0$}}
 \end{tabular}
 \end{minipage}

 \begin{minipage}{0.15\hsize}
 \begin{tabular}{c}
  \\
  \\
 $(1,1,1)$ \\
 \multicolumn{1}{c}{\lower1ex\hbox{$\mathcal{D}_w^0$}}
 \end{tabular}
 \end{minipage}

\end{tabular}
\end{center}
\end{table}

\setlength\unitlength{0.21cm}
\begin{figure}[H]
\begin{center}
\begin{tabular}{c}

 \begin{minipage}{0\hsize}
 \begin{center}
 \begin{picture}(37,6)
 \put(-32.5,-1.3){$x$}
 \put(-37,0){\line(1,0){10}}
 \qbezier(-32,0)(-29,0)(-29,3)
 \qbezier(-32,0)(-29,0)(-29,-3)
 \qbezier(-32,0)(-29,0)(-29,2)
 \qbezier(-32,0)(-35,0)(-35,-2)
 \put(-17,-1.3){$y$}
 \put(-21,0){\line(1,0){10}}
 \put(-16,-5){\line(0,1){10}}
 \qbezier(-16,0)(-14,0)(-14,2)
 \qbezier(-16,0)(-18,0)(-18,2)
 \qbezier(-16,0)(-16,3)(-13,3)
 \qbezier(-16,0)(-16,-3)(-13,-3)
 \put(0.5,-1.5){$z$}
 \put(0.5,2.5){$w$}
 \put(-5,0){\line(1,0){10}}
 \put(0,-5){\line(0,1){10}} 
 \put(-5,4){\line(1,0){10}}
 \qbezier(-2,-2)(0,0)(2,2)
 \qbezier(-2,2)(0,4)(2,6)
 \put(20,3){$w$}
 \put(11,0){\line(1,0){10}}
 \multiput(16,-5)(0,0.4){25}{\line(0,1){0.2}}
 \put(11,4){\line(1,0){10}}
 \put(19.5,2){\line(0,1){4}}
 \put(14,2){\line(1,0){4}}
 \qbezier(18,2.5)(19.5,4)(21,5.5)
 \put(27,0){\line(1,0){10}}
 \multiput(32,-5)(0,0.4){25}{\line(0,1){0.2}}
 \put(27,3){\line(1,0){10}}
 \put(30,1.5){\line(1,0){4}}
 \put(33,5){\line(1,0){4}}
 \put(33.5,4){\line(1,0){3}}
 \multiput(35,0.5)(0,0.4){13}{\line(0,1){0.2}}
 \put(-22,0){\vector(-1,0){4}}
 \put(-6,0){\vector(-1,0){4}}
 \put(10,0){\vector(-1,0){4}}
 \put(26,0){\vector(-1,0){4}}
 \end{picture}
 \end{center}
 \end{minipage}

\end{tabular}
\end{center}
\end{figure}

\end{enumerate}

\bigskip

\ \\

\noindent For example, we consider the following sequence of singularity diagrams associated with $\Gamma_p$

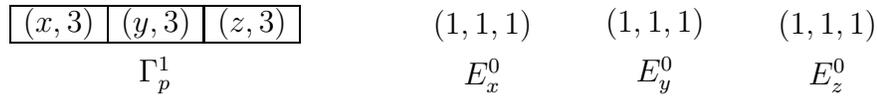
\begin{table}[H]
\begin{center}
\begin{tabular}{c}

 \begin{minipage}{0.32\hsize}
 \begin{tabular}{ccc}
   & & \\ \cline{1-3}
 \multicolumn{1}{|c|}{$(x,3)$} &  \multicolumn{1}{|c|}{$(y,3)$} & \multicolumn{1}{|c|}{$(z,3)$}\\ \cline{1-3}
 \multicolumn{3}{c}{\lower1ex\hbox{$\Gamma_p^1$}}
 \end{tabular}
 \end{minipage}

 \begin{minipage}{0.13\hsize}
 \begin{tabular}{c}
  \\
 $(1,1,1)$ \\
 \multicolumn{1}{c}{\lower1ex\hbox{$E_x^0$}}
 \end{tabular}
 \end{minipage}

 \begin{minipage}{0.13\hsize}
 \begin{tabular}{c}
  \\
 $(1,1,1)$ \\
 \multicolumn{1}{c}{\lower1ex\hbox{$E_y^0$}}
 \end{tabular}
 \end{minipage}

 \begin{minipage}{0.15\hsize}
 \begin{tabular}{c}
  \\
 $(1,1,1)$ \\
 \multicolumn{1}{c}{\lower1ex\hbox{$E_z^0$}}
 \end{tabular}
 \end{minipage}

\end{tabular}
\end{center}
\end{table}

After blowing-ups at $x$, $y$, $z$, the branch locus $\widetilde{R}$ near the fiber $\widetilde{\Gamma}_p$ of $\widetilde{\varphi}$ over $p$ 
is as follows.

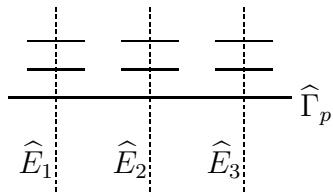
\setlength\unitlength{0.25cm}
\begin{figure}[H]
\begin{center}
\begin{tabular}{c}

 \begin{minipage}{0\hsize}
 \begin{center}
 \begin{picture}(7.5,5)
 \put(-7.5,0){\line(1,0){15}}
 \put(8,-1){$\widehat{\Gamma}_p$}
 \put(-7,-4){$\widehat{E}_1$}
 \put(-2,-4){$\widehat{E}_2$}
 \put(3,-4){$\widehat{E}_3$}
 \multiput(0,-5)(0,0.4){25}{\line(0,1){0.2}}
 \multiput(5,-5)(0,0.4){25}{\line(0,1){0.2}}
 \multiput(-5,-5)(0,0.4){25}{\line(0,1){0.2}}
 \put(-1.5,1.5){\line(1,0){3}}
 \put(-6.5,1.5){\line(1,0){3}}
 \put(3.5,1.5){\line(1,0){3}}
 \put(-1.5,3){\line(1,0){3}}
 \put(-6.5,3){\line(1,0){3}}
 \put(3.5,3){\line(1,0){3}}
 \end{picture}
 \end{center}
 \end{minipage}

\end{tabular}
\end{center}
\end{figure}

\ \\
\ \\

\noindent where $\widehat{E}_k$ denotes the proper transform of $E_k$ and $\widetilde{\Gamma}_p=\widehat{\Gamma}_p+\widehat{E}_1
+\widehat{E}_2+\widehat{E}_3$. Taking $3$-cyclic covering and contracting $(-1)$-curve in a fiber of $\widetilde{f}$, the fibers $\widetilde{F}_p$, $F_p$ of $\widetilde{f}$, $f$ over $p$ are as follows.

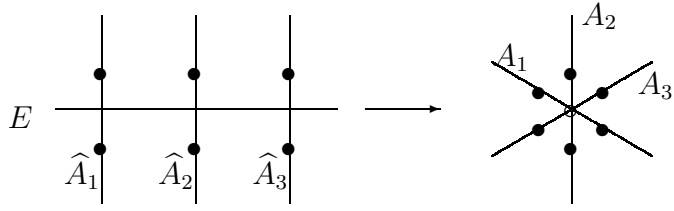
\setlength\unitlength{0.25cm}
\begin{figure}[H]
\begin{center}
\begin{tabular}{c}

 \begin{minipage}{0\hsize}
 \begin{center}
 \begin{picture}(15.5,5)
 \put(-20,-1){$E$}
 \put(-17,-4){$\widehat{A}_1$}
 \put(-12,-4){$\widehat{A}_2$}
 \put(-7,-4){$\widehat{A}_3$}
 \put(-10,-5){\line(0,1){10}}
 \put(-5,-5){\line(0,1){10}}
 \put(-15,-5){\line(0,1){10}}
 \put(-10.5075,1.45){$\bullet$}
 \put(-10.5075,-2.55){$\bullet$}
 \put(-15.5075,1.45){$\bullet$}
 \put(-15.5075,-2.55){$\bullet$}
 \put(-5.5075,1.45){$\bullet$}
 \put(-5.5075,-2.55){$\bullet$}
 \put(-17.5,0){\line(1,0){15}}
 
\put(-1,0){\vector(1,0){4}}

 \put(5.8,2.3){$A_1$}
 \put(10.5,4.5){$A_2$}
 \put(13.5,0.8){$A_3$}
 \put(10,-5){\line(0,1){10}}
 \qbezier(5.75,-2.5)(10,0)(14.25,2.5)
 \qbezier(5.75,2.5)(10,0)(14.25,-2.5)
 \put(9.4925,-0.55){$\circ$}
 \put(9.4925,1.45){$\bullet$}
 \put(11.1925,0.45){$\bullet$}
 \put(11.1925,-1.55){$\bullet$}
 \put(7.7925,0.45){$\bullet$}
 \put(7.7925,-1.55){$\bullet$}
 \put(9.4925,-2.55){$\bullet$}
   \end{picture}
  \end{center}
 \end{minipage}

\end{tabular}
\end{center}
\end{figure}

\ \\
\ \\

\noindent
where $\widehat{A}_k=\widetilde{\theta}^*\widehat{E}_k$, $3E=\widetilde{\theta}^*\widehat{\Gamma}_p$, $A_k=\rho(\widehat{A}_k)$,
 $\widetilde{F}_p=3E+\widehat{A}_1+\widehat{A}_2+\widehat{A}_3$, $F_p=A_1+A_2+A_3$, the symbol $\bullet$ and $\circ$ respectively denote a point on a $1$-dimensional fixed component of the automorphisms $\sigma$, $\widetilde{\sigma}$ of $f$, $\widetilde{f}$ and an isolated fixed point of $\sigma$.
Similarly, we can determine the shape of $F_p$ from another admissible sequence of abstract singularity diagrams. Thus, we can classify all fibers of primitive cyclic covering fibrations of type $(4,0,3)$ as follows:

\begin{table}[H]
\begin{center}
\begin{tabular}{|c|c|c|c|}
\multicolumn{4}{c}{} \\ \hline
 & \multicolumn{3}{|c|}{$\Gamma_p\subset R$} \\ \hline
 \begin{minipage}{17mm}
 \begin{center}
 \ \\
 The\\
diagram\\
of $\Gamma_p$
 \end{center}
 \end{minipage}
 & \begin{minipage}{35mm}
 \begin{center}
 \begin{tabular}{ccc}
  & & \\ \cline{1-3}
 \multicolumn{1}{|l|}{3}& \multicolumn{1}{|l|}{3} & \multicolumn{1}{|l|}{3}\\ \cline{1-3}
 \end{tabular}
 \end{center}
 \end{minipage}
 & \begin{minipage}{35mm}
 \begin{center}
 \begin{tabular}{cc}
  & \\ \cline{1-1}
 \multicolumn{1}{|l|}{3}& \\ \cline{1-2}
 \multicolumn{1}{|l|}{3}& \multicolumn{1}{|l|}{3} \\ \cline{1-2}
 \end{tabular}
 \end{center}
 \end{minipage}
 &  \begin{minipage}{35mm}
 \begin{center}
 \begin{tabular}{cc}
  & \\ \cline{1-1}
 \multicolumn{1}{|l|}{3}& \\ \cline{1-2}
 \multicolumn{1}{|l|}{4}& \multicolumn{1}{|l|}{3} \\ \cline{1-2}
 \end{tabular}
 \end{center}
 \end{minipage}
\\
 & & & \\ \hline
 \begin{minipage}{17mm}
 \begin{center}
 \ \\
The type\\
of $F_p$
 \end{center}
 \end{minipage}
  & $({\rm I\hspace{-.1em}V}_k)$
  & $({\rm V}_{i,j})$
  & $({\rm V\hspace{-.1em}I\hspace{-.1em}I}_k)$ \\ \hline
\end{tabular}
\end{center}
\end{table}

\begin{table}[H]
\begin{center}
\begin{tabular}{|c|c|c|c|}
\multicolumn{4}{c}{} \\ \hline
 & \multicolumn{2}{|c|}{$\Gamma_p\subset R$} & \multicolumn{1}{|c|}{$\Gamma_p\not\subset R$ \& $c=0$} \\ \hline
  \begin{minipage}{14mm}
 \begin{center}
 \ \\
 The\\
diagram\\
of $\Gamma_p$
 \end{center}
 \end{minipage}
 & \begin{minipage}{35mm}
 \begin{center}
 \begin{tabular}{c}
  \\ \cline{1-1}
 \multicolumn{1}{|l|}{3} \\ \cline{1-1}
 \multicolumn{1}{|l|}{3} \\ \cline{1-1}
 \multicolumn{1}{|l|}{3} \\ \cline{1-1}
 \end{tabular}
 \end{center}
 \end{minipage}
 & \begin{minipage}{35mm}
 \begin{center}
 \begin{tabular}{c}
  \\ \cline{1-1}
 \multicolumn{1}{|l|}{3} \\ \cline{1-1}
 \multicolumn{1}{|l|}{4} \\ \cline{1-1}
 \multicolumn{1}{|l|}{4} \\ \cline{1-1}
 \end{tabular}
 \end{center}
 \end{minipage}
 & \begin{minipage}{35mm}
 \begin{center}
 \begin{tabular}{c}
 $(i_1,\ldots,i_m)$ etc.
 \end{tabular}
 \end{center}
 \end{minipage}
\\
 & & & \\ \hline
 \begin{minipage}{17mm}
 \begin{center}
 \ \\
The type\\
of $F_p$
 \end{center}
 \end{minipage}
  & $({\rm V\hspace{-.1em}I}_k)$
  & $({\rm V\hspace{-.1em}I\hspace{-.1em}I\hspace{-.1em}I})$
  & $(0_{i_1,\ldots,i_m})$ \\ \hline
\end{tabular}
\end{center}
\end{table}

\begin{table}[H]
\begin{center}
\begin{tabular}{|c|c|c|c|}
\multicolumn{4}{c}{} \\ \hline
 & \multicolumn{3}{|c|}{$\Gamma_p\not\subset R$ \& $c=1$} \\ \hline
  \begin{minipage}{17mm}
 \begin{center}
 \ \\
 The\\
diagram\\
of $\Gamma_p$
 \end{center}
 \end{minipage}
 & \begin{minipage}{35mm}
 \begin{center}
 \begin{tabular}{cc}
  & \\ \cline{1-1}
 \multicolumn{1}{|c|}{3}& $(1,1,1)$ etc.\\ \cline{1-1}
 \end{tabular}
 \end{center}
 \end{minipage}
 &  \begin{minipage}{35mm}
 \begin{center}
 \begin{tabular}{cc}
 $(1)$  & \\ \cline{1-1}
 \multicolumn{1}{|c|}{3}& $(1,1)$ etc.\\ \cline{1-1}
 \end{tabular}
 \end{center}
 \end{minipage}
 & \begin{minipage}{35mm}
 \begin{center}
 \begin{tabular}{cc}
 $(2)$  & \\ \cline{1-1}
 \multicolumn{1}{|c|}{3}& $(1)$ \\ \cline{1-1}
 \end{tabular}
 \end{center}
 \end{minipage}
\\
 & & & \\ \hline
 \begin{minipage}{17mm}
 \begin{center}
 \ \\
The type\\
of $F_p$
 \end{center}
 \end{minipage}
  & $({\rm I}_{i,j,l})$
  & $({\rm I\hspace{-.1em}I\hspace{-.1em}I}_{i,j})$
  & $({\rm I\hspace{-.1em}I\hspace{-.1em}I}_{1,j})$ \\ \hline
\end{tabular}
\end{center}
\end{table}

\begin{table}[H]
\begin{center}
\begin{tabular}{|c|c|c|c|}
\multicolumn{4}{c}{} \\ \hline
 & \multicolumn{1}{|c|}{$\Gamma_p\not\subset R$ \& $c=1$}
 & \multicolumn{2}{|c|}{$\Gamma_p\not\subset R$ \& $c=2$} \\ \hline
  \begin{minipage}{17mm}
 \begin{center}
 \ \\
 The\\
diagram\\
of $\Gamma_p$
 \end{center}
 \end{minipage}
 & \begin{minipage}{35mm}
 \begin{center}
 \begin{tabular}{c}
 $(3)$  \\ \cline{1-1}
 \multicolumn{1}{|c|}{3} \\ \cline{1-1}
 \end{tabular}
 \end{center}
 \end{minipage}
& \begin{minipage}{35mm}
 \begin{center}
 \begin{tabular}{cc}
  & \\ \cline{1-2}
 \multicolumn{1}{|l|}{3}& \multicolumn{1}{|l|}{3} \\ \cline{1-2}
 \end{tabular}
 \end{center}
 \end{minipage}
 & \begin{minipage}{35mm}
 \begin{center}
 \begin{tabular}{c}
  \\ \cline{1-1}
 \multicolumn{1}{|l|}{3} \\ \cline{1-1}
 \multicolumn{1}{|l|}{3} \\ \cline{1-1}
 \end{tabular}
 \end{center}
 \end{minipage}
\\
 & & & \\ \hline
 \begin{minipage}{17mm}
 \begin{center}
 \ \\
The type\\
of $F_p$
 \end{center}
 \end{minipage}
  & $({\rm I\hspace{-.1em}I\hspace{-.1em}I}_{2,j})$ 
  & $({\rm I}_{i,j,l})$
  & $({\rm I\hspace{-.1em}I}_{k,l})$
\\ \hline
\end{tabular}
\end{center}
\end{table}

By computing $\alpha_k(F_p)$, $\varepsilon(F_p)$, $\mathrm{Ind}(F_p)$ and $\sigma(F_p)$, we get the list in Theorem~\ref{classthm}.

\begin{cor} \label{403Horcor}
Let $f\colon S\to B$ be a primitive cyclic covering fibration of type $(4,0,3)$.
Then we have
$$
K_f^2=\frac{24}{7}\chi_f+\mathrm{Ind}
$$
and $\mathrm{Ind}$ is given by
\begin{align*}
{\rm Ind}=&\sum_{l\ge 0} \frac{3}{7}(l+1)\nu({\rm I}_{*,*,l})+\sum_{l\ge 0} \frac{3}{7}(l+2)\nu({\rm I\hspace{-.1em}I}_{*,l})+\frac{3}{7}\nu({\rm I\hspace{-.1em}I\hspace{-.1em}I}_{*,*})+\frac{16}{7}\nu({\rm I\hspace{-.1em}V}_{*}) \\
&+\frac{16}{7}\nu({\rm V}_{*,*})+\frac{16}{7}\nu({\rm V\hspace{-.1em}I}_{*,*})+\frac{26}{7}\nu({\rm V\hspace{-.1em}I\hspace{-.1em}I}_{*})+\frac{33}{7}\nu({\rm V\hspace{-.1em}I\hspace{-.1em}I\hspace{-.1em}I}) \\
\end{align*}
where $\nu(*)$ denotes the number of fibers of type $(*)$ and $\nu({\rm I_{*,*,l}}):=\sum_{i,j}\nu({\rm I}_{i,j,l})$, etc.
\end{cor}

Since we see that $\sigma(F_p)\le 0$ for any fiber $F_p$ from the list, we have the following:

\begin{cor} \label{403signcor}
The signature of a complex surface with a primitive cyclic covering fibration of type $(4,0,3)$ is not positive.
\end{cor}

\begin{exa}\normalfont \label{constexa}
We can construct a primitive cyclic covering fibration of type $(4,0,3)$ having one singular fiber of any type. 
Indeed, we construct a fibration $f$ with a multiple fiber $F_p$ of type $({\rm V\hspace{-.1em}I\hspace{-.1em}I\hspace{-.1em}I})$ as follows. 
Let $P:=\mathbb{P}^1\times\mathbb{P}^1$, $B:=\mathbb{P}^1$ and $\varphi:=pr_2\colon P\rightarrow B$. 
Let $h$ and $\Gamma$ respectively denote general fibers of $pr_1$ and $\varphi=pr_2$ and set $\mathfrak{d}:=2h+m\Gamma$. We fix $p\in B$ arbitrarily. For a sufficiently large $m$, we can take $R\in |3\mathfrak{d}|$ such that $\Gamma_p\subset R$, $R\setminus \Gamma_p$ is smooth, and the appearance of $R$ and $\Gamma_p$ is as follows (see $(10)$ in the classification (iii)).

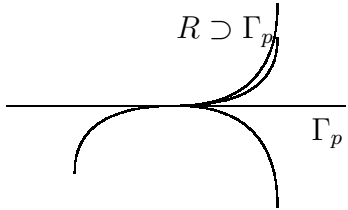
\setlength\unitlength{0.45cm}
\begin{figure}[H]
\begin{center}
\begin{tabular}{c}
 \begin{minipage}{0\hsize}
 \begin{center}
 \begin{picture}(0,3)
 \put(-5,0){\line(1,0){10}}
 \qbezier(0,0)(3,0)(3,3)
 \qbezier(0,0)(3,0)(3,-3)
 \qbezier(0,0)(3,0)(3,2)
 \qbezier(0,0)(-3,0)(-3,-2)
 \put(4,-1){$\Gamma_p$}
 \put(0,2){$R\supset \Gamma_p$}
 \end{picture}
 \end{center}
 \end{minipage}
\end{tabular}
\end{center}
\end{figure}
\ \\
\ \\
Let $\widetilde{\psi}=\psi_w\circ\psi_z\circ\psi_y\circ\psi_x \colon \widetilde{P}\rightarrow P$ be the composite of $4$ blow-ups at $x,y,z,w$ as follows.

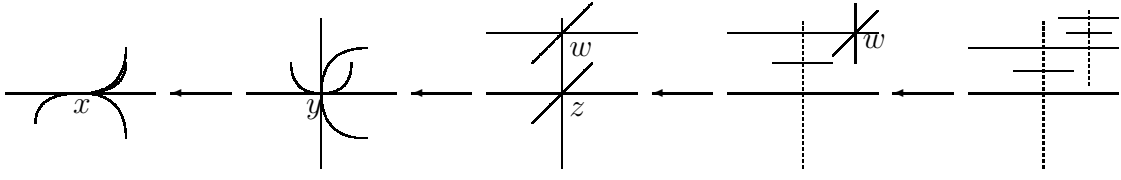
\setlength\unitlength{0.20cm}
\begin{figure}[H]
\begin{center}
\begin{tabular}{c}

 \begin{minipage}{0\hsize}
 \begin{center}
 \begin{picture}(37,6)
 \put(-32.5,-1.3){$x$}
 \put(-37,0){\line(1,0){10}}
 \qbezier(-32,0)(-29,0)(-29,3)
 \qbezier(-32,0)(-29,0)(-29,-3)
 \qbezier(-32,0)(-29,0)(-29,2)
 \qbezier(-32,0)(-35,0)(-35,-2)
 \put(-17,-1.3){$y$}
 \put(-21,0){\line(1,0){10}}
 \put(-16,-5){\line(0,1){10}}
 \qbezier(-16,0)(-14,0)(-14,2)
 \qbezier(-16,0)(-18,0)(-18,2)
 \qbezier(-16,0)(-16,3)(-13,3)
 \qbezier(-16,0)(-16,-3)(-13,-3)
 \put(0.5,-1.5){$z$}
 \put(0.5,2.5){$w$}
 \put(-5,0){\line(1,0){10}}
 \put(0,-5){\line(0,1){10}} 
 \put(-5,4){\line(1,0){10}}
 \qbezier(-2,-2)(0,0)(2,2)
 \qbezier(-2,2)(0,4)(2,6)
 \put(20,3){$w$}
 \put(11,0){\line(1,0){10}}
 \multiput(16,-5)(0,0.4){25}{\line(0,1){0.2}}
 \put(11,4){\line(1,0){10}}
 \put(19.5,2){\line(0,1){4}}
 \put(14,2){\line(1,0){4}}
 \qbezier(18,2.5)(19.5,4)(21,5.5)
 \put(27,0){\line(1,0){10}}
 \multiput(32,-5)(0,0.4){25}{\line(0,1){0.2}}
 \put(27,3){\line(1,0){10}}
 \put(30,1.5){\line(1,0){4}}
 \put(33,5){\line(1,0){4}}
 \put(33.5,4){\line(1,0){3}}
 \multiput(35,0.5)(0,0.4){13}{\line(0,1){0.2}}
 \put(-22,0){\vector(-1,0){4}}
 \put(-6,0){\vector(-1,0){4}}
 \put(10,0){\vector(-1,0){4}}
 \put(26,0){\vector(-1,0){4}}
 \end{picture}
 \end{center}
 \end{minipage}

\end{tabular}
\end{center}
\end{figure}

\ \\
\ \\
Then, the divisor $\widetilde{R}:=
\widetilde{\psi}^{\ast}R-3({\bf E}_x+{\bf E}_y
+{\bf E}_z+{\bf E}_w)$ is smooth and $3$-divisible in $\mathrm{Pic}(\widetilde{P})$, where ${\bf E}_{\bullet}$ denotes the total transform of the exceptional curve $E_{\bullet}$ for $\psi_{\bullet}$.
Then, we can take a $3$-cyclic covering $\widetilde{\theta}\colon \widetilde{S}\rightarrow \widetilde{P}$ branched over $\widetilde{R}$.
Let $f\colon S\rightarrow B$ be the relatively minimal model of  $\widetilde{f}:=\varphi\circ\widetilde{\psi}\circ\widetilde{\theta}\colon \widetilde{S}\rightarrow B$.
Then the fiber $F_p$ of $f$ over $p$ is clearly of type $({\rm V\hspace{-.1em}I\hspace{-.1em}I\hspace{-.1em}I})$. 
We can construct a fibration with another type of singular fiber in a similar way.
\end{exa}

\begin{exa}\normalfont
We can construct primitive cyclic covering fibrations of type $(g,0,3)$ with a triple fiber, except for $g=7$.
Note that all the possible genera $g$ are $3k+1$, $k\in \mathbb{Z}_{>0}$.

Firstly, we consider the following admissible sequence of abstract singularity diagrams for $k>0$.

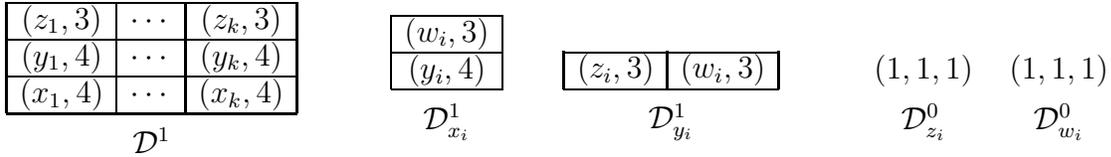
\begin{table}[H]
\begin{center}
\begin{tabular}{c}

 \begin{minipage}{0.3\hsize}
 \begin{tabular}{ccc}
 & & \\ \cline{1-3}                     
 \multicolumn{1}{|c|}{$(z_1,3)$}& $\cdots$ & \multicolumn{1}{|c|}{$(z_k,3)$}\\ \cline{1-3}
 \multicolumn{1}{|c|}{$(y_1,4)$}& $\cdots$ & \multicolumn{1}{|c|}{$(y_k,4)$}\\ \cline{1-3}
 \multicolumn{1}{|c|}{$(x_1,4)$}& $\cdots$ & \multicolumn{1}{|c|}{$(x_k,4)$}\\ \cline{1-3}
 \multicolumn{3}{c}{\lower1ex\hbox{$\mathcal{D}^1$}}
 \end{tabular}
 \end{minipage}

 \begin{minipage}{0.13\hsize}
 \begin{tabular}{c}
   \\ \cline{1-1}
 \multicolumn{1}{|c|}{$(w_i,3)$} \\ \cline{1-1}
 \multicolumn{1}{|c|}{$(y_i,4)$} \\ \cline{1-1}
 \multicolumn{1}{c}{\lower1ex\hbox{$\mathcal{D}_{x_i}^1$}}
 \end{tabular}
 \end{minipage}

\begin{minipage}{0.23\hsize}
 \begin{tabular}{cc}
   & \\
   & \\ \cline{1-2}
 \multicolumn{1}{|c|}{$(z_i,3)$} & \multicolumn{1}{|c|}{$(w_i,3)$}\\ \cline{1-2}
 \multicolumn{2}{c}{\lower1ex\hbox{$\mathcal{D}_{y_i}^1$}}
 \end{tabular}
 \end{minipage}

 \begin{minipage}{0.1\hsize}
 \begin{tabular}{c}
  \\
  \\
 $(1,1,1)$ \\
 \multicolumn{1}{c}{\lower1ex\hbox{$\mathcal{D}_{z_i}^0$}}
 \end{tabular}
 \end{minipage}

 \begin{minipage}{0.15\hsize}
 \begin{tabular}{c}
  \\
  \\
 $(1,1,1)$ \\
 \multicolumn{1}{c}{\lower1ex\hbox{$\mathcal{D}_{w_i}^0$}}
 \end{tabular}
 \end{minipage}

\end{tabular}
\end{center}
\end{table}

\noindent One can check that the sequence gives us a triple fiber $F_p$. 
By an argument similar to that in Example \ref{constexa}, we obtain a primitive cyclic covering fibration $f$ of type $(6k-2,0,3)$ which has such a triple fiber for $k>0$.

Next, we consider the following admissible sequence of abstract singularity diagrams for $k\ge 0$.

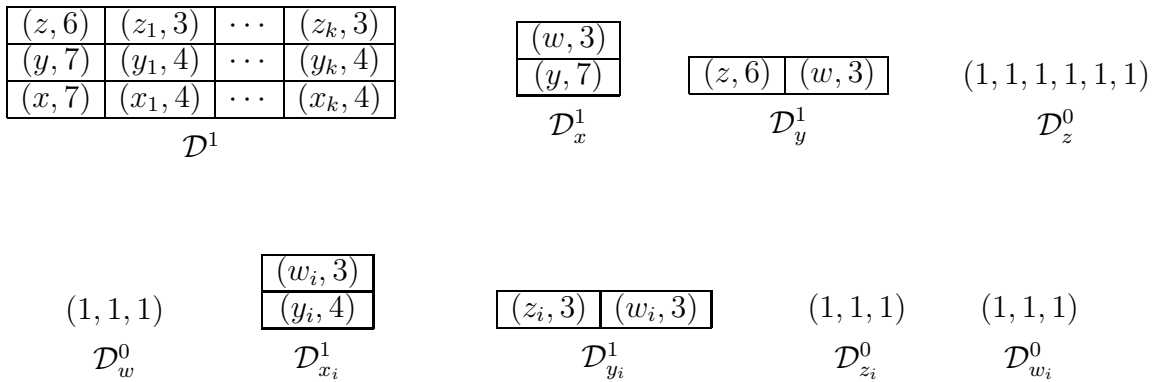
\begin{table}[H]
\begin{center}
\begin{tabular}{c}

 \begin{minipage}{0.4\hsize}
 \begin{tabular}{cccc}
 & & & \\ \cline{1-4}                     
 \multicolumn{1}{|c|}{$(z,6)$}& \multicolumn{1}{|c|}{$(z_1,3)$}& $\cdots$ & \multicolumn{1}{|c|}{$(z_k,3)$}\\ \cline{1-4}
 \multicolumn{1}{|c|}{$(y,7)$}& \multicolumn{1}{|c|}{$(y_1,4)$}& $\cdots$ & \multicolumn{1}{|c|}{$(y_k,4)$}\\ \cline{1-4}
 \multicolumn{1}{|c|}{$(x,7)$}& \multicolumn{1}{|c|}{$(x_1,4)$}& $\cdots$ & \multicolumn{1}{|c|}{$(x_k,4)$}\\ \cline{1-4}
 \multicolumn{4}{c}{\lower1ex\hbox{$\mathcal{D}^1$}}
 \end{tabular}
 \end{minipage}

 \begin{minipage}{0.13\hsize}
 \begin{tabular}{c}
   \\ \cline{1-1}
 \multicolumn{1}{|c|}{$(w,3)$} \\ \cline{1-1}
 \multicolumn{1}{|c|}{$(y,7)$} \\ \cline{1-1}
 \multicolumn{1}{c}{\lower1ex\hbox{$\mathcal{D}_{x}^1$}}
 \end{tabular}
 \end{minipage}

\begin{minipage}{0.2\hsize}
 \begin{tabular}{cc}
   & \\
   & \\ \cline{1-2}
 \multicolumn{1}{|c|}{$(z,6)$} & \multicolumn{1}{|c|}{$(w,3)$}\\ \cline{1-2}
 \multicolumn{2}{c}{\lower1ex\hbox{$\mathcal{D}_{y}^1$}}
 \end{tabular}
 \end{minipage}

 \begin{minipage}{0.2\hsize}
 \begin{tabular}{c}
  \\
  \\
 $(1,1,1,1,1,1)$ \\
 \multicolumn{1}{c}{\lower1ex\hbox{$\mathcal{D}_{z}^0$}}
 \end{tabular}
 \end{minipage}

\end{tabular}
\end{center}
\end{table}

\begin{table}[H]
\begin{center}
\begin{tabular}{c}

 \begin{minipage}{0.16\hsize}
 \begin{tabular}{c}
  \\
  \\
 $(1,1,1)$ \\
 \multicolumn{1}{c}{\lower1ex\hbox{$\mathcal{D}_{w}^0$}}
 \end{tabular}
 \end{minipage}

 \begin{minipage}{0.18\hsize}
 \begin{tabular}{c}
   \\ \cline{1-1}
 \multicolumn{1}{|c|}{$(w_i,3)$} \\ \cline{1-1}
 \multicolumn{1}{|c|}{$(y_i,4)$} \\ \cline{1-1}
 \multicolumn{1}{c}{\lower1ex\hbox{$\mathcal{D}_{x_i}^1$}}
 \end{tabular}
 \end{minipage}

\begin{minipage}{0.23\hsize}
 \begin{tabular}{cc}
   & \\
   & \\ \cline{1-2}
 \multicolumn{1}{|c|}{$(z_i,3)$} & \multicolumn{1}{|c|}{$(w_i,3)$}\\ \cline{1-2}
 \multicolumn{2}{c}{\lower1ex\hbox{$\mathcal{D}_{y_i}^1$}}
 \end{tabular}
 \end{minipage}

 \begin{minipage}{0.13\hsize}
 \begin{tabular}{c}
  \\
  \\
 $(1,1,1)$ \\
 \multicolumn{1}{c}{\lower1ex\hbox{$\mathcal{D}_{z_i}^0$}}
 \end{tabular}
 \end{minipage}

 \begin{minipage}{0.15\hsize}
 \begin{tabular}{c}
  \\
  \\
 $(1,1,1)$ \\
 \multicolumn{1}{c}{\lower1ex\hbox{$\mathcal{D}_{w_i}^0$}}
 \end{tabular}
 \end{minipage}

\end{tabular}
\end{center}
\end{table}

\noindent One can check that the sequence gives us a triple fiber $F_p$. 
Similarly as in Example \ref{constexa}, we obtain a primitive cyclic covering fibration $f$ of type $(6k+15,0,3)$ with such a triple fiber for $k\ge 0$.

\end{exa} 

\section{Fibers of hyperelliptic fibrations of genus $3$}

Let $f\colon S\rightarrow B$ be a hyperelliptic fibration of genus $g$, that is, a primitive cyclic covering fibration of type $(g,0,2)$.
We use freely the notation in the previous sections.
In order to classify fibers of hyperelliptic fibration of genus $g$, it is sufficient to classify admissible sequences of abstract singularity diagrams with $n=2$ and $t^1=r=2g+2$.
However, such sequences are too many to classify them all, and from the point of view of invariants, it seems that we do not have to find out all fibers explicitly, because there are singularities of $R$ which do not affect important invariants in the hyperelliptic case.
Indeed, singularities of multiplicity $2$ or $3$ contribute nothing to the Horikawa index when $n=2$.
If a singular point $x$ with multiplicity $2$ or $3$ has no singular points with multiplicity greater than $4$ infinitely near it, we call it a {\it negligible singularity}.
If a singular point $x$ is not a negligible singularity, we call it an {\it essential singularity}.
We decompose $\widetilde{\psi}\colon \widetilde{W}\to W$ into 
$\overline{\psi}\colon \widetilde{W}\to \widehat{W}$, the composite of blow-ups at negligible singularities and $\widehat{\psi}\colon \widehat{W}\to W$, the composite of blow-ups at essential singularities.
We call $\widehat{\psi}\colon \widehat{W}\to W$ the {\em even resolution of essential singularities}.

In this section, we first introduce abstract essential singularities and admissible sequences of them in order to classify fibers of hyperelliptic fibrations of genus $g$ according to the Horikawa index.
Next, we classify all fibers of hyperelliptic fibrations of genus $3$ into $12$ types and compute the Horikawa index for any types by classifying admissible sequences of abstract essential singularities.

\subsection{Abstract essential singularity diagrams}

\begin{defn}
Let $\mathcal{D}$ be an abstract singularity diagram with $n=2$ and $(x_{i,j},m_{i,j})$ an entry of $\mathcal{D}$. 
Then $(x_{i,j},m_{i,j})$ is {\em strictly negligible} if one of the following holds:

\smallskip

\noindent
(i) $m_{i,j}=2$.

\smallskip

\noindent
(ii) $m_{i,j}=3$, $j<i_{\mathrm{bm}}$ and $m_{i,j+1}\neq 4$.

\smallskip

\noindent
(iii) $m_{i,j}=3$, $j=i_{\mathrm{bm}}$ and $\mathcal{D}$ is of type $1$.

\smallskip

\noindent
(iv) $m_{i,j}=3$, $j=i_{\mathrm{bm}}>1$ and $m_{i,j-1}=3$.

\end{defn}

\begin{rem}
If $\mathcal{D}$ is a singularity diagram and $(x_{i,j},m_{i,j})$ is strictly negligible, then $x_{i,j}$ is a negligible singularity. However, the inverse is not true.
\end{rem}

\begin{defn}
Let $t\ge 2$ be an integer. We define an {\em abstract essential singularity diagram} $\mathcal{D}$ for $t$ to be an abbreviation of an abstract singularity diagram for $(2,t)$ by the following rule:

\smallskip

\noindent
(i) We denote a strictly negligible entry $(x_{i,j},m_{i,j})$ by $(x_{i,j},{\rm I\hspace{-.1em}I})$ if $m_{i,j}=2$, or $(x_{i,j},{\rm I\hspace{-.1em}I\hspace{-.1em}I})$ if $m_{i,j}=3$.

\smallskip

\noindent
(ii) We leave it blank for a strictly negligible entry.
\end{defn}

\begin{defn}
Let $\mathcal{D}_1,\mathcal{D}_2,\dots,\mathcal{D}_N$ be a sequence of abstract essential singularity diagrams. 
Let $I:=\{(\mu,i)|1\le \mu \le N,\; 1\le i\le l^\mu \; \text{and}\; i_{\mathrm{bm}}>0\}$ and $I_k:=\{(\mu,i)\in I|m_{i,1}^\mu\ge k\}$.
Then $\mathcal{D}_1,\mathcal{D}_2,\dots,\mathcal{D}_N$ is said to be {\em admissible} if there exist a subset $I_4\subset I_{\mathrm{ess}}\subset I_3$ and a bijection $\Phi\colon I_{\mathrm{ess}}\to \{2,\dots,N\}$ such that $\mu$ and $\nu:=\Phi(\mu,i)$ satisfy  $\mu<\Phi(\mu,i)$ and Lemma~\ref{connlem} $(1)$, $(2)$, $(3)$ and if $m_{i,1}^\mu=3$, then $\mathcal{D}_\nu$ has no strictly negligible entries for any $(\mu,i)\in I_{\mathrm{ess}}$.

\smallskip

\noindent
Let $\mathcal{D}_1,\mathcal{D}_2,\dots,\mathcal{D}_N$ be an admissible sequence of abstract essential singularity diagrams. Then, $(x_{i,1}^\mu,m_{i,1}^\mu)$ is said to be {\em negligible} (resp.\ {\em essential}\ ) if $(\mu,i)\not\in I_{\mathrm{ess}}$ (resp.\ $(\mu,i)\in I_{\mathrm{ess}}$). Clearly, a strictly negligible entry $(x_{i,1}^\mu,m_{i,1}^\mu)$ is negligible. We also denote a negligible entry $(x_{i,1}^\mu,3)$ by $(x_{i,1}^\mu,{\rm I\hspace{-.1em}I\hspace{-.1em}I})$.
\end{defn}

\begin{defn}
Two admissible sequences $\mathcal{D}_1$,\dots,$\mathcal{D}_N$ and $\mathcal{D}'_1$,\dots,$\mathcal{D}'_N$ of abstract essential singularity diagrams are {\em equivalent}, if there exists a bijection $\Psi\colon \{1,\dots,N\}\to \{1,\dots,N\}$ with $\Psi(1)=1$ such that the essential part of $\mathcal{D}_k$ is equivalent to that of $\mathcal{D}_{\Psi(k)}$ for any $1\le k\le N$, and $\mathcal{D}_\mu \overset{i}{\rightsquigarrow} \mathcal{D}_\nu$ if and only if $\mathcal{D}_{\Psi(\mu)} \overset{i}{\rightsquigarrow} \mathcal{D}_{\Psi(\nu)}$ for any $\mu<\nu$ and $1\le i\le l^\mu$ (after a suitable replacement of columns of $D_\mu$), where the essential part of $\mathcal{D}_k$ is the diagram which consists of essential entries of $\mathcal{D}_k$ only, and the definition of equivalence of essential parts of abstract essential singularity diagrams is the same as that of abstract singularity diagrams.

\end{defn}

In order to classify all fibers of hyperelliptic fibrations with genus $g$ according to the Horikawa index, it suffices to classify admissible sequences of abstract essential singularity diagrams with $t^1=r=2g+2$ and $m_{i,1}^1\le g+1+k$ if $\mathcal{D}_1$ is of type $k$ for $k=0,1$.
We proceed with the following steps.

\begin{enumerate}

\item[(i)] Classify abstract essential singularity diagrams of type $k$ for $2g+2$ with $m_{i,1}\le g+1+k$ for $k=0,1$.

\item[(ii)] Classify admissible sequences of abstract essential singularity diagrams with $t^1=3,4,\dots,g+2$ and $\mathcal{D}_1$ is of type $k$ if $t^1\equiv k$ (mod $2$).

\item[(iii)] Classify admissible sequences of abstract essential singularity diagrams with $t^1=r=2g+2$ and $m_{i,1}^1\le g+1+k$ if $\mathcal{D}_1$ is of type $k$ for $k=0,1$.

\end{enumerate}

\subsection{Classification: genus $3$ case}
\ \\

\noindent (i) All essential parts of abstract essential singularity diagrams of type $0$ for $8$ with $m_{i,1}\le 4$ are as follows.

\begin{table}[H]
\begin{center}
\begin{tabular}{|c|c|c|}
\multicolumn{3}{c}{} \\ \hline
(0,a) &(0,b) & (0,c) \\ \hline
 \begin{minipage}{37mm}
 \begin{center}
 \ \\
$(\bullet)$, ${\rm I\hspace{-.1em}I}$, ${\rm I\hspace{-.1em}I\hspace{-.1em}I}$ \\
only
 \end{center}
 \end{minipage}

 & \begin{minipage}{37mm}
 \begin{center}
 \begin{tabular}{c}
  \\ \cline{1-1}
 \multicolumn{1}{|l|}{3}  \\ \cline{1-1}
 \end{tabular}
 \end{center}
 \end{minipage}

 & \begin{minipage}{37mm}
 \begin{center}
 \begin{tabular}{c}
  \\ \cline{1-1}
 \multicolumn{1}{|l|}{4}  \\ \cline{1-1}
 \end{tabular}
 \end{center}
 \end{minipage}

\\
 & & \\ \hline
\begin{minipage}{37mm}
\begin{center}

 \end{center}
 \end{minipage}
  &
 \begin{minipage}{37mm}
\begin{center}
$0$
 \end{center}
 \end{minipage}
  & 
\begin{minipage}{37mm}
\begin{center}
$0$, $\rm I\hspace{-.1em}I$
 \end{center}
 \end{minipage}
 \\ \hline
$(0)$ & $({\rm I\hspace{-.1em}I})$ & $({\rm I\hspace{-.1em}I})$  \\ \hline
\end{tabular}
\end{center}
\end{table}

\begin{table}[H]
\begin{center}
\begin{tabular}{|c|c|c|}
\multicolumn{3}{c}{} \\ \hline
 (0,d) & (0,e) & (0,f) \\ \hline
 \begin{minipage}{37mm}
 \begin{center}
 \begin{tabular}{c}
  \\ \cline{1-1}
\multicolumn{1}{|l|}{4} \\ \cline{1-1}
 \multicolumn{1}{|l|}{3} \\ \cline{1-1}
 \end{tabular}
 \end{center}
 \end{minipage}

 &  \begin{minipage}{37mm}
 \begin{center}
\begin{tabular}{c}
  \\ \cline{1-1}
\multicolumn{1}{|l|}{3} \\ \cline{1-1}
 \multicolumn{1}{|l|}{4} \\ \cline{1-1}
 \end{tabular}
 \end{center}
 \end{minipage}
 & \begin{minipage}{37mm}
 \begin{center}
 \begin{tabular}{c}
  \\ \cline{1-1}
\multicolumn{1}{|l|}{4} \\ \cline{1-1}
 \multicolumn{1}{|l|}{4} \\ \cline{1-1}
 \end{tabular}
 \end{center}
 \end{minipage}

\\
 & & \\ \hline
 \begin{minipage}{37mm}
\begin{center}
$0$, $\rm I\hspace{-.1em}I$
 \end{center}
 \end{minipage}
  & 
\begin{minipage}{37mm}
\begin{center}
 $0$
 \end{center}
 \end{minipage}
  & 
\begin{minipage}{37mm}
\begin{center}
$0$
 \end{center}
 \end{minipage}
 \\ \hline
$({\rm V\hspace{-.1em}I})$ & $({\rm V\hspace{-.1em}I\hspace{-.1em}I\hspace{-.1em}I})$ & $({\rm I\hspace{-.1em}V})$  \\ \hline
\end{tabular}
\end{center}
\end{table}

\begin{table}[H]
\begin{center}
\begin{tabular}{|c|c|c|}
\multicolumn{3}{c}{} \\ \hline
 (0,g) & (0,h) & (0,i)\\ \hline
 \begin{minipage}{37mm}
 \begin{center}
 \begin{tabular}{cc}
  & \\ \cline{1-2}
 \multicolumn{1}{|l|}{3}& \multicolumn{1}{|l|}{3}  \\ \cline{1-2}
 \end{tabular}
 \end{center}
 \end{minipage}

 & \begin{minipage}{37mm}
 \begin{center}
 \begin{tabular}{cc}
  & \\ \cline{1-2}
 \multicolumn{1}{|l|}{4}& \multicolumn{1}{|l|}{3}  \\ \cline{1-2}
 \end{tabular}
 \end{center}
 \end{minipage}

 &  \begin{minipage}{37mm}
 \begin{center}
\begin{tabular}{cc}
  & \\ \cline{1-2}
 \multicolumn{1}{|l|}{4}& \multicolumn{1}{|l|}{4}  \\ \cline{1-2}
 \end{tabular}
 \end{center}
 \end{minipage}

\\
 & & \\ \hline
\begin{minipage}{37mm}
\begin{center}
$(0,0)$
 \end{center}
 \end{minipage}
  & 
 \begin{minipage}{37mm}
\begin{center}
$(0,0)$
 \end{center}
 \end{minipage}
  & 
\begin{minipage}{37mm}
\begin{center}
$(0,0)$
 \end{center}
 \end{minipage}
 \\ \hline
$({\rm I\hspace{-.1em}I})$ & $({\rm I\hspace{-.1em}I})$ & $({\rm I\hspace{-.1em}I})$  \\ \hline
\end{tabular}
\end{center}
\end{table}

All essential parts of abstract essential singularity diagrams of type $1$ for $8$ with $m_{i,1}\le 5$ are as follows.

\begin{table}[H]
\begin{center}
\begin{tabular}{|c|c|c|c|}
\multicolumn{4}{c}{} \\ \hline
(1,a)& (1,b) & (1,c) & (1,d) \\ \hline
 \begin{minipage}{32mm}
 \begin{center}
 \ \\
${\rm I\hspace{-.1em}I}$, ${\rm I\hspace{-.1em}I\hspace{-.1em}I}$ \\
only
 \end{center}
 \end{minipage}

 & \begin{minipage}{32mm}
 \begin{center}
 \begin{tabular}{c}
  \\ \cline{1-1}
 \multicolumn{1}{|l|}{4} \\ \cline{1-1}
 \end{tabular}
 \end{center}
 \end{minipage}

 & \begin{minipage}{32mm}
 \begin{center}
  \begin{tabular}{c}
  \\ \cline{1-1}
 \multicolumn{1}{|l|}{5} \\ \cline{1-1}
 \end{tabular}
 \end{center}
 \end{minipage}

 & \begin{minipage}{32mm}
 \begin{center}
 \begin{tabular}{c}
  \\ \cline{1-1}
 \multicolumn{1}{|l|}{4} \\ \cline{1-1}
 \multicolumn{1}{|l|}{3} \\ \cline{1-1}
 \end{tabular}
 \end{center}
 \end{minipage}
\\
 & & & \\ \hline

  & $0$, ${\rm I\hspace{-.1em}I}$, ${\rm I\hspace{-.1em}I\hspace{-.1em}I}$
  & ${\rm I\hspace{-.1em}I}$, ${\rm I\hspace{-.1em}I\hspace{-.1em}I}$
  & $0$, ${\rm I\hspace{-.1em}I}$, ${\rm I\hspace{-.1em}I\hspace{-.1em}I}$ \\ \hline
$(0)$ & $({\rm I})$ & $({\rm I})$  & $({\rm I})$ \\ \hline
\end{tabular}
\end{center}
\end{table}

\begin{table}[H]
\begin{center}
\begin{tabular}{|c|c|c|c|}
\multicolumn{4}{c}{} \\ \hline
 (1,e) & (1,f) & (1,g) & (1,h) \\ \hline
 \begin{minipage}{32mm}
 \begin{center}
 \begin{tabular}{c}
  \\ \cline{1-1}
 \multicolumn{1}{|l|}{4} \\ \cline{1-1}
 \multicolumn{1}{|l|}{4} \\ \cline{1-1}
 \end{tabular}
 \end{center}
 \end{minipage}

& \begin{minipage}{32mm}
 \begin{center}
 \begin{tabular}{c}
  \\ \cline{1-1}
 \multicolumn{1}{|l|}{4} \\ \cline{1-1}
 \multicolumn{1}{|l|}{5} \\ \cline{1-1}
 \end{tabular}
 \end{center}
 \end{minipage}

  & \begin{minipage}{32mm}
 \begin{center}
 \begin{tabular}{c}
  \\ \cline{1-1}
 \multicolumn{1}{|l|}{5} \\ \cline{1-1}
 \multicolumn{1}{|l|}{5} \\ \cline{1-1}
 \end{tabular}
 \end{center}
 \end{minipage}

 & \begin{minipage}{32mm}
 \begin{center}
  \begin{tabular}{c}
  \\ \cline{1-1}
 \multicolumn{1}{|l|}{4} \\ \cline{1-1}
 \multicolumn{1}{|l|}{3} \\ \cline{1-1}
 \multicolumn{1}{|l|}{4} \\ \cline{1-1}
 \end{tabular}
 \end{center}
 \end{minipage}
\\
 & & & \\ \hline
     $0$, ${\rm I\hspace{-.1em}I}$, ${\rm I\hspace{-.1em}I\hspace{-.1em}I}$
  & $0$, ${\rm I\hspace{-.1em}I}$, ${\rm I\hspace{-.1em}I\hspace{-.1em}I}$
  & ${\rm I\hspace{-.1em}I}$, ${\rm I\hspace{-.1em}I\hspace{-.1em}I}$
  & $0$, ${\rm I\hspace{-.1em}I}$
 \\ \hline
$({\rm V})$ & $({\rm I\hspace{-.1em}X})$ & $({\rm V\hspace{-.1em}I\hspace{-.1em}I})$  & $({\rm X})$ \\ \hline
\end{tabular}
\end{center}
\end{table}

\begin{table}[H]
\begin{center}
\begin{tabular}{|c|c|c|c|}
\multicolumn{4}{c}{} \\ \hline
(1,i) & (1,j) & (1,k) & (1,l) \\ \hline
 \begin{minipage}{32mm}
 \begin{center}
 \begin{tabular}{c}
  \\ \cline{1-1}
 \multicolumn{1}{|l|}{4} \\ \cline{1-1}
 \multicolumn{1}{|l|}{3} \\ \cline{1-1}
 \multicolumn{1}{|l|}{4} \\ \cline{1-1}
 \multicolumn{1}{|l|}{3} \\ \cline{1-1}
 \end{tabular}
 \end{center}
 \end{minipage}

 &  \begin{minipage}{32mm}
 \begin{center}
 \begin{tabular}{c}
  \\ \cline{1-1}
 \multicolumn{1}{|l|}{6} \\ \cline{1-1}
 \multicolumn{1}{|l|}{5} \\ \cline{1-1}
 \end{tabular}
 \end{center}
 \end{minipage}

& \begin{minipage}{32mm}
 \begin{center}
\begin{tabular}{cc}
  & \\ \cline{1-2}
 \multicolumn{1}{|l|}{4} & \multicolumn{1}{|l|}{4}  \\ \cline{1-2}
 \end{tabular}
 \end{center}
 \end{minipage}
 & \begin{minipage}{32mm}
 \begin{center}
 \begin{tabular}{cc}
  & \\ \cline{1-2}
 \multicolumn{1}{|l|}{5} & \multicolumn{1}{|l|}{4}  \\ \cline{1-2}
 \end{tabular}
 \end{center}
 \end{minipage}
\\
 & & & \\ \hline
     $0$
  & $0$
  & 
\begin{minipage}{32mm}
\begin{center}
$(0,0)$, $({\rm I\hspace{-.1em}I},0)$, \\
$({\rm I\hspace{-.1em}I\hspace{-.1em}I},0)$, 
$({\rm I\hspace{-.1em}I}, {\rm I\hspace{-.1em}I})$ 
 \end{center}
 \end{minipage}
 &
\begin{minipage}{32mm}
\begin{center}
$({\rm I\hspace{-.1em}I},0)$, 
$({\rm I\hspace{-.1em}I\hspace{-.1em}I},0)$, \\
$({\rm I\hspace{-.1em}I}, {\rm I\hspace{-.1em}I})$ 
 \end{center}
 \end{minipage}
 \\ \hline
$({\rm X\hspace{-.1em}I})$ & $({\rm I\hspace{-.1em}I\hspace{-.1em}I})$ & $({\rm I})$  & $({\rm I})$ \\ \hline
\end{tabular}
\end{center}
\end{table}

\begin{table}[H]
\begin{center}
\begin{tabular}{|c|c|c|c|}
\multicolumn{4}{c}{} \\ \hline
(1,m) & (1,n) & (1,o) & (1,p) \\ \hline
 \begin{minipage}{32mm}
 \begin{center}
  \begin{tabular}{cc}
  & \\ \cline{1-2}
 \multicolumn{1}{|l|}{5} & \multicolumn{1}{|l|}{5}  \\ \cline{1-2}
 \end{tabular}
 \end{center}
 \end{minipage}

 & \begin{minipage}{32mm}
 \begin{center}
 \begin{tabular}{cc}
  & \\ \cline{1-1}
\multicolumn{1}{|l|}{4} & \\ \cline{1-2}
 \multicolumn{1}{|l|}{3} & \multicolumn{1}{|l|}{4}  \\ \cline{1-2}
 \end{tabular}
 \end{center}
 \end{minipage}

 &  \begin{minipage}{32mm}
 \begin{center}
 \begin{tabular}{cc}
  & \\ \cline{1-1}
\multicolumn{1}{|l|}{4} & \\ \cline{1-2}
 \multicolumn{1}{|l|}{3} & \multicolumn{1}{|l|}{5}  \\ \cline{1-2}
 \end{tabular}
 \end{center}
 \end{minipage}

& \begin{minipage}{32mm}
 \begin{center}
\begin{tabular}{cc}
  & \\ \cline{1-2}
 \multicolumn{1}{|l|}{4} & \multicolumn{1}{|l|}{4} \\ \cline{1-2}
 \multicolumn{1}{|l|}{3} & \multicolumn{1}{|l|}{3}  \\ \cline{1-2}
 \end{tabular}
 \end{center}
 \end{minipage}
\\
 & & & \\ \hline
\begin{minipage}{32mm}
\begin{center}
$({\rm I\hspace{-.1em}I}, {\rm I\hspace{-.1em}I})$ 
 \end{center}
 \end{minipage}
  & 
 \begin{minipage}{32mm}
\begin{center}
$(0,0)$, $({\rm I\hspace{-.1em}I},0)$, \\
$(0, {\rm I\hspace{-.1em}I})$ 
 \end{center}
 \end{minipage}
  & 
\begin{minipage}{32mm}
\begin{center}
$(0, {\rm I\hspace{-.1em}I})$ 
 \end{center}
 \end{minipage}
  & 
\begin{minipage}{32mm}
\begin{center}
$(0,0)$
 \end{center}
 \end{minipage}
 \\ \hline
$({\rm I})$ & $({\rm I})$ & $({\rm I})$  & $({\rm I})$ \\ \hline
\end{tabular}
\end{center}
\end{table}

\noindent
where $0$, ${\rm I\hspace{-.1em}I}$, ${\rm I\hspace{-.1em}I\hspace{-.1em}I}$ in the second row from the bottom are all possible strictly negligible entries on the top of the column of the diagram and $0$ means no strictly negligible entries, and the entry $(\bullet)$ in the bottom row is the type of a fiber corresponding to the diagram which is defined as:

\begin{defn}\normalfont
Let $f\colon S\to B$ be a hyperelliptic fibered surface of genus $3$. 
We say a fiber $F_p$ of $f$ $($resp.\ $\widetilde{F}_p$ of $\widetilde{f}$, $\Gamma_p$ of $\varphi$, $\widetilde{\Gamma}_p$ of $\widetilde{\varphi})$ is a {\em fiber of type $(0)$} if the essential part of the singularity diagram of $\Gamma_p$ is (0,a) or (1,a), a {\em fiber of type $({\rm I})$} if that is (1,b), (1,c), (1,d), (1,k), (1,l), (1,m), (1,n), (1,o) or (1,p), a {\em fiber of type $({\rm I\hspace{-.1em}I})$} if that is (0,b), (0,c), (0,g), (0,h) or (0,i), a {\em fiber of type $({\rm I\hspace{-.1em}I\hspace{-.1em}I})$} if that is (1,j), a {\em fiber of type $({\rm I\hspace{-.1em}V})$} if that is (0,f), a {\em fiber of type $({\rm V})$} if that is (1,e), a {\em fiber of type $({\rm V\hspace{-.1em}I})$} if that is (0,d), a {\em fiber of type $({\rm V\hspace{-.1em}I\hspace{-.1em}I})$} if that is (1,g), a {\em fiber of type $({\rm V\hspace{-.1em}I\hspace{-.1em}I\hspace{-.1em}I})$} if that is (0,e), a {\em fiber of type $({\rm I\hspace{-.1em}X})$} if that is (1,f), a {\em
  fiber of type $({\rm X})$} if that is (1,h), a {\em fiber of type $({\rm X\hspace{-.1em}I})$} if that is (1,i).
\end{defn}

\noindent (ii)  All admissible sequences of abstract essential singularity diagrams with $t^1=3$ and $\mathcal{D}_1$ is of type $1$ (i.e. arising from an entry $(x_1,3)$) are as follows.

\begin{table}[H]
\begin{center}
\begin{tabular}{c}

 \begin{minipage}{0.2\hsize}
$\Bigl($
 \begin{tabular}{c}
   \\ \cline{1-1}
 \multicolumn{1}{|c|}{$(x_1,3)$}\\ \cline{1-1}
 \multicolumn{1}{c}{}
 \end{tabular}
$\Bigr)$
 \end{minipage}

\begin{minipage}{0.13\hsize}
 \begin{tabular}{c}
   \\ \cline{1-1}
 \multicolumn{1}{|c|}{$(y_1,4)$}\\ \cline{1-1}
 \multicolumn{1}{c}{\lower1ex\hbox{$\mathcal{D}_{x_1}^1$}}
 \end{tabular}
 \end{minipage}

\begin{minipage}{0.1\hsize}
 \begin{tabular}{c}
  $\cdots$
 \end{tabular}
 \end{minipage}

\begin{minipage}{0.2\hsize}
 \begin{tabular}{c}
   \\ \cline{1-1}
 \multicolumn{1}{|c|}{$(x_k,3)$}\\ \cline{1-1}
 \multicolumn{1}{c}{\lower1ex\hbox{$\mathcal{D}_{y_{k-1}}^0$}}
 \end{tabular}
 \end{minipage}

\begin{minipage}{0.2\hsize}
 \begin{tabular}{c}
   \\ \cline{1-1}
 \multicolumn{1}{|c|}{$(y_k,4)$}\\ \cline{1-1}
 \multicolumn{1}{c}{\lower1ex\hbox{$\mathcal{D}_{x_k}^1$}}
 \end{tabular}
 \end{minipage}

\end{tabular}
\end{center}
\end{table}

\noindent Then, $(x_1,3)$ is said to be of type $(3$-$4)^k$. An entry of type $(3$-$4)^0$ is nothing more than a negligible entry.

All admissible sequences of abstract essential singularity diagrams with $t^1=4$ and $\mathcal{D}_1$ being of type $0$ (i.e. arising from an entry $(x_1,4)$) are as follows.

\begin{table}[H]
\begin{center}
\begin{tabular}{c}

 \begin{minipage}{0.2\hsize}
 $\Bigl($
 \begin{tabular}{c}
   \\ \cline{1-1}
 \multicolumn{1}{|c|}{$(x_1,4)$}\\ \cline{1-1}
 \multicolumn{1}{c}{}
 \end{tabular}
 $\Bigr)$
 \end{minipage}

\begin{minipage}{0.15\hsize}
 \begin{tabular}{c}
   \\ \cline{1-1}
 \multicolumn{1}{|c|}{$(x_2,4)$}\\ \cline{1-1}
 \multicolumn{1}{c}{\lower1ex\hbox{$\mathcal{D}_{x_1}^0$}}
 \end{tabular}
 \end{minipage}

\begin{minipage}{0.1\hsize}
 \begin{tabular}{c}
  $\cdots$
 \end{tabular}
 \end{minipage}

\begin{minipage}{0.15\hsize}
 \begin{tabular}{c}
   \\ \cline{1-1}
 \multicolumn{1}{|c|}{$(x_k,4)$}\\ \cline{1-1}
 \multicolumn{1}{c}{\lower1ex\hbox{$\mathcal{D}_{x_{k-1}}^0$}}
 \end{tabular}
 \end{minipage}

\begin{minipage}{0.25\hsize}
 \begin{tabular}{c}
 \\
 $(x_{k+1},3)$ is of type $(3$-$4)^l$\\
\multicolumn{1}{c}{\lower1ex\hbox{$\mathcal{D}_{x_{k}}^1$}}
 \end{tabular}
 \end{minipage}

\end{tabular}
\end{center}
\end{table}

\noindent Then, $(x_1,4)$ is said to be of type $4^k$-$(3$-$4)^l$. 

All admissible sequences of abstract essential singularity diagrams with $t^1=5$ and $\mathcal{D}_1$ is of type $1$ (i.e. arising from an entry $(x_1,5)$) such that no entries with multiplicity $6$ appear are the following $3$ cases.

\begin{table}[H]
\begin{center}
\begin{tabular}{c}

 \begin{minipage}{0.2\hsize}
 $\Bigl($
 \begin{tabular}{c}
   \\ \cline{1-1}
 \multicolumn{1}{|c|}{$(x_1,5)$}\\ \cline{1-1}
 \multicolumn{1}{c}{}
 \end{tabular}
 $\Bigr)$
 \end{minipage}

\begin{minipage}{0.15\hsize}
 \begin{tabular}{c}
   \\ \cline{1-1}
 \multicolumn{1}{|c|}{$(x_2,5)$}\\ \cline{1-1}
 \multicolumn{1}{c}{\lower1ex\hbox{$\mathcal{D}_{x_1}^1$}}
 \end{tabular}
 \end{minipage}

\begin{minipage}{0.1\hsize}
 \begin{tabular}{c}
  $\cdots$
 \end{tabular}
 \end{minipage}

\begin{minipage}{0.2\hsize}
 \begin{tabular}{c}
   \\ \cline{1-1}
 \multicolumn{1}{|c|}{$(x_k,5)$}\\ \cline{1-1}
 \multicolumn{1}{c}{\lower1ex\hbox{$\mathcal{D}_{x_{k-1}}^1$}}
 \end{tabular}
 \end{minipage}

\end{tabular}
\end{center}
\end{table}

\begin{table}[H]
\begin{center}
\begin{tabular}{c}

 \begin{minipage}{0.18\hsize}
 $\Bigl($
 \begin{tabular}{c}
   \\ \cline{1-1}
 \multicolumn{1}{|c|}{$(x_1,5)$}\\ \cline{1-1}
 \multicolumn{1}{c}{}
 \end{tabular}
 $\Bigr)$
 \end{minipage}

\begin{minipage}{0.11\hsize}
 \begin{tabular}{c}
   \\ \cline{1-1}
 \multicolumn{1}{|c|}{$(x_2,5)$}\\ \cline{1-1}
 \multicolumn{1}{c}{\lower1ex\hbox{$\mathcal{D}_{x_1}^1$}}
 \end{tabular}
 \end{minipage}

\begin{minipage}{0.07\hsize}
 \begin{tabular}{c}
  $\cdots$
 \end{tabular}
 \end{minipage}

\begin{minipage}{0.13\hsize}
 \begin{tabular}{c}
   \\ \cline{1-1}
 \multicolumn{1}{|c|}{$(x_k,5)$}\\ \cline{1-1}
 \multicolumn{1}{c}{\lower1ex\hbox{$\mathcal{D}_{x_{k-1}}^1$}}
 \end{tabular}
 \end{minipage}

\begin{minipage}{0.14\hsize}
 \begin{tabular}{c}
   \\ \cline{1-1}
 \multicolumn{1}{|c|}{$(x_{k+1},4)$}\\ \cline{1-1}
 \multicolumn{1}{c}{\lower1ex\hbox{$\mathcal{D}_{x_k}^1$}}
 \end{tabular}
 \end{minipage}

\begin{minipage}{0.3\hsize}
 \begin{tabular}{c}
 \\
 $(x_{k+2},3)$ is of type $(3$-$4)^{l}$\\
\multicolumn{1}{c}{\lower1ex\hbox{$\mathcal{D}_{x_{k+1}}^0$}}
 \end{tabular}
 \end{minipage}

\end{tabular}
\end{center}
\end{table}

\begin{table}[H]
\begin{center}
\begin{tabular}{c}

 \begin{minipage}{0.18\hsize}
 $\Bigl($
 \begin{tabular}{c}
   \\ \cline{1-1}
 \multicolumn{1}{|c|}{$(x_1,5)$}\\ \cline{1-1}
 \multicolumn{1}{c}{}
 \end{tabular}
 $\Bigr)$
 \end{minipage}

\begin{minipage}{0.13\hsize}
 \begin{tabular}{c}
   \\ \cline{1-1}
 \multicolumn{1}{|c|}{$(x_2,5)$}\\ \cline{1-1}
 \multicolumn{1}{c}{\lower1ex\hbox{$\mathcal{D}_{x_1}^1$}}
 \end{tabular}
 \end{minipage}

\begin{minipage}{0.08\hsize}
 \begin{tabular}{c}
  $\cdots$
 \end{tabular}
 \end{minipage}

\begin{minipage}{0.18\hsize}
 \begin{tabular}{c}
   \\ \cline{1-1}
 \multicolumn{1}{|c|}{$(x_k,5)$}\\ \cline{1-1}
 \multicolumn{1}{c}{\lower1ex\hbox{$\mathcal{D}_{x_{k-1}}^1$}}
 \end{tabular}
 \end{minipage}

\begin{minipage}{0.2\hsize}
 \begin{tabular}{c}
 \\ \cline{1-1}
 \multicolumn{1}{|c|}{$(x_{k+2},4)$}\\ \cline{1-1}
 \multicolumn{1}{|c|}{$(x_{k+1},3)$}\\ \cline{1-1}
 \multicolumn{1}{c}{\lower1ex\hbox{$\mathcal{D}_{x_k}^1$}}
 \end{tabular}
 \end{minipage}

\end{tabular}
\end{center}
\end{table}

\noindent Then, $(x_1,5)$ is said to be of types $5^k$, $5^k$-$4$-$(3$-$4)^l$ and  $5^k$-$34$, respectively

We remark no singular points with multiplicity $6$ appear over $\Gamma_p$ unless it is of type $({\rm I\hspace{-.1em}I\hspace{-.1em}I})$.
All admissible sequences of abstract essential singularity diagrams with  $\mathcal{D}_1$ the diagram on a fiber $\Gamma_p$ of type $({\rm I\hspace{-.1em}I\hspace{-.1em}I})$ are as follows.

\begin{table}[H]
\begin{center}
\begin{tabular}{c}

 \begin{minipage}{0.2\hsize}
 \begin{tabular}{c}
   \\ \cline{1-1}
 \multicolumn{1}{|c|}{$(x_2,6)$}\\ \cline{1-1}
 \multicolumn{1}{|c|}{$(x_1,5)$}\\ \cline{1-1}
 \multicolumn{1}{c}{\lower1ex\hbox{$\mathcal{D}^1$}}
 \end{tabular}
 \end{minipage}

\begin{minipage}{0.18\hsize}
 \begin{tabular}{c}
   \\ \cline{1-1}
 \multicolumn{1}{|c|}{$(x_2,6)$}\\ \cline{1-1}
 \multicolumn{1}{c}{\lower1ex\hbox{$\mathcal{D}_{x_1}^1$}}
 \end{tabular}
 \end{minipage}

\begin{minipage}{0.4\hsize}
 \begin{tabular}{c}
   \\ 
 $(x_{3},4)$ or $(x_3,3)$ is of type $4^k$-$(3$-$4)^l$\\
 \multicolumn{1}{c}{\lower1ex\hbox{$\mathcal{D}_{x_2}^0$}}
 \end{tabular}
 \end{minipage}

\end{tabular}
\end{center}
\end{table}

\noindent (iii) Let $\Gamma_p$ be a fiber of type $({\rm I})$. 
First, if there exist singularities with multiplicity $5$, we blow up these singularities until there exist no singularities with multiplicity $5$ on the total transform of $\Gamma_p$. 
Let $i$ be the number of singularities of multiplicity $5$ (i.e.\ the number of these blow-ups). Then the dual graph of the total transform of $\Gamma_p$ is the Dynkin graph of type $A_{i+1}$, and all possible essential singularities on both ends of the chain are of types $4$-$(3$-$4)^{j}$ or $34$. 
We say $\Gamma_p$ is of type $({\rm I_{i,j,k}})$ more precisely if there exists a singularity of type $4$-$(3$-$4)^{j-1}$ on one end and there exists a singularity of type $4$-$(3$-$4)^{k-1}$ on another end, where $j$ or $k=0$ means there exists no essential singularity on the end. 
When there exists a singularity of type $34$ on one end, 
we denote $j$ or $k=\infty$.

Let $\Gamma_p$ be a fiber of type $({\rm I\hspace{-.1em}I})$. 
Similarly, if there exist singularities with multiplicity $4$, we blow up these singularities until there exist no singularities with multiplicity $4$ on the total transform of $\Gamma_p$. 
Let $i$ be the number of these blow-ups. Then the dual graph of the total transform of $\Gamma_p$ is also the Dynkin graph of type $A_{i+1}$, and all possible essential singularities on both ends of the chain are of type $(3$-$4)^{j}$. 
We say $\Gamma_p$ is of type $({\rm I\hspace{-.1em}I_{i,j,k}})$ if there exists a singularity of type $(3$-$4)^{j}$ on one end and there exists a singularity of type $(3$-$4)^{k}$ on another end.

Let $\Gamma_p$ be a fiber of type $({\rm I\hspace{-.1em}I\hspace{-.1em}I})$ (resp.\ a fiber of type $({\rm I\hspace{-.1em}V})$). 
We say $\Gamma_p$ is of type $({\rm I\hspace{-.1em}I\hspace{-.1em}I_{i,j}})$ (resp.\ $({\rm I\hspace{-.1em}V_{i,j}})$) if there exists a singularity of type $4^i$-$(3$-$4)^j$ on the exceptional curve for the blow-up at the singularity with multiplicity $6$ infinitely near to the singularity with multiplicity $5$ on $\Gamma_p$ (resp.\ the singularity with multiplicity $4$ infinitely near to the singularity with multiplicity $4$ on $\Gamma_p$). 

Let $\Gamma_p$ be a fiber of type $({\rm V})$ (resp.\ $({\rm V\hspace{-.1em}I})$). We say $\Gamma_p$ is of type $({\rm V_{j}})$ (resp.\ $({\rm V\hspace{-.1em}I_{j}})$) if there exists a singularity of type $(3$-$4)^j$ on the exceptional curve for the blow-up at the singularity with multiplicity $4$ infinitely near to the singularity with multiplicity $4$ on $\Gamma_p$ (resp.\ the singularity with multiplicity $4$ infinitely near to the singularity with multiplicity $3$ on $\Gamma_p$).

Let $\Gamma_p$ be a fiber of type $({\rm V\hspace{-.1em}I\hspace{-.1em}I})$ (resp.\ $({\rm V\hspace{-.1em}I\hspace{-.1em}I\hspace{-.1em}I})$). 
We say $\Gamma_p$ is of type $({\rm V\hspace{-.1em}I\hspace{-.1em}I_{j}})$ (resp.\ $({\rm V\hspace{-.1em}I\hspace{-.1em}I\hspace{-.1em}I_{j}})$) if there exists a singularity of type $4$-$(3$-$4)^{j-1}$ on the exceptional curve for the blow-up at the singularity with multiplicity $5$ infinitely near to the singularity with multiplicity $5$ on $\Gamma_p$ (resp.\ the singularity of multiplicity $3$ infinitely near to the singularity with multiplicity $4$ on $\Gamma_p$), where $j=0$ means there exists no essential singularity on the exceptional curve.

Now, the following lemma is straightforward.

\begin{lem}
Let $\widehat{\psi}\colon \widehat{W}\rightarrow W$ be the resolution of essential singularities and $\widehat{R}$ the branch locus on $\widehat{W}$. The total transform on $\widehat{W}$ of each fiber $\Gamma$ of types $({\rm I})$,\ldots, $({\rm X\hspace{-.1em}I})$ has the following configuration.

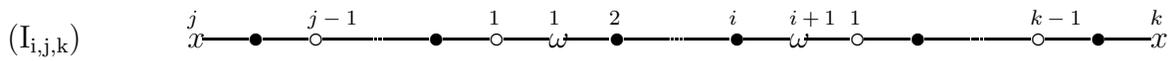
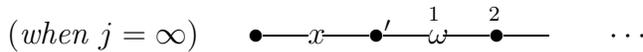
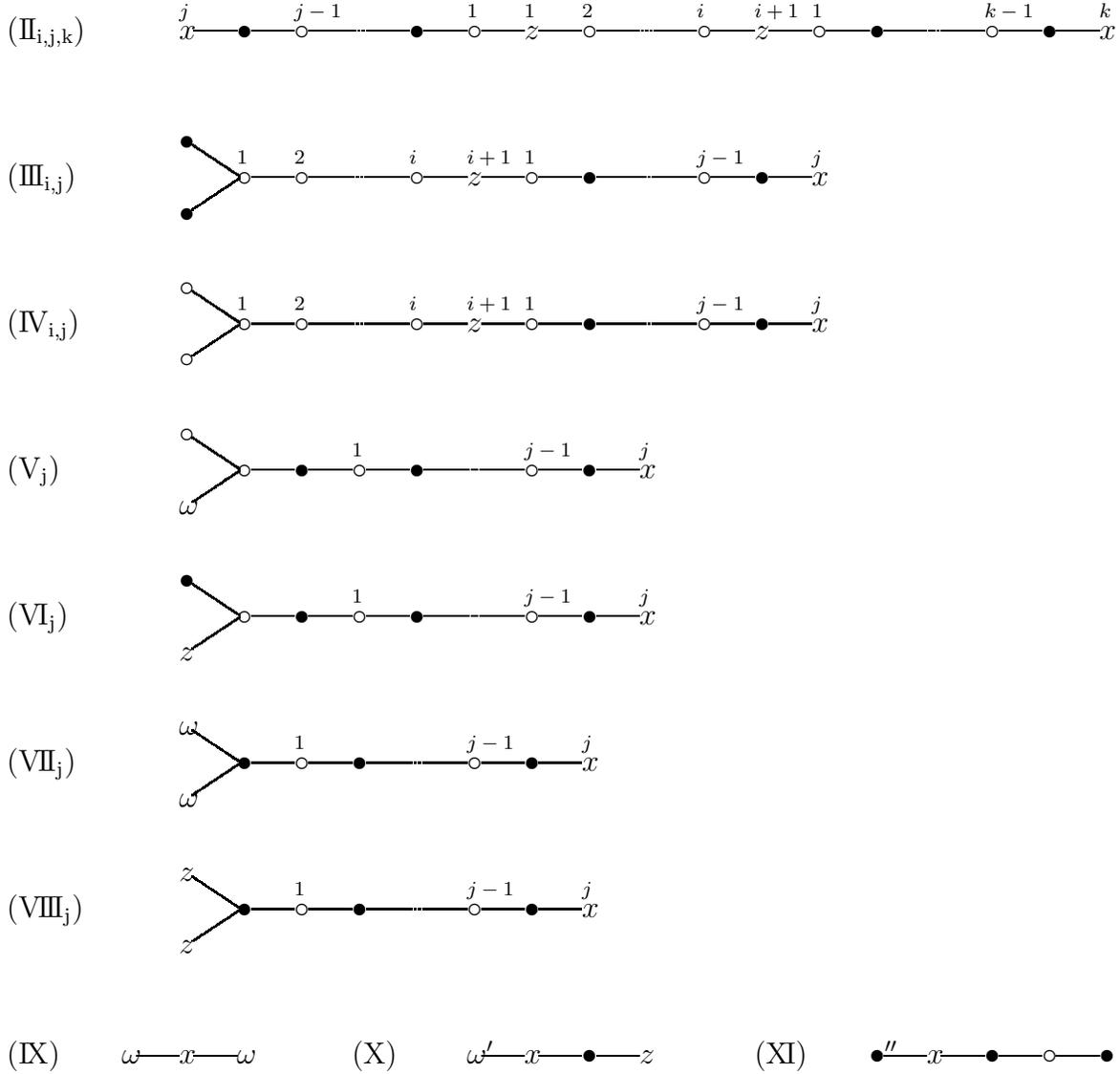
\setlength\unitlength{0.2cm}
\begin{figure}[H]
\begin{center}
\begin{tabular}{c}

 \begin{minipage}{0.6\hsize}
 \begin{center}
\begin{picture}(0,5)
 \multiput(-27,0.55)(4,0){16}{\line(1,0){3}}
 \multiput(-18,0.55)(0.3,0){20}{\line(1,0){0.15}}
 \multiput(2,0.55)(0.3,0){20}{\line(1,0){0.15}}
 \multiput(22,0.55)(0.3,0){20}{\line(1,0){0.15}}
 \put(-40,0){$({\rm I_{i,j,k}})$}
  \put(-28,1.5){{\tiny $j$}}
 \put(-28,0){$x$}
 \put(-24,0){$\bullet$}
  \put(-20,1.5){{\tiny $j-1$}}
 \put(-20,0){$\circ$}
 \put(-16,0){}
 \put(-12,0){$\bullet$}
  \put(-8,1.5){{\tiny $1$}}
 \put(-8,0){$\circ$}
  \put(-4,1.5){{\tiny $1$}}
 \put(-4,0){$\omega$}
  \put(0,1.5){{\tiny $2$}}
 \put(0,0){$\bullet$}
 \put(4,0){}
  \put(8,1.5){{\tiny $i$}}
 \put(8,0){$\bullet$}
  \put(12,1.5){{\tiny $i+1$}}
 \put(12,0){$\omega$}
  \put(16,1.5){{\tiny $1$}}
 \put(16,0){$\circ$}
 \put(20,0){$\bullet$}
 \put(24,0){}
  \put(28,1.5){{\tiny $k-1$}}
 \put(28,0){$\circ$}
 \put(32,0){$\bullet$}
  \put(36,1.5){{\tiny $k$}}
 \put(36,0){$x$}
 \end{picture}
 \end{center}
 \end{minipage}

\end{tabular}
\end{center}
\end{figure}

\setlength\unitlength{0.2cm}
\begin{figure}[H]
\begin{center}
\begin{tabular}{c}

 \begin{minipage}{0.6\hsize}
 \begin{center}
\begin{picture}(0,5)
 \multiput(-23,0.55)(4,0){5}{\line(1,0){3}}
 \put(-40,0){$($when $j=\infty$$)$}
 \put(-24,0){$\bullet$}
 \put(-20,0){$x$}
 \put(-16,0){$\bullet^{\prime}$}
  \put(-12,1.5){{\tiny $1$}}
 \put(-12,0){$\omega$}
  \put(-8,1.5){{\tiny $2$}}
 \put(-8,0){$\bullet$}
 \put(0,0){$\cdots$}
 
 \end{picture}
 \end{center}
 \end{minipage}

\end{tabular}
\end{center}
\end{figure}

\setlength\unitlength{0.2cm}
\begin{figure}[H]
\begin{center}
\begin{tabular}{c}

 \begin{minipage}{0.6\hsize}
 \begin{center}
\begin{picture}(0,5)
 \multiput(-27,0.55)(4,0){16}{\line(1,0){3}}
 \multiput(-18,0.55)(0.3,0){20}{\line(1,0){0.15}}
 \multiput(2,0.55)(0.3,0){20}{\line(1,0){0.15}}
 \multiput(22,0.55)(0.3,0){20}{\line(1,0){0.15}}
 \put(-40,0){$({\rm I\hspace{-.1em}I_{i,j,k}})$}
  \put(-28,1.5){{\tiny $j$}}
 \put(-28,0){$x$}
 \put(-24,0){$\bullet$}
  \put(-20,1.5){{\tiny $j-1$}}
 \put(-20,0){$\circ$}
 \put(-16,0){}
 \put(-12,0){$\bullet$}
  \put(-8,1.5){{\tiny $1$}}
 \put(-8,0){$\circ$}
  \put(-4,1.5){{\tiny $1$}}
 \put(-4,0){$z$}
  \put(0,1.5){{\tiny $2$}}
 \put(0,0){$\circ$}
 \put(4,0){}
  \put(8,1.5){{\tiny $i$}}
 \put(8,0){$\circ$}
  \put(12,1.5){{\tiny $i+1$}}
 \put(12,0){$z$}
  \put(16,1.5){{\tiny $1$}}
 \put(16,0){$\circ$}
 \put(20,0){$\bullet$}
 \put(24,0){}
  \put(28,1.5){{\tiny $k-1$}}
 \put(28,0){$\circ$}
 \put(32,0){$\bullet$}
  \put(36,1.5){{\tiny $k$}}
 \put(36,0){$x$}
 \end{picture}
 \end{center}
 \end{minipage}

\end{tabular}
\end{center}
\end{figure}

\setlength\unitlength{0.2cm}
\begin{figure}[H]
\begin{center}
\begin{tabular}{c}

 \begin{minipage}{0.6\hsize}
 \begin{center}
\begin{picture}(0,5)
 \multiput(-23,0.55)(4,0){10}{\line(1,0){3}}
 \multiput(-19,0.55)(0.3,0){20}{\line(1,0){0.15}}
 \multiput(1,0.55)(0.3,0){20}{\line(1,0){0.15}}
 \qbezier(-23.8,0.6)(-23.8,0.6)(-27.1,2.8)
 \qbezier(-23.8,0.5)(-23.8,0.5)(-27.1,-1.7)
 \put(-40,0){$({\rm I\hspace{-.1em}I\hspace{-.1em}I_{i,j}})$}
 \put(-28,-2.5){$\bullet$}
 \put(-28,2.5){$\bullet$}
  \put(-24,1.5){{\tiny $1$}}
 \put(-24,0){$\circ$}
  \put(-20,1.5){{\tiny $2$}}
 \put(-20,0){$\circ$}
 \put(-16,0){}
  \put(-12,1.5){{\tiny $i$}}
 \put(-12,0){$\circ$}
  \put(-8,1.5){{\tiny $i+1$}}
 \put(-8,0){$z$}
  \put(-4,1.5){{\tiny $1$}}
 \put(-4,0){$\circ$}
 \put(0,0){$\bullet$}
 \put(4,0){}
  \put(8,1.5){{\tiny $j-1$}}
 \put(8,0){$\circ$}
 \put(12,0){$\bullet$}
  \put(16,1.5){{\tiny $j$}}
 \put(16,0){$x$}
 \end{picture}
 \end{center}
 \end{minipage}

\end{tabular}
\end{center}
\end{figure}

\setlength\unitlength{0.2cm}
\begin{figure}[H]
\begin{center}
\begin{tabular}{c}

 \begin{minipage}{0.6\hsize}
 \begin{center}
\begin{picture}(0,5)
 \multiput(-23,0.55)(4,0){10}{\line(1,0){3}}
 \multiput(-19,0.55)(0.3,0){20}{\line(1,0){0.15}}
 \multiput(1,0.55)(0.3,0){20}{\line(1,0){0.15}}
 \qbezier(-23.8,0.6)(-23.8,0.6)(-27.1,2.8)
 \qbezier(-23.8,0.5)(-23.8,0.5)(-27.1,-1.7)
 \put(-40,0){$({\rm I\hspace{-.1em}V_{i,j}})$}
 \put(-28,-2.5){$\circ$}
 \put(-28,2.5){$\circ$}
  \put(-24,1.5){{\tiny $1$}}
 \put(-24,0){$\circ$}
  \put(-20,1.5){{\tiny $2$}}
 \put(-20,0){$\circ$}
 \put(-16,0){}
  \put(-12,1.5){{\tiny $i$}}
 \put(-12,0){$\circ$}
  \put(-8,1.5){{\tiny $i+1$}}
 \put(-8,0){$z$}
  \put(-4,1.5){{\tiny $1$}}
 \put(-4,0){$\circ$}
 \put(0,0){$\bullet$}
 \put(4,0){}
  \put(8,1.5){{\tiny $j-1$}}
 \put(8,0){$\circ$}
 \put(12,0){$\bullet$}
  \put(16,1.5){{\tiny $j$}}
 \put(16,0){$x$}
 \end{picture}
 \end{center}
 \end{minipage}

\end{tabular}
\end{center}
\end{figure}

\setlength\unitlength{0.2cm}
\begin{figure}[H]
\begin{center}
\begin{tabular}{c}

 \begin{minipage}{0.6\hsize}
 \begin{center}
\begin{picture}(0,5)
 \multiput(-23,0.55)(4,0){7}{\line(1,0){3}}
 \multiput(-11,0.55)(0.3,0){20}{\line(1,0){0.15}}
 \qbezier(-23.8,0.6)(-23.8,0.6)(-27.1,2.8)
 \qbezier(-23.8,0.5)(-23.8,0.5)(-27.1,-1.7)
 \put(-40,0){$({\rm V_{j}})$}
 \put(-28,-2.5){$\omega$}
 \put(-28,2.5){$\circ$}
 \put(-24,0){$\circ$}
 \put(-20,0){$\bullet$}
  \put(-16,1.5){{\tiny $1$}}
 \put(-16,0){$\circ$}
 \put(-12,0){$\bullet$}
 \put(-8,0){}
  \put(-4,1.5){{\tiny $j-1$}}
 \put(-4,0){$\circ$}
 \put(0,0){$\bullet$}
  \put(4,1.5){{\tiny $j$}}
 \put(4,0){$x$}
 \end{picture}
 \end{center}
 \end{minipage}

\end{tabular}
\end{center}
\end{figure}

\setlength\unitlength{0.2cm}
\begin{figure}[H]
\begin{center}
\begin{tabular}{c}

 \begin{minipage}{0.6\hsize}
 \begin{center}
\begin{picture}(0,5)
 \multiput(-23,0.55)(4,0){7}{\line(1,0){3}}
 \multiput(-11,0.55)(0.3,0){20}{\line(1,0){0.15}}
 \qbezier(-23.8,0.6)(-23.8,0.6)(-27.1,2.8)
 \qbezier(-23.8,0.5)(-23.8,0.5)(-27.1,-1.7)
 \put(-40,0){$({\rm V\hspace{-.1em}I_{j}})$}
 \put(-28,-2.5){$z$}
 \put(-28,2.5){$\bullet$}
 \put(-24,0){$\circ$}
 \put(-20,0){$\bullet$}
  \put(-16,1.5){{\tiny $1$}}
 \put(-16,0){$\circ$}
 \put(-12,0){$\bullet$}
 \put(-8,0){}
  \put(-4,1.5){{\tiny $j-1$}}
 \put(-4,0){$\circ$}
 \put(0,0){$\bullet$}
  \put(4,1.5){{\tiny $j$}}
 \put(4,0){$x$}
 \end{picture}
 \end{center}
 \end{minipage}

\end{tabular}
\end{center}
\end{figure}

\setlength\unitlength{0.2cm}
\begin{figure}[H]
\begin{center}
\begin{tabular}{c}

 \begin{minipage}{0.6\hsize}
 \begin{center}
\begin{picture}(0,5)
 \multiput(-23,0.55)(4,0){6}{\line(1,0){3}}
 \multiput(-15,0.55)(0.3,0){20}{\line(1,0){0.15}}
 \qbezier(-23.8,0.6)(-23.8,0.6)(-27.1,2.8)
 \qbezier(-23.8,0.5)(-23.8,0.5)(-27.1,-1.7)
 \put(-40,0){$({\rm V\hspace{-.1em}I\hspace{-.1em}I_{j}})$}
 \put(-28,-2.5){$\omega$}
 \put(-28,2.5){$\omega$}
 \put(-24,0){$\bullet$}
  \put(-20,1.5){{\tiny $1$}}
 \put(-20,0){$\circ$}
 \put(-16,0){$\bullet$}
 \put(-12,0){}
  \put(-8,1.5){{\tiny $j-1$}}
 \put(-8,0){$\circ$}
 \put(-4,0){$\bullet$}
  \put(0,1.5){{\tiny $j$}}
 \put(0,0){$x$}
 \end{picture}
 \end{center}
 \end{minipage}

\end{tabular}
\end{center}
\end{figure}

\setlength\unitlength{0.2cm}
\begin{figure}[H]
\begin{center}
\begin{tabular}{c}

 \begin{minipage}{0.6\hsize}
 \begin{center}
\begin{picture}(0,5)
 \multiput(-23,0.55)(4,0){6}{\line(1,0){3}}
 \multiput(-15,0.55)(0.3,0){20}{\line(1,0){0.15}}
 \qbezier(-23.8,0.6)(-23.8,0.6)(-27.1,2.8)
 \qbezier(-23.8,0.5)(-23.8,0.5)(-27.1,-1.7)
 \put(-40,0){$({\rm V\hspace{-.1em}I\hspace{-.1em}I\hspace{-.1em}I_{j}})$}
 \put(-28,-2.5){$z$}
 \put(-28,2.5){$z$}
 \put(-24,0){$\bullet$}
  \put(-20,1.5){{\tiny $1$}}
 \put(-20,0){$\circ$}
 \put(-16,0){$\bullet$}
 \put(-12,0){}
  \put(-8,1.5){{\tiny $j-1$}}
 \put(-8,0){$\circ$}
 \put(-4,0){$\bullet$}
  \put(0,1.5){{\tiny $j$}}
 \put(0,0){$x$}
 \end{picture}
 \end{center}
 \end{minipage}

\end{tabular}
\end{center}
\end{figure}

\setlength\unitlength{0.2cm}
\begin{figure}[H]
\begin{center}
\begin{tabular}{c}

 \begin{minipage}{0.6\hsize}
 \begin{center}
\begin{picture}(0,5)
 \multiput(-31,0.55)(4,0){2}{\line(1,0){3}}
 \put(-40,0){$({\rm I\hspace{-.1em}X})$}
 \put(-32,0){$\omega$}
 \put(-28,0){$x$}
 \put(-24,0){$\omega$}

 \multiput(-7,0.55)(4,0){3}{\line(1,0){3}}
 \put(-16,0){$({\rm X})$}
 \put(-8,0){$\omega^{\prime}$}
 \put(-4,0){$x$}
 \put(0,0){$\bullet$}
 \put(4,0){$z$}

 \multiput(21,0.55)(4,0){4}{\line(1,0){3}}
 \put(12,0){$({\rm X\hspace{-.1em}I})$}
 \put(20,0){$\bullet^{\prime\prime}$}
 \put(24,0){$x$}
 \put(28,0){$\bullet$}
 \put(32,0){$\circ$}
 \put(36,0){$\bullet$}
 \end{picture}
 \end{center}
 \end{minipage}

\end{tabular}
\end{center}
\end{figure}
\noindent where the symbol $x$ is a $(-1)$-curve, the symbol $\bullet$ is a $(-2)$-curve which is contained in $\widehat{R}$ and is disjoint from other components of $\widehat{R}$, the symbol $\circ$ is a $(-2)$-curve disjoint from $\widehat{R}$, the symbol $\omega$ is a $(-2)$-curve which is contained in $\widehat{R}$ and intersects with other components of $\widehat{R}$, the symbol $z$ is a $(-2)$-curve which is not contained in $\widehat{R}$ and intersects with $\widehat{R}$, the symbol $\bullet^{\prime}$ is a $(-3)$-curve which is contained in $\widehat{R}$ and is disjoint from other components of $\widehat{R}$, the symbol $\omega^{\prime}$ is a $(-3)$-curve which is contained in $\widehat{R}$ and intersects with other components of $\widehat{R}$ and the symbol $\bullet^{\prime\prime}$ is a $(-4)$-curve contained in $\widehat{R}$ which is contained in $\widehat{R}$ and is disjoint from other components of $\widehat{R}$. 
\end{lem}
Taking double covering, we get the following list:
\begin{thm}
Each fiber $F$ of $f$ of types $({\rm I})$,\ldots, $({\rm X\hspace{-.1em}I})$ has the following configuration.

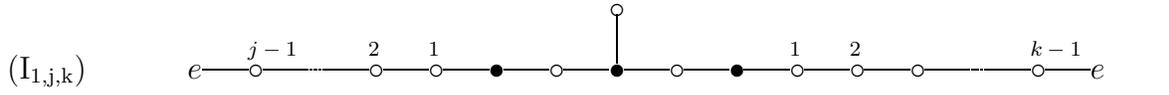
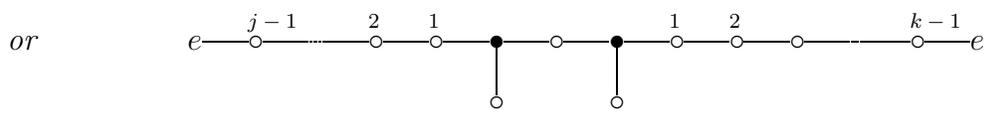
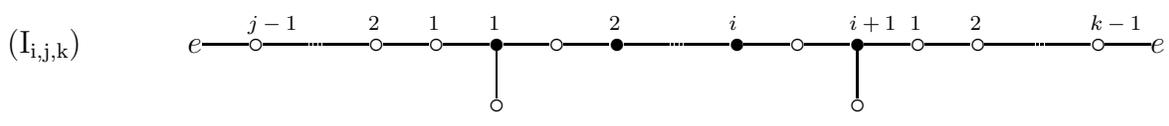
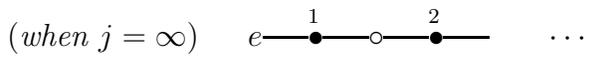
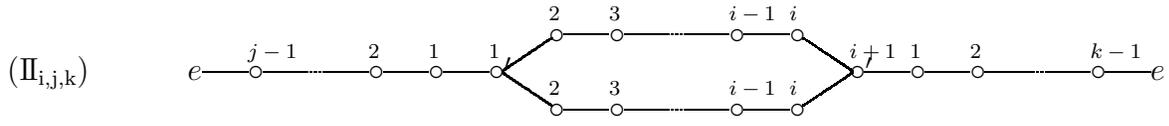
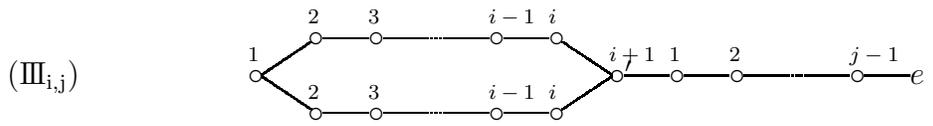
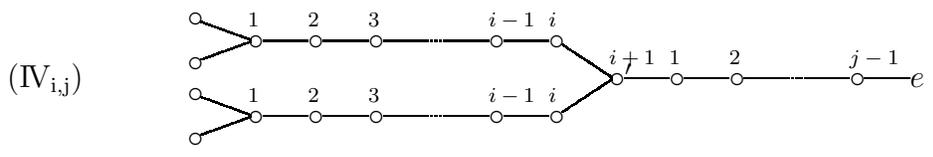
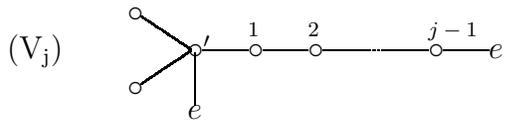
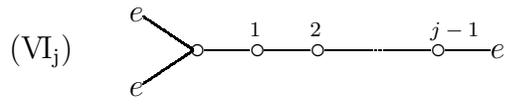
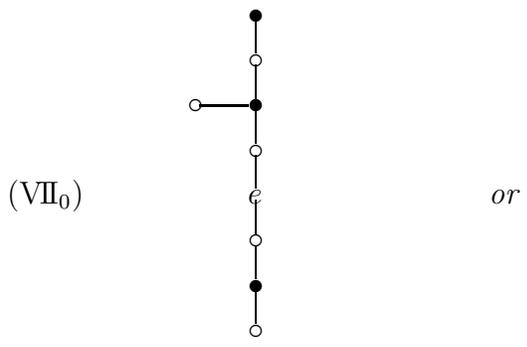
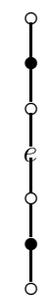
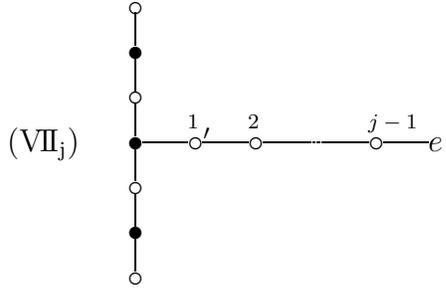
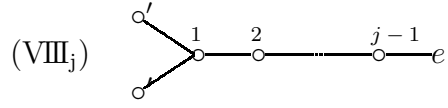
\setlength\unitlength{0.2cm}
\begin{figure}[H]
\begin{center}
\begin{tabular}{c}

 \begin{minipage}{0.6\hsize}
 \begin{center}
\begin{picture}(0,6)
 \multiput(-27,0.55)(4,0){15}{\line(1,0){3}}
 \multiput(-22,0.55)(0.3,0){20}{\line(1,0){0.15}}
 \multiput(22,0.55)(0.3,0){20}{\line(1,0){0.15}}
 \put(-40,0){$({\rm I_{1,j,k}})$}
 \put(-28,0){$e$}
  \put(-24,1.5){{\tiny $j-1$}}
 \put(-24,0){$\circ$}
 \put(-20,0){}
  \put(-16,1.5){{\tiny $2$}}
 \put(-16,0){$\circ$}
  \put(-12,1.5){{\tiny $1$}}
 \put(-12,0){$\circ$}
 \put(-8,0){$\bullet$}
 \put(0,4){$\circ$}
 \put(0.5,1){\line(0,1){3.2}}
 \put(-4,0){$\circ$}
 \put(0,0){$\bullet$}
 \put(4,0){$\circ$}
 \put(8,0){$\bullet$}
 \put(12,0){$\circ$}
  \put(12,1.5){{\tiny $1$}}
 \put(16,0){$\circ$}
  \put(16,1.5){{\tiny $2$}}
 \put(20,0){$\circ$}
 \put(24,0){}
  \put(28,1.5){{\tiny $k-1$}}
 \put(28,0){$\circ$}
 \put(32,0){$e$}
 \end{picture}
 \end{center}
 \end{minipage}

\end{tabular}
\end{center}
\end{figure}

\setlength\unitlength{0.2cm}
\begin{figure}[H]
\begin{center}
\begin{tabular}{c}

 \begin{minipage}{0.6\hsize}
 \begin{center}
\begin{picture}(0,6)
 \multiput(-27,0.55)(4,0){13}{\line(1,0){3}}
 \multiput(-22,0.55)(0.3,0){20}{\line(1,0){0.15}}
 \multiput(14,0.55)(0.3,0){20}{\line(1,0){0.15}}
 \put(-40,0){or}
 \put(-28,0){$e$}
  \put(-24,1.5){{\tiny $j-1$}}
 \put(-24,0){$\circ$}
 \put(-20,0){}
  \put(-16,1.5){{\tiny $2$}}
 \put(-16,0){$\circ$}
  \put(-12,1.5){{\tiny $1$}}
 \put(-12,0){$\circ$}
 \put(-8,0){$\bullet$}
 \put(-8,-4){$\circ$}
 \put(-7.5,-3){\line(0,1){3.2}}
 \put(-4,0){$\circ$}
 \put(0,0){$\bullet$}
 \put(0,-4){$\circ$}
 \put(0.5,-3){\line(0,1){3.2}}
 \put(4,0){$\circ$}
  \put(4,1.5){{\tiny $1$}}
 \put(8,0){$\circ$}
  \put(8,1.5){{\tiny $2$}}
 \put(12,0){$\circ$}
 \put(16,0){}
  \put(20,1.5){{\tiny $k-1$}}
 \put(20,0){$\circ$}
 \put(24,0){$e$}
 \end{picture}
 \end{center}
 \end{minipage}

\end{tabular}
\end{center}
\end{figure}

\setlength\unitlength{0.2cm}
\begin{figure}[H]
\begin{center}
\begin{tabular}{c}

 \begin{minipage}{0.6\hsize}
 \begin{center}
\begin{picture}(0,6)
 \multiput(-27,0.55)(4,0){16}{\line(1,0){3}}
 \multiput(-22,0.55)(0.3,0){20}{\line(1,0){0.15}}
 \multiput(2,0.55)(0.3,0){20}{\line(1,0){0.15}}
 \multiput(26,0.55)(0.3,0){20}{\line(1,0){0.15}}
 \put(-40,0){$({\rm I_{i,j,k}})$}
 \put(-28,0){$e$}
  \put(-24,1.5){{\tiny $j-1$}}
 \put(-24,0){$\circ$}
 \put(-20,0){}
  \put(-16,1.5){{\tiny $2$}}
 \put(-16,0){$\circ$}
  \put(-12,1.5){{\tiny $1$}}
 \put(-12,0){$\circ$}
  \put(-8,1.5){{\tiny $1$}}
 \put(-8,0){$\bullet$}
 \put(-8,-4){$\circ$}
 \put(-7.5,-3){\line(0,1){3.2}}
 \put(-4,0){$\circ$}
  \put(0,1.5){{\tiny $2$}}
 \put(0,0){$\bullet$}
 \put(4,0){}
  \put(8,1.5){{\tiny $i$}}
 \put(8,0){$\bullet$}
 \put(12,0){$\circ$}
  \put(16,1.5){{\tiny $i+1$}}
 \put(16,0){$\bullet$}
 \put(16,-4){$\circ$}
 \put(16.5,-3){\line(0,1){3.2}}
  \put(20,1.5){{\tiny $1$}}
 \put(20,0){$\circ$}
  \put(24,1.5){{\tiny $2$}}
 \put(24,0){$\circ$}
 \put(28,0){}
  \put(32,1.5){{\tiny $k-1$}}
 \put(32,0){$\circ$}
 \put(36,0){$e$}
 \end{picture}
 \end{center}
 \end{minipage}

\end{tabular}
\end{center}
\end{figure}

\setlength\unitlength{0.2cm}
\begin{figure}[H]
\begin{center}
\begin{tabular}{c}

 \begin{minipage}{0.6\hsize}
 \begin{center}
\begin{picture}(0,6)
 \multiput(-23,0.55)(4,0){4}{\line(1,0){3}}
 \put(-40,0){$($when $j=\infty$$)$}
 \put(-24,0){$e$}
  \put(-20,1.5){{\tiny $1$}}
 \put(-20,0){$\bullet$}
 \put(-16,0){$\circ$}
  \put(-12,1.5){{\tiny $2$}}
 \put(-12,0){$\bullet$}
 \put(-4,0){$\cdots$}
 
 \end{picture}
 \end{center}
 \end{minipage}

\end{tabular}
\end{center}
\end{figure}

\setlength\unitlength{0.2cm}
\begin{figure}[H]
\begin{center}
\begin{tabular}{c}

 \begin{minipage}{0.6\hsize}
 \begin{center}
\begin{picture}(0,6)
 \multiput(-27,0.55)(4,0){5}{\line(1,0){3}}
 \multiput(17,0.55)(4,0){5}{\line(1,0){3}}
 \multiput(-3,3.05)(4,0){4}{\line(1,0){3}}
 \multiput(-3,-1.95)(4,0){4}{\line(1,0){3}}
 \multiput(-22,0.55)(0.3,0){20}{\line(1,0){0.15}}
 \multiput(2,3.05)(0.3,0){20}{\line(1,0){0.15}}
 \multiput(2,-1.95)(0.3,0){20}{\line(1,0){0.15}}
 \multiput(26,0.55)(0.3,0){20}{\line(1,0){0.15}}
 \qbezier(16.2,0.6)(16.2,0.6)(12.9,2.8)
 \qbezier(16.2,0.5)(16.2,0.5)(12.9,-1.7)
 \qbezier(-7.1,0.6)(-7.1,0.6)(-3.8,2.8)
 \qbezier(-7.1,0.5)(-7.1,0.5)(-3.8,-1.7)
 \put(-40,0){$({\rm I\hspace{-.1em}I_{i,j,k}})$}
 \put(-28,0){$e$}
  \put(-24,1.5){{\tiny $j-1$}}
 \put(-24,0){$\circ$}
 \put(-20,0){}
  \put(-16,1.5){{\tiny $2$}}
 \put(-16,0){$\circ$}
  \put(-12,1.5){{\tiny $1$}}
 \put(-12,0){$\circ$}

  \put(-8,1.5){{\tiny $1$}}
 \put(-8,0){$\circ^{\prime}$}
  \put(-4,4){{\tiny $2$}}
 \put(-4,2.5){$\circ$}
  \put(0,4){{\tiny $3$}}
 \put(0,2.5){$\circ$}
 \put(4,2.5){}
  \put(8,4){{\tiny $i-1$}}
 \put(8,2.5){$\circ$}
  \put(12,4){{\tiny $i$}}
 \put(12,2.5){$\circ$}
  \put(16,1.5){{\tiny $i+1$}}
 \put(16,0){$\circ^{\prime}$}

  \put(-4,-1){{\tiny $2$}}
 \put(-4,-2.5){$\circ$}
  \put(0,-1){{\tiny $3$}}
 \put(0,-2.5){$\circ$}
 \put(4,-2.5){}
  \put(8,-1){{\tiny $i-1$}}
 \put(8,-2.5){$\circ$}
  \put(12,-1){{\tiny $i$}}
 \put(12,-2.5){$\circ$}

  \put(20,1.5){{\tiny $1$}}
 \put(20,0){$\circ$}
  \put(24,1.5){{\tiny $2$}}
 \put(24,0){$\circ$}
 \put(28,0){}
  \put(32,1.5){{\tiny $k-1$}}
 \put(32,0){$\circ$}
 \put(36,0){$e$}
 \end{picture}
 \end{center}
 \end{minipage}

\end{tabular}
\end{center}
\end{figure}

\setlength\unitlength{0.2cm}
\begin{figure}[H]
\begin{center}
\begin{tabular}{c}

 \begin{minipage}{0.6\hsize}
 \begin{center}
\begin{picture}(0,6)
 \multiput(1,0.55)(4,0){5}{\line(1,0){3}}
 \multiput(-19,3.05)(4,0){4}{\line(1,0){3}}
 \multiput(-19,-1.95)(4,0){4}{\line(1,0){3}}
 \multiput(-14,3.05)(0.3,0){20}{\line(1,0){0.15}}
 \multiput(-14,-1.95)(0.3,0){20}{\line(1,0){0.15}}
 \multiput(10,0.55)(0.3,0){20}{\line(1,0){0.15}}
 \qbezier(0.2,0.6)(0.2,0.6)(-3.1,2.8)
 \qbezier(0.2,0.5)(0.2,0.5)(-3.1,-1.7)
 \qbezier(-23.1,0.6)(-23.1,0.6)(-19.8,2.8)
 \qbezier(-23.1,0.5)(-23.1,0.5)(-19.8,-1.7)
 \put(-40,0){$({\rm I\hspace{-.1em}I\hspace{-.1em}I_{i,j}})$}

  \put(-24,1.5){{\tiny $1$}}
 \put(-24,0){$\circ$}
  \put(-20,4){{\tiny $2$}}
 \put(-20,2.5){$\circ$}
  \put(-16,4){{\tiny $3$}}
 \put(-16,2.5){$\circ$}
 \put(-12,2.5){}
  \put(-8,4){{\tiny $i-1$}}
 \put(-8,2.5){$\circ$}
  \put(-4,4){{\tiny $i$}}
 \put(-4,2.5){$\circ$}
  \put(0,1.5){{\tiny $i+1$}}
 \put(0,0){$\circ^{\prime}$}

  \put(-20,-1){{\tiny $2$}}
 \put(-20,-2.5){$\circ$}
  \put(-16,-1){{\tiny $3$}}
 \put(-16,-2.5){$\circ$}
 \put(-12,-2.5){}
  \put(-8,-1){{\tiny $i-1$}}
 \put(-8,-2.5){$\circ$}
  \put(-4,-1){{\tiny $i$}}
 \put(-4,-2.5){$\circ$}

  \put(4,1.5){{\tiny $1$}}
 \put(4,0){$\circ$}
  \put(8,1.5){{\tiny $2$}}
 \put(8,0){$\circ$}
 \put(12,0){}
  \put(16,1.5){{\tiny $j-1$}}
 \put(16,0){$\circ$}
 \put(20,0){$e$}
 \end{picture}
 \end{center}
 \end{minipage}

\end{tabular}
\end{center}
\end{figure}

\setlength\unitlength{0.2cm}
\begin{figure}[H]
\begin{center}
\begin{tabular}{c}

 \begin{minipage}{0.6\hsize}
 \begin{center}
\begin{picture}(0,6)
 \multiput(1,0.55)(4,0){5}{\line(1,0){3}}
 \multiput(-23,3.05)(4,0){5}{\line(1,0){3}}
 \multiput(-23,-1.95)(4,0){5}{\line(1,0){3}}
 \multiput(-14,3.05)(0.3,0){20}{\line(1,0){0.15}}
 \multiput(-14,-1.95)(0.3,0){20}{\line(1,0){0.15}}
 \multiput(10,0.55)(0.3,0){20}{\line(1,0){0.15}}
 \qbezier(0.2,0.6)(0.2,0.6)(-3.1,2.8)
 \qbezier(0.2,0.5)(0.2,0.5)(-3.1,-1.7)
 \qbezier(-23.8,3.1)(-23.8,3.1)(-27.1,4.3)
 \qbezier(-23.8,3)(-23.8,3)(-27.1,1.8)
 \qbezier(-23.8,-1.9)(-23.8,-1.9)(-27.1,-0.7)
 \qbezier(-23.8,-2)(-23.8,-2)(-27.1,-3.2)
 \put(-40,0){$({\rm I\hspace{-.1em}V_{i,j}})$}

 \put(-28,4){$\circ$}
 \put(-28,1){$\circ$}
  \put(-24,4){{\tiny $1$}}
 \put(-24,2.5){$\circ$}
  \put(-20,4){{\tiny $2$}}
 \put(-20,2.5){$\circ$}
  \put(-16,4){{\tiny $3$}}
 \put(-16,2.5){$\circ$}
 \put(-12,2.5){}
  \put(-8,4){{\tiny $i-1$}}
 \put(-8,2.5){$\circ$}
  \put(-4,4){{\tiny $i$}}
 \put(-4,2.5){$\circ$}
  \put(0,1.5){{\tiny $i+1$}}
 \put(0,0){$\circ^{\prime}$}

 \put(-28,-1){$\circ$}
 \put(-28,-4){$\circ$}
  \put(-24,-1){{\tiny $1$}}
 \put(-24,-2.5){$\circ$}
  \put(-20,-1){{\tiny $2$}}
 \put(-20,-2.5){$\circ$}
  \put(-16,-1){{\tiny $3$}}
 \put(-16,-2.5){$\circ$}
 \put(-12,-2.5){}
  \put(-8,-1){{\tiny $i-1$}}
 \put(-8,-2.5){$\circ$}
  \put(-4,-1){{\tiny $i$}}
 \put(-4,-2.5){$\circ$}

  \put(4,1.5){{\tiny $1$}}
 \put(4,0){$\circ$}
  \put(8,1.5){{\tiny $2$}}
 \put(8,0){$\circ$}
 \put(12,0){}
  \put(16,1.5){{\tiny $j-1$}}
 \put(16,0){$\circ$}
 \put(20,0){$e$}
 \end{picture}
 \end{center}
 \end{minipage}

\end{tabular}
\end{center}
\end{figure}

\setlength\unitlength{0.2cm}
\begin{figure}[H]
\begin{center}
\begin{tabular}{c}

 \begin{minipage}{0.6\hsize}
 \begin{center}
\begin{picture}(0,6)
 \multiput(-27,0.55)(4,0){5}{\line(1,0){3}}
 \multiput(-19,0.55)(0.3,0){20}{\line(1,0){0.15}}
 \qbezier(-27.8,0.6)(-27.8,0.6)(-31.1,2.8)
 \qbezier(-27.8,0.5)(-27.8,0.5)(-31.1,-1.7)
 \qbezier(-27.5,0.3)(-27.5,0.3)(-27.5,-3.1)
 \put(-40,0){$({\rm V_{j}})$}
 \put(-32,-2.5){$\circ$}
 \put(-32,2.5){$\circ$}
 \put(-28,-4){$e$}
 \put(-28,0){$\circ^{\prime}$}
  \put(-24,1.5){{\tiny $1$}}
 \put(-24,0){$\circ$}
  \put(-20,1.5){{\tiny $2$}}
 \put(-20,0){$\circ$}
 \put(-16,0){}
  \put(-12,1.5){{\tiny $j-1$}}
 \put(-12,0){$\circ$}
 \put(-8,0){$e$}

 \multiput(13,0.55)(4,0){5}{\line(1,0){3}}
 \multiput(21,0.55)(0.3,0){20}{\line(1,0){0.15}}
 \qbezier(12.2,0.6)(12.2,0.6)(8.9,2.8)
 \qbezier(12.2,0.5)(12.2,0.5)(8.9,-1.7)
 \put(0,0){$({\rm V\hspace{-.1em}I_{j}})$}
 \put(8,-2.5){$e$}
 \put(8,2.5){$e$}
 \put(12,0){$\circ$}
  \put(16,1.5){{\tiny $1$}}
 \put(16,0){$\circ$}
  \put(20,1.5){{\tiny $2$}}
 \put(20,0){$\circ$}
 \put(24,0){}
  \put(28,1.5){{\tiny $j-1$}}
 \put(28,0){$\circ$}
 \put(32,0){$e$}
 \end{picture}
 \end{center}
 \end{minipage}

\end{tabular}
\end{center}
\end{figure}

\setlength\unitlength{0.2cm}
\begin{figure}[H]
\begin{center}
\begin{tabular}{c}

 \begin{minipage}{0.6\hsize}
 \begin{center}
\begin{picture}(0,14)
 \multiput(-23.5,-8)(0,3){7}{\line(0,1){2.2}}
 \put(-40,0){$({\rm V\hspace{-.1em}I\hspace{-.1em}I_{0}})$}
 \put(-24,0){$e$}
 \put(-24,3){$\circ$}
 \put(-24,6){$\bullet$}
 \put(-28,6){$\circ$}
 \put(-27.2,6.5){\line(1,0){3.2}}
 \put(-24,9){$\circ$}
 \put(-24,12){$\bullet$}
 \put(-24,-3){$\circ$}
 \put(-24,-6){$\bullet$}
 \put(-24,-9){$\circ$}

\multiput(8.5,-8)(0,3){6}{\line(0,1){2.2}}
 \put(-8,0){or}
 \put(8,0){$e$}
 \put(8,3){$\circ$}
 \put(8,6){$\bullet$}
 \put(8,9){$\circ$}
 \put(8,-3){$\circ$}
 \put(8,-6){$\bullet$}
 \put(8,-9){$\circ$}
 \end{picture}
 \end{center}
 \end{minipage}

\end{tabular}
\end{center}
\end{figure}

\setlength\unitlength{0.2cm}
\begin{figure}[H]
\begin{center}
\begin{tabular}{c}

 \begin{minipage}{0.6\hsize}
 \begin{center}
\begin{picture}(0,14)
 \multiput(-31,0.55)(4,0){5}{\line(1,0){3}}
 \multiput(-31.5,-8)(0,3){6}{\line(0,1){2.2}}
 \multiput(-23,0.55)(0.3,0){20}{\line(1,0){0.15}}
 \put(-40,0){$({\rm V\hspace{-.1em}I\hspace{-.1em}I_{j}})$}
 \put(-32,0){$\bullet$}
 \put(-32,3){$\circ$}
 \put(-32,6){$\bullet$}
 \put(-32,9){$\circ$}
 \put(-32,-3){$\circ$}
 \put(-32,-6){$\bullet$}
 \put(-32,-9){$\circ$}
  \put(-28,1.5){{\tiny $1$}}
 \put(-28,0){$\circ^{\prime}$}
  \put(-24,1.5){{\tiny $2$}}
 \put(-24,0){$\circ$}
 \put(-20,0){}
  \put(-16,1.5){{\tiny $j-1$}}
 \put(-16,0){$\circ$}
 \put(-12,0){$e$}

 \multiput(13,0.55)(4,0){4}{\line(1,0){3}}
 \multiput(17,0.55)(0.3,0){20}{\line(1,0){0.15}}
 \qbezier(12.2,0.6)(12.2,0.6)(8.9,2.8)
 \qbezier(12.2,0.5)(12.2,0.5)(8.9,-1.7)
 \put(0,0){$({\rm V\hspace{-.1em}I\hspace{-.1em}I\hspace{-.1em}I_{j}})$}
 \put(8,-2.5){$\circ^{\prime}$}
 \put(8,2.5){$\circ^{\prime}$}
  \put(12,1.5){{\tiny $1$}}
 \put(12,0){$\circ$}
  \put(16,1.5){{\tiny $2$}}
 \put(16,0){$\circ$}
 \put(20,0){}
  \put(24,1.5){{\tiny $j-1$}}
 \put(24,0){$\circ$}
 \put(28,0){$e$}
 \end{picture}
 \end{center}
 \end{minipage}

\end{tabular}
\end{center}
\end{figure}

\setlength\unitlength{0.2cm}
\begin{figure}[H]
\begin{center}
\begin{tabular}{c}

 \begin{minipage}{0.6\hsize}
 \begin{center}
\begin{picture}(0,8)
 \multiput(-31,0.55)(4,0){2}{\line(1,0){3}}
 \put(-40,0){$({\rm I\hspace{-.1em}X})$}
 \put(-32,0){$e$}
 \put(-28,0){$e$}
 \put(-24,0){$e$}

 \multiput(-7,0.55)(4,0){3}{\line(1,0){3}}
 \put(-16,0){$({\rm X})$}
 \put(-8,0){$\bullet$}
 \put(-4,0){$\circ$}
 \put(0,0){$\circ^{\prime}$}
 \put(4,0){$\circ^{\prime}$}

 \multiput(21,0.55)(4,0){3}{\line(1,0){3}}
 \put(12,0){or}
 \put(20,0){$\circ$}
 \put(24,0){$\bullet$}
 \put(28,0){e}
 \put(32,0){$\circ^{\prime}$}
 \end{picture}
 \end{center}
 \end{minipage}

\end{tabular}
\end{center}
\end{figure}

\setlength\unitlength{0.2cm}
\begin{figure}[H]
\begin{center}
\begin{tabular}{c}

 \begin{minipage}{0.6\hsize}
 \begin{center}
\begin{picture}(0,8)
 \multiput(-31,0.55)(4,0){2}{\line(1,0){3}}
 \put(-40,0){$({\rm X\hspace{-.1em}I})$}
 \put(-32,0){$\bullet$}
 \put(-28,0){$e$}
 \put(-24,0){$\circ$}
 \end{picture}
 \end{center}
 \end{minipage}

\end{tabular}
\end{center}
\end{figure}
\noindent where the symbol $\bullet$ is a $(-2)$-curve contained in ${\rm Fix}(\widetilde{\sigma})$, the symbol $\circ$ is a $(-2)$-curve not contained in ${\rm Fix}(\widetilde{\sigma})$, the symbol $\circ^{\prime}$ is a $(-3)$-curve not contained in ${\rm Fix}(\widetilde{\sigma})$, and the symbol $e$ is an effective divisor. 
\end{thm}

Let $f\colon S\to B$ be a hyperelliptic fibered surface of genus $3$.
Recall the slope equality in Theorem~\ref{slopeeq} for $g=3$ and $n=2$:
$$
K_f=\frac{8}{3}\chi_f+{\rm Ind},
$$
where ${\rm Ind}=\sum_{p\in B}{\rm Ind}(F_p)$ is defined by ${\rm Ind}(F_p)=\alpha_2(F_p)+\varepsilon(F_p)$, which is called the Horikawa index of $F_p$.
Computing the Horikawa index for a fiber $F_p$ of each type,
we get:
\begin{thm} \label{g3hypthm}
The Horikawa index of a hyperelliptic fibration of genus $3$ is given by
\begin{eqnarray*}
{\rm Ind}&=&\sum_{i} \frac{2}{3}i\nu({\rm I}_{i,0,0})+\sum_{i,k\ge 1} \left(\frac{2}{3}i+\frac{5}{3}k-1\right)\nu({\rm I}_{i,0,k})+\sum_{i} \left(\frac{2}{3}i+\frac{5}{3}\right)\nu({\rm I}_{i,0,\infty})\\
&+&\sum_{i,j\ge 1,k\ge 1} \left(\frac{2}{3}i+\frac{5}{3}(j+k)-2\right)\nu({\rm I}_{i,j,k})+\sum_{i,j\ge 1} \left(\frac{2}{3}i+\frac{5}{3}j+\frac{2}{3}\right)\nu({\rm I}_{i,j,\infty})\\
&+&\sum_{i} \left(\frac{2}{3}i+\frac{10}{3}\right)\nu({\rm I}_{i,\infty,\infty})+\sum_{i,j,k} \left(\frac{2}{3}i+\frac{5}{3}(j+k)\right)\nu({\rm I\hspace{-.1em}I}_{i,j,k})\\
&+&\sum_{i,j} \left(\frac{2}{3}i+\frac{5}{3}j+\frac{8}{3}\right)\nu({\rm I\hspace{-.1em}I\hspace{-.1em}I}_{i,j})+\sum_{i,j} \left(\frac{2}{3}i+\frac{5}{3}j+\frac{4}{3}\right)\nu({\rm I\hspace{-.1em}V}_{i,j})\\
&+&\sum_{j} \left(\frac{5}{3}j+\frac{4}{3}\right)\nu({\rm V}_{j})+\sum_{j} \left(\frac{5}{3}j+\frac{5}{3}\right)\nu({\rm V\hspace{-.1em}I}_{j})\\
&+&\frac{4}{3}\nu({\rm V\hspace{-.1em}I\hspace{-.1em}I}_{0})+\sum_{j\ge 1} \left(\frac{5}{3}j+\frac{1}{3}\right)\nu({\rm V\hspace{-.1em}I\hspace{-.1em}I}_{j})+\sum_{j\ge 1} \left(\frac{5}{3}j+\frac{2}{3}\right)\nu({\rm V\hspace{-.1em}I\hspace{-.1em}I\hspace{-.1em}I}_{j})\\
&+&\frac{4}{3}\nu({\rm I\hspace{-.1em}X})+\frac{7}{3}\nu({\rm X})+\frac{10}{3}\nu({\rm X\hspace{-.1em}I})\\
\end{eqnarray*}
where $\nu(*)$ denotes the number of fibers of type $(*)$.
\end{thm}

\section{Multiple fibers}

For any fiber $F$ of a fibered surface, the intersection form on the support of $F$ is negative semi-definite (with $1$-dimensional kernel) by Zariski's lemma.
Hence, we have the smallest non-zero effective divisor $D$ with support in $\mathrm{Supp}(F)$ such that $DC=0$ holds for any irreducible component 
$C$ of $F$.
We call $D$ the numerical cycle.
There exists a positive integer $k$ such that $F=kD$.
When $k>1$, we call $F$ a multiple fiber of multiplicity $k$.

The following lemma is easy to prove.
\begin{lem} \label{easylem}
Let $n\ge 4$ be a positive integer and $a$, $b$ integers such that ${\rm gcd}(a,b,n)=1$. Then, it follows that $a+2b\notin n\mathbb{Z}$ or $2a+b\notin n\mathbb{Z}$.
\end{lem}

From Lemma \ref{easylem}, we obtain the following assertion on multiple fibers of $f$.

\begin{prop} \label{multfiberprop}
Let $f\colon S\rightarrow B$ be a primitive cyclic covering fibration of type $(g,0,n)$. 
If $n\le 3$, then any multiple fiber of $f$ is an $n$-fold fiber. 
If $n\ge 4$, then $f$ has no multiple fibers.
\end{prop}

\begin{proof}
It is sufficient to show the claim with respect to $\widetilde{f}$. Suppose that $\widetilde{F}=\widetilde{F}_p$ is a multiple fiber of $\widetilde{f}$. 
We write $\widetilde{F}=kD$ where $k\ge 2$ and $D$ is the numerical cycle. 
Let $\widetilde{\Gamma}=\widetilde{\Gamma}_p$ be the fiber of $\widetilde{\varphi}$ at $p=\widetilde{f}(\widetilde{F})$. 
We write $\widetilde{\Gamma}=L_1+\cdots+L_s$ and $L_i=k_i\Gamma_i$ where $\Gamma_i\simeq\mathbb{P}^1$ and $\Gamma_i\neq\Gamma_j$ if $i\neq j$. 
At least two $k_i$'s are $1$ since $\widetilde{\Gamma}$ is the total transform of a fiber $\Gamma\simeq\mathbb{P}^1$ of $\varphi$. 
We may assume $k_1=k_2=1$. 
The numerical cycle $D$ is decomposed to $D_1+\cdots+D_s$ such that $\widetilde{\theta}(D_i)=\Gamma_i$. 
Since $\widetilde{\theta}^\ast \widetilde{\Gamma}=\widetilde{F}$, it follows $\widetilde{\theta}^\ast L_i=kD_i$. 
In particular, we have $\widetilde{\theta}^\ast \Gamma_1=kD_1$. Thus, $\Gamma_1$ is contained in $\widetilde{R}$ since $k\ge 2$. 
Hence it follows $k=n$. Suppose that $\Gamma_i$ is not contained in $\widetilde{R}$. 
Since $\widetilde{\theta}^\ast \Gamma_i$ is reduced and $nD_i=k_i\widetilde{\theta}^\ast \Gamma_i$, then it follows $k_i\in n\mathbb{Z}$. 
Hence, $\widetilde{\Gamma}$ satisfies the following $(\#)$:

\smallskip

\noindent
$(\#)$ We take an irreducible component $\Gamma_j$ such that $k_j\notin n\mathbb{Z}$ arbitrarily. If another component $\Gamma_i$ intersects $\Gamma_j$, then it follows that $k_i\in n\mathbb{Z}$,

\smallskip

\noindent
since $\widetilde{R}$ consists of smooth disjoint curves. However, we can show that there exist no reducible fibers of $\widetilde{\varphi}$ which satisfy $(\#)$ if $n\ge 4$. This can be shown as follows. Let $\Gamma$ be the fiber of $\varphi$ at $p$. If $\widetilde{\Gamma}$ is reducible, $\Gamma$ is blown up by $\widetilde{\psi}$ at least once. we may assume that $\psi_1$,\ldots,$\psi_{s-1}$ are blow-ups at a point on the fiber at $p$. Put $\widetilde{\Gamma}_i=(\psi_i\circ\cdots\circ\psi_1)^\ast \Gamma=L_1+\cdots+L_{i+1}$ where we identify $L_k$ with the proper transform of $L_k$. Then, we have $\widetilde{\Gamma}_1=L_1+L_2=\Gamma_1+\Gamma_2$. Since  $\widetilde{\Gamma}$ satisfies $(\#)$, the intersection point of $\Gamma_1$ and $\Gamma_2$ is blown up. Thus, we may assume $\psi_2$ is the blow-up at the point. Then, we have $\widetilde{\Gamma}_2=L_1+L_2+L_3=\Gamma_1+\Gamma_2+2\Gamma_3$. This operation repeats unless the total transform of $\Gamma$ satisfies $(\#)$. Blowing u
 p at the intersection point of $L_i=a\Gamma_i$ and $L_j=b\Gamma_j$, the multiplicity of the new exceptional curve is $a+b$. From this observation and Lemma \ref{easylem}, the operation would repeat endlessly, which leads us to a contradiction. Hence there exist no  reducible fibers of $\widetilde{\varphi}$ satisfying $(\#)$. If $\widetilde{\Gamma}\simeq\mathbb{P}^1$, then $\widetilde{F}=nD$ and $D\simeq\mathbb{P}^1$. It contradicts $g\ge 2$ of $\widetilde{f}$.
\end{proof}

\begin{rem}
In the case where $n=2$, i.e., $f$ is a hyperelliptic fibration of genus $g$, it is known that there exists a fibration $f$ with a double fiber for any odd $g$. In the case where $n=3$, we have shown in \S3 that there exists a fibration $f$ with a triple fiber.
\end{rem}

\end{document}